\documentclass[12pt]{article}
\usepackage{latexsym}
\usepackage{amsthm}
\usepackage{amssymb}
\usepackage{amsmath}
\usepackage{rotating}
\usepackage{multirow}
\usepackage{mathrsfs}
\usepackage{wasysym}
\usepackage{enumerate}
\usepackage{esint}
\usepackage{hyperref}
\usepackage{rawfonts}
\input{prepictex}
\input{pictex}
\input{postpictex}
\DeclareMathOperator{\End}{End}
\def\dir{\not \negthickspace D}
\DeclareMathOperator{\Diff}{\zap{Diff}}
\DeclareGraphicsExtensions{.pdf, .jpg}
\usepackage[OT2,OT1]{fontenc}
\def\cyr{%
\renewcommand\rmdefault{wncyr}%
\renewcommand\sfdefault{wncyss}%
\renewcommand\encodingdefault{OT2}%
\normalfont
\selectfont}
\DeclareMathAlphabet{\zap}{OT1}{pzc}{m}{it}
\DeclareTextFontCommand{\textcyr}{\cyr}
\def\CC{\mathbb C}
\def\HH{\mathbb H}

\DeclareMathOperator{\Hom}{Hom}

\newtheorem{thm}{Theorem}[section]
\newtheorem{prop}[thm]{Proposition}
\newtheorem{defn}[thm]{Definition}
\newtheorem{cor}[thm]{Corollary}

\newenvironment{xpl}{\mbox{ }\\ {\bf  Example.}\mbox{ }}{
\hfill $\diamondsuit$\mbox{}\bigskip}
\newenvironment{rmk}{\mbox{ }\\{\bf  Remark.}\mbox{ }}{
\hfill $\diamondsuit$\mbox{}\bigskip}
\def\BW{\mathcal{BW}}
\def\ZZ{{\mathbb Z}}
\def\NN{\mathbb{N}}
\def\RR{{\mathbb R}}
\def\CP{\mathbb{CP}}
\def\cpbar{\overline{\CP}_2}

\DeclareMathOperator{\Ind}{Ind}

\DeclareMathOperator{\vol}{Vol}

\DeclareMathOperator{\kod}{Kod}
\DeclareMathOperator{\hull}{Hull}
\DeclareMathOperator{\Tor}{Torsion}
\begin{document}

\title{An Overview of  Einstein $4$-Manifolds}

\author{Claude LeBrun}

\date{} 

\maketitle

\section{Einstein Metrics}

A Riemannian metric $g$ on a smooth $n$-manifold $M$ is said to be {\em Einstein} if it has constant Ricci curvature. Since the 
{\em Ricci curvature}  is by definition   the function 
\begin{eqnarray*}
UTM &\longrightarrow& \RR\\
v&\longmapsto& r(v,v)  
\end{eqnarray*}
on the unit tangent bundle   $UTM:= \{ v\in TM~|~ g(v,v) =1\}$, where  $r$  denotes   the Ricci tensor  of $g$, 
 this is equivalent to saying that $g$ solves the  {\em Einstein equation}  \cite{bes}
\begin{equation}
\label{eineq}
r= \lambda g
\end{equation}
for  a  real constant $\lambda$  that  is then called  the {\em Einstein constant} of $g$.
Of course, this equation first arose 
in the  context of general relativity,
where  \eqref{eineq} is usually called  the ``Einstein vacuum equation
   with cosmological constant,'' although in that context  the metric
    $g$ would be taken to be     indefinite\footnote{While we Riemannian geometers habitually  denote  a Riemannian  metric by $g$, very few of us 
   seem to be aware of the   historical origin of this notation, which   arose from   Einstein's conceptualization  of the space-time metric $g$ as being  the {\em gravitational potential!}}     
   rather than Riemannian.  In fact, Einstein did  briefly consider the possibility that the constant $\lambda$ in \eqref{eineq} should  
   be taken to be non-zero, although he  later came to reject  this proposal, and was 
   eventually   quoted \cite[p. 44]{gamow}  as saying that he considered   the idea to be  ``the biggest  blunder
   of his life.''  So,  even as  we  honor Einstein by calling  \eqref{eineq} the ``Einstein equation'' 
   and  by calling $\lambda$ the ``Einstein constant,'' we might  do well to  also remember that 
    these are honors that 
 Einstein himself might have been  far from eager to accept!

   In this article, the term {\em Einstein manifold} will always  mean  a pair $(M,g)$  consisting of a  smooth compact connected $n$-manifold $M$ (always without boundary)   and 
    an  Einstein metric $g$ on $M$. Before long, we will actually  narrow our scope and focus entirely on  the $n=4$ case. But before doing so,    we will first   explain  why dimension four 
    is special for the problem, both in the sense that it represents a transition between low-dimensional ($n\leq 3$) and high dimensional ($n\geq 5$) behaviors, and in the sense
    that it is the setting for a host of phenomena that  have  no  analogs at all in other dimensions.

   Of course,  one of the most celebrated   problems in present-day Riemannian geometry is to determine precisely which smooth closed $n$-manifolds admit Einstein metrics. 
   But this {\em existence problem} is naturally  accompanied by a companion {\em uniqueness problem}  that only becomes intelligible 
    after some careful discussion.   The subtlety   is that  whenever $M$ admits an Einstein metric $g$, then it always admits an infinite-dimensional family  of others. This is  an unavoidable 
    consequence of the fact that the  Einstein condition    \eqref{eineq} is diffeomorphism-invariant; thus,  if 
     $g$ is an Einstein metric, and if $\Psi : M \to M$ is any self-diffeomorphism of $M$, 
   then $\Psi^*g$ is also an Einstein metric. At the same time, however,   $(M,g)$ and $(M, \Psi^*g)$ are actually {\em isometric},
   and should therefore be understood as being  geometrically indistinguishable.

   Yet another  easy  way of generating new Einstein metrics from a given one is 
   to {\em rescale}. This works because multiplying a metric $g$ by a positive constant $c$ does not change the Levi-Civita connection $\nabla$, 
   which after all is characterized by 
   $$\nabla g = 0 \quad \text{and} \quad \Tor (\nabla )=0, $$
  so that constant rescaling   therefore  changes neither  the Riemann curvature tensor $\mathcal{R}$,  as characterized in our conventions by 
   $${\mathcal{R}^a}_{bcd} v^b = 2\nabla_{[c}\nabla_{d]} v^a ,$$
   nor the Ricci tensor, as defined by 
   $$r_{ab} = {\mathcal{R}^c}_{acb}.$$
   As a consequence,  if $g$ satisfies $r= \lambda g$, then $\widehat{g} := c g$ satisfies $\widehat{r} = \widehat{\lambda} \widehat{g}$, where $\widehat{r}=r$ and $\widehat{\lambda} = c^{-1}\lambda$. 
   Of course, for $c\neq 1$, the resulting  Einstein manifolds $(M,cg)$ are no longer isometric to the  Einstein manifold $(M,g)$ we started with;  instead, they are said to be  said to be {\em homothetic} to it. 
   On the other hand, the rescaling operation $g \rightsquigarrow cg$ affects the total volume $V = \int_M d\mu_g$ of $(M,g)$ by $V\rightsquigarrow c^{n/2}V$, and some authors \cite{bes}  therefore 
    choose to avoid this
   issue by  only considering Einstein manifolds of total volume $1$. 
   
 In light of this discussion, we   now define the  {\em moduli space} of Einstein metrics on any  given smooth compact $n$-manifold $M$ to be 
 $$\mathscr{E}(M) := \{ \text{Einstein metrics } g \text{ on } M\} / (\Diff (M) \times \RR^+ ),$$
   where the diffeomorphism group 
   $$\Diff (M) := \{ \text{diffeomorphisms } \Psi : M \to M \}$$
   acts on metrics by pull-backs, while  the multiplicative group $\RR^+$ of positive real numbers acts on metrics by rescaling. 
While  this in particular endows 
 $\mathscr{E}(M)$ with   the quotient topology 
induced by  the $C^{\infty}$-topology on the space of Einstein metric tensors,  it is 
it is worth emphasizing  that, by an argument \cite{andmod,bes} involving elliptic regularity \cite{det-kaz},  this coincides  with the metric topology
 induced by the Gromov-Hausdorff distance \cite{gromhaus} 
between unit-volume Einstein metrics. The simplest version of the   above-mentioned 
uniqueness problem  is simply  to  determine whether or not $\mathscr{E}(M)$ consists of a single point; but failing this, the challenge 
becomes to   understand the    topological properties of $\mathscr{E}(M)$, starting with     whether or not  it is connected. 

Finally, notice  that we can naturally decompose the Ricci tensor $r$  of any Riemannian manifold $(M^n, g)$ as 
$$r =\mathring{r} + \frac{s}{n}g$$
where 
$$s= r^a_a= {\mathcal{R}^{ab}}_{ab}$$
denotes  the {\em scalar curvature}, and where 
$\mathring{r} = r - \frac{s}{n}g$
is  called the {\em trace-free Ricci tensor}. In these terms, the contracted Bianchi identity 
$$\nabla\cdot r = \frac{1}{2}d s$$
implies  that 
$$\nabla \cdot \mathring{r} = (\frac{1}{2}-\frac{1}{n} ) ds.$$
Thus,  in any dimension $n \neq 2$, the Einstein condition \eqref{eineq} is equivalent to 
\begin{equation}
\label{better}
\mathring{r}=0,
\end{equation}
since \eqref{better}  implies that $s$ is constant, thereby allowing us to 
 set   $\lambda = \frac{s}{n}$. 

   \section{Why  Focus on Dimension Four?} \label{optic} 
   
 We will now recount   some  of the key ways   in which the Einstein manifolds 
  seem to behave differently in dimension four than   in other dimensions.

In dimensions $2$ and $3$, the Ricci tensor algebraically determines
the entire Riemann curvature tensor. As  a consequence, 
 a  Riemannian  manifold $(M,g)$ of dimension 2 or 3  is  Einstein if and only if it has constant 
sectional curvature.  One 
  striking consequence  is that,  when  non-empty, 
    the moduli space of Einstein metrics on any closed $2$- or   $3$-manifold is always connected;
 indeed,  in the $3$-dimensional case, it  typically   consists  of a single point \cite{mostow}. 
But while   every $2$-manifold does actually  admit constant-curvature metrics, 
the situation in dimension $3$ is considerably more subtle. Indeed, since the universal cover
of a constant-curvature $3$-manifold must  be diffeomorphic to either 
$S^3$ or $\RR^3$,  it follows that any $3$-manifold with $\pi_2\neq 0$ cannot admit Einstein metrics. 
However, Perelman \cite{lott} was able to show that 
Hamilton's Ricci flow naturally breaks an arbitrary compact $3$-manifold up 
into Einstein and collapsed pieces, in a way that leads to a proof \cite{bbbpm,lott} of Thurston's geometrization conjectures \cite{thurston-orb,thur3}. 
For example, when an oriented  $3$-manifold has $\pi_2\neq 0$, 
this occurs because the manifold  contains  $S^2 \times \RR$  ``necks''    that the  Ricci flow then  
pinches off in order to  express the manifold as a  {\em connected sum} of  {\sf prime} $3$-manifolds. 
These geometrization results,
along  with recent proofs of several  other remarkable conjectures made by  Thurston, provide us with  a
vivid    picture  of  the  closed $3$-manifold  that admit constant-curvature  metrics. 
For example, a closed $3$-manifold admits an Einstein metric iff  \cite{agol}  it is finitely covered by $S^3$, $T^3=S^1 \times S^1 \times S^1$, 
or a closed-surface bundle over $S^1$ for which the monodromy map of the  higher-genus fiber  is pseudo-Anosov.

By contrast, Einstein metrics in  high dimensions have an entirely different feel to them. Here, even on the most 
familiar manifolds $M^n$, $n\geq 5$, the Einstein moduli space $\mathscr{E}(M)$ tends to be highly disconnected. 
For example, it was first proved by  B\"ohm \cite{bohm}  that $\mathscr{E}(S^5)$ has {\em infinitely many} connected components; 
and a very different  second proof of this fact by Collins and Szekelyhidi \cite{collszek} was recently extended  by Liu, 
Sano, and Tasin \cite{liusanotasin} to prove the analogous assertion  on all the odd-dimensional spheres $S^{2m+1}$, $m \geq 2$. 
 Related constructions  of Boyer, Galicki, and Koll\'ar  \cite{bg,bgk} moreover yield  Einstein metrics on ``typical'' simply-connected spin $5$-manifolds,
and in many of these cases the Einstein moduli space can also be shown to have infinitely many connected components. 
On the other hand, constructions   of Catanese and LeBrun \cite{cat} and of   R\u{a}sdeaconu and \c{S}uvaina \cite{raresioana}
produce various    high-dimensional smooth  manifolds that  admit both both $\lambda > 0$ and $\lambda < 0$ Einstein metrics;
thus,  differential topology cannot even reliably  predict the sign of the Einstein constant in this high-dimensional  realm.

Dimension four, by contrast, combines many of the best features of both  realms. In dimension $n=4$, 
the Einstein equation is  in no sense   locally trivial: as in higher dimensions, most solutions of \eqref{eineq} no longer
have constant sectional curvature. Nonetheless, in spite of this local flexibility, there are various settings
where Einstein $4$-manifolds display a global rigidity that is reminiscent of the situation we have described  in lower dimensions. 
The first result of this type was discovered by Berger  \cite{bergb}:

\begin{thm}[Berger] \label{toro} 
Let  $T^4= S^1 \times S^1 \times S^1\times S^1$ denote  the $4$-torus, just considered   as  a smooth compact $4$-manifold. 
Then every  Einstein metric 
$g$ on $T^4$ is flat. Consequently  every such $(T^4,g)$ is  isometric to   Euclidean $4$-space
$\RR^4$ divided by a lattice $\Lambda \cong \ZZ^{\oplus 4}$ of translations, and   the 
moduli space $\mathscr{E}(T^4)$ of Einstein metrics on $T^4$ is therefore   connected.
\end{thm}

 The reader should incidentally avoid confusing   Theorem \ref{toro} with an  older  result of  Bochner \cite{bochner} which holds in all dimensions, but
only classifies    Einstein metrics on tori with $\lambda \geq 0$. By contrast, Berger's result  also excludes the existence 
of  Einstein metrics  on the $4$-torus 
with $\lambda < 0$. To date, no such result has ever been  obtained in any dimension $n\geq 5$.  

The next such result is due to Hitchin \cite{hit}, and instead concerns a specific  remarkable  smooth compact {\em simply-connected} $4$-manifold. 

\begin{thm}[Hitchin] \label{hitme} 
Let $K3$ denote  the smooth oriented simply-connected  four-manifold  underlying  the  compact  complex surface 
$$\{ [t,u,v,w]\in \CP_3~|~ t^4 + u^4 + v^4 + w^4 =0\}.$$
Then any Einstein metric $g$ on $K3$ is Ricci-flat  and hyper-K\"ahler, in the sense that it is K\"ahler with respect to 
an $S^2$-family of orientation-compatible   integrable complex structures 
 on $K3$.  As a consequence, the moduli space $\mathscr{E} (K3)$ is therefore connected. 
\end{thm}

Rather interestingly,  Hitchin actually proved the  first  part  of Theorem \ref{hitme}   several years before the existence of such metrics 
on $K3$ was established by  Yau's proof   \cite{yauma}   of the Calabi conjecture.  Deducing the fact that  the moduli space
$\mathscr{E}(K3)$  is  therefore connected then delicately depends on  the work  of several  others, including Kodaira \cite{kodrr},
 Siu \cite{siu},  and Kobayashi-Todorov \cite{kobtod}.


The next result of this type, due to Besson, Courtois, and Gallot \cite{bcg}, concerns real hyperbolic $4$-manifolds. 

\begin{thm}[Besson-Courtois-Gallot]
Let $M^4$ be a smooth compact isometric quotient of 
hyperbolic 4-space ${\cal H}^4=SO_+(4,1)/SO(4)$,
 and let $g_0$ be its tautological  metric of constant sectional curvature. 
 Then the moduli space $\mathscr{E}(M)$ of Einstein metrics
 on $M$ consists of a single point, represented by $g_0$. 
That is, 
every Einstein metric $g$  on $M$   is of the form 
$g=\Psi^*(cg_0)$ for  some self-diffeomorphism  $\Psi$ of $M$ and some positive constant 
   $c$.  
  \label{rehyp}
\end{thm}

Finally, we come to   an analogous result  \cite{lmo} that pertains to  another interesting class of locally-symmetric $4$-manifolds. 

\begin{thm}[LeBrun]
Let $M^4$ be a smooth compact isometric quotient of the complex-hyperbolic plane  
$\mathbb{C}{\cal H}_2=SU(2,1)/U(2)$, 
 and let $g_0$ be its standard complex-hyperbolic metric.
 Then the moduli space $\mathscr{E}(M)$ of Einstein metrics
 on $M$ consists of a single point, represented by $g_0$.  That is, 
every Einstein metric $g$  on $M$   is of the form 
$g=\Psi^*(cg_0)$ for  some self-diffeomorphism  $\Psi$ of $M$ and some positive constant 
   $c$.  
\label{cohyp}
\end{thm}

Note that while the existence of infinite hierarchies of compact real- and complex-hyperbolic $4$-manifolds 
is often  proved by citing  a general result  of Borel \cite{borel} on   compact quotients of symmetric spaces, this approach  does not  in any sense single 
out dimension $4$.  Nonetheless, for reasons that will be explained later in this article, 
both Theorems \ref{rehyp} and \ref{cohyp} are specific to dimension $4$,
 and there does not currently seem to be any reason to expect  that anything   analogous will hold in any higher dimension. 

The upshot is that dimension four seems  like a {\em Goldilocks zone}\footnote{Here the term {\em Goldilocks zone} has been borrowed from exoplanetary astronomy, where it
refers to the {\sf habitable zone} within the exoplanetary system surrounding   a star; 
cf. \url{https://science.nasa.gov/resource/goldilocks-zone/}. Both   usages in turn refer to the traditional English children's story, {\em Goldilocks and the Three Bears}, 
in which, for example, Goldilocks first rejects one bed as being ``too hard,'' and then another as being ``too soft,''  before finally  selecting a third bed  that seems ``just right!''}  for the Einstein equation \eqref{eineq}.
In low dimensions, Einstein metrics are ``too rigid.''  In high dimensions, they are ``too flexible.'' 
But in dimension four, the fit between Einstein metrics and  smooth topology seems ``just right!''

\section{Self-Duality on Four-Manifolds}
\label{non-simple}

The distinctive  nature of $4$-dimensional geometry is 
 largely due  to a single Lie-theoretic fluke: while the rotation group  $\mathbf{SO}(n)$ is  a {simple} Lie group
 for every other $n\geq 3$, this fails in dimension four. Indeed, the universal cover 
 $\mathbf{Spin}(4)$ of $\mathbf{SO}(4)$ can be identified with the Cartesian product 
 \begin{equation}
\label{spin4}
\mathbf{Spin}(4) = \mathbf{Sp}(1) \times  \mathbf{Sp}(1)
\end{equation}
  of two copies of the unit quaternions  $\mathbf{Sp}(1) := S^3 \subset \HH^\times$, which respectively act 
  on the quaternions $\HH = \RR^4$ by left- and right-multiplication. Thus 
   $$\mathbf{SO}(4)/\{ \pm I\}   =  [\mathbf{Sp}(1) \times  \mathbf{Sp}(1)]\{ (\pm 1, \pm 1)\}  \cong \mathbf{SO}(3) \times  \mathbf{SO}(3), $$
and   the adjoint action of $\mathbf{SO}(4)$ on $\mathfrak{so}(4)$ is therefore    reducible:
 \begin{equation}
\label{al-deco}
\mathfrak{so}(4)\cong  \mathfrak{so}(3)\oplus \mathfrak{so}(3).
\end{equation}
 This has an immediate and  powerful impact on the geometry of $2$-forms, 
because, for any $n$,    ${\mathfrak s \mathfrak o}(n)$ and  $\Lambda^2({\mathbb R}^n)$ 
are   isomorphic as $\mathbf{SO}(n)$-modules. Thus, the decomposition \eqref{al-deco}  implies  that 
 the rank-6 bundle of 2-forms on an oriented Riemannian 4-manifold $(M,g)$ 
    invariantly  decomposes as  the Whitney sum of 
 two rank-3 bundles
\begin{equation} 
\Lambda^2 = \Lambda^+ \oplus \Lambda^- ,
\label{f-deco} 
\end{equation}
a phenomenon completely alien to   other   dimensions. 
One can then show that  the summands $\Lambda^\pm$ are  just  the $(\pm 1)$-eigenspaces of the Hodge star
operator 
\begin{equation}
\label{starring}
\star: \Lambda^2 \to \Lambda^2,
\end{equation}
and, for this reason,  $\Lambda^+$ is  called  the bundle of {\sf self-dual} $2$-forms, 
while   $\Lambda^-$ is called the bundle of {\sf anti-self-dual} $2$-forms.
Note that  the distinction between these two sub-bundles depends on a choice
of orientation; reversing the orientation of $M$ exactly  interchanges $\Lambda^+$ and $\Lambda^-$.

Using the  decomposition \eqref{f-deco} of
the $2$-forms,  we next obtain an  invariant decomposition of  the Riemann tensor $\mathcal{R}$
into simpler pieces. Here  it is convenient to  {\em raise} an index of ${\mathcal{R}^{a}}_{bcd}$
in order to identify the Riemann curvature with 
the so-called {\em curvature operator} 
\begin{eqnarray*}
 \Lambda^2 &\stackrel{\mathcal{R}}{\longrightarrow}& \Lambda^2 \\
\varphi_{ab}&\longmapsto& {\textstyle \frac{1}{2}} \varphi_{cd} {\mathcal{R}^{cd}}_{ab}.
\end{eqnarray*}
Since $\Lambda^2 = \Lambda^+ \oplus \Lambda^-$, we thus obtain  a decomposition 
\begin{equation}
\label{deco-riem}
{\mathcal R}=
\left(
\mbox{
\begin{tabular}{c|c}
&\\
$W_++\frac{s}{12}I$&$\mathring{r}$\\ &\\
\cline{1-2}&\\
$\mathring{r}$ & $W_-+\frac{s}{12}I$\\&\\
\end{tabular}
} \right) ,
\end{equation}
where the {\em self-dual Weyl curvature} $W_+ : \Lambda^+ \to \Lambda^+$ and {\em anti-self-dual Weyl curvature} $W_- : \Lambda^- \to \Lambda^-$ are the trace-free pieces of the corresponding blocks;
the corresponding tensors ${(W_\pm)^a}_{bcd}$ are conformally invariant, and their sum $W= W_++W_-$ is exactly Hermann Weyl's conformal  curvature tensor of $(M,g)$. 
Meanwhile, 
$$s = r^a_a = {\mathcal {R}^{ab}}_{ab}$$
is simply the scalar curvature. Finally, 
the two remaining blocks, both  indicated by  $\mathring{r}$ 
in \eqref{deco-riem}, are  adjoints of one another, and faithfully  encode the trace-free part $\mathring{r}=r-\frac{s}{4}g$ of the Ricci tensor, which acts  on $2$-forms  by 
$$\varphi_{ab} \mapsto  \mathring{r}^c_{[b} \varphi_{a]c} .  $$
One can then show that this action then maps  $\Lambda^+$ to $\Lambda^-$, and vice versa. 

Much of the fundamental importance of  \eqref{f-deco} derives from the fact that
 the curvature of any connection is a bundle-valued $2$-form. Thus,  if $(E,\triangledown )$ is a vector-bundle-with-connection 
over a smooth oriented Riemannian $4$-manifold $(M,g)$, its curvature $F_\triangledown$ is a section of $\Lambda^2 \otimes \End (E)$, and so we may 
uniquely express it as 
$F_\triangledown = F_\triangledown^++  F_\triangledown^-$, where 
$$F_\triangledown^\pm \in \Gamma (\Lambda^\pm \otimes \End (E)).$$

When $F_\triangledown^-=0$, so that $F_\triangledown = F_\triangledown^+$, one says that $(E,\triangledown )$ is a {\em self-dual} Yang-Mills instanton; if instead, 
$F_\triangledown^+=0$, so that $F_\triangledown = F_\triangledown^-$, one says that $(E,\triangledown )$ is an {\em anti-self-dual} Yang-Mills instanton. For example, 
the two left-hand blocks of \eqref{deco-riem} represent the curvature of $\Lambda^+\to M$ with respect to the Levi-Civita connection $\nabla$, while the two right-hand blocks
of \eqref{deco-riem} similarly represent the curvature of $\Lambda^-\to M$. Given  our discussion of \eqref{better}, it 
thus follows that $(\Lambda^+, \nabla )$  is a {\em self-dual}  instanton iff $(M,g)$
is Einstein, and that this is in turn happens iff $(\Lambda^-, \nabla )$ is an {\em anti-self-dual} Yang-Mills instanton. 

\section{The Intersection Form}
\label{intersection}

Several    of the most fundamental  topological invariants of a smooth compact $4$-manifold are intimately related to 
  its {\em  intersection form}. The simplest version of this object is just the symmetric pairing  
\begin{eqnarray}
Q  :
H^{2}_{dR}(M, {\mathbb R})\times H^{2}_{dR}(M, {\mathbb R})	
 & \longrightarrow & ~~~~ {\mathbb R}  \label{int-form} \\
	( ~ [\varphi ] \quad , \quad [\psi ] ~) \qquad & 
	\longmapsto  & \int_{M}\varphi \wedge \psi  \nonumber
\end{eqnarray}
on deRahm cohomology. Poincar\'e duality ensures that $Q$  is a non-degenerate symmetric pairing, 
so $Q$ can therefore be represented by the  matrix 
\begin{equation} \label{int-matrix} \left[ 
	  \begin{array}{rl}  \underbrace{
	 \begin{array}{ccc}
	   		 1 &  &	  \\
	   		  &	\ddots &   \\
	   		  &	 & 1
	   	  \end{array}}_{b_{+}(M)}
	   	   & 
	   		   \\ 
	  
	 {\scriptstyle b_{-}(M)}  \!
	    \left\{\begin{array}{r}
	    \\
	    \\
	    \\
	    \\
	    \end{array}
	    \right. \! \! \! \! \! \! \! \! \! \! \! \! \! \! \! 
	    &\begin{array}{ccc}
	   		 -1	&  &   \\
	   		  &	\ddots &   \\
	   		  &	 & -1
	   	  \end{array}
	    \end{array} 
	    \right] 
\end{equation}
  by choosing a suitable basis for $H^2_{dR}(M)$. 
Here  $b_+(M)$ (respectively,  $b_-(M)$) precisely equals  the maximum possible dimension of a  linear  subspace on which the restriction of $Q$ is positive-definite (respectively, negative-definite). 
By identifying deRham cohomology $H^2_{dR}(M, \RR)$ with 
the singular cohomology $H^2(M, \RR)$, and reinterpreting $Q$ in terms of cup products on cohomology, followed by pairing with the fundamental $4$-cycle of $M$, 
it is easy to see that the numbers $b_\pm (M)$ are in fact oriented homotopy invariants of $M$, and so  in particular
are  oriented homeomorphism  invariants of  the  underlying   topological manifold of   $M$.  The difference 
$$\tau (M) = b_+(M) - b_- (M)$$
between these two invariants is called the {\em signature} of $M$, 
while their sum 
$$b_2 (M) =  b_+(M) + b_- (M)$$
is exactly the second Betti number of $M$. Notice that the latter is a key contributor   to the {\em Euler characteristic}
$$
\chi (M) = \sum_j (-1)^j b_j (M) = 2 - 2b_1 (M) + b_2 (M)
$$
of our  compact, connected, oriented $4$-manifold $M$. 
In particular,  when $M$ is {\em simply connected}, we just  have $\chi (M) = 2 + b_+(M) + b_-(M)$, so 
that the pair $(\chi (M), \tau (M))$ then completely determines the pair $(b_+(M) , b_-(M))$.

Now  let us temporarily equip our smooth compact oriented manifold $M$ with some  fixed 
Riemannian metric $g$. 
The  Hodge theorem then tells us that each  deRham class
on $M$ has a unique harmonic representative with respect to $g$, and thus gives us   a canonical identification 
$$H^2(M,{\mathbb R}) =\{ \varphi \in \Gamma (\Lambda^2) ~|~
d\varphi = 0, ~ d\star \varphi =0 \}, $$
where $\star$ of course denotes the Hodge star operator of our chosen metric. \
However,  notice that  $\star$ defines an involution of  the right-hand side, 
since  $\star^2 = I$. The corresponding  eigenspace  decomposition then becomes 
\begin{equation}
	H^2(M, {\mathbb R}) = {\mathcal H}^+_{g}\oplus {\mathcal H}^-_{g},
	\label{harm}
\end{equation}
where
$${\mathcal H}^\pm_{g}= \{ \varphi \in \Gamma (\Lambda^\pm) ~|~
d\varphi = 0\} $$
are the spaces of self-dual and anti-self-dual harmonic forms.
But since an arbitrary $2$-form can be 
uniquely expressed  as 
$$\varphi = \varphi^+ + \varphi^-$$
with  $\varphi^\pm \in \Lambda^\pm$, we also  have a key point-wise equality 
$$\varphi \wedge \varphi = \langle \varphi, \star \varphi \rangle \,  d\mu_g = \Big( |\varphi^+|^2 - |\varphi^-|^2\Big) d\mu_g , $$
where $d\mu_g$ denotes the metric volume form associated with $g$ and the
fixed orientation. It follows that the intersection form $Q$ of \eqref{int-form} is 
 positive-definite when restricted to 
${\mathcal H}^+_{g}$, and  is negative-definite when restricted to 
${\mathcal H}^-_{g}$. Moreover, a similar argument shows these two subspaces are 
mutually orthogonal with respect to $Q$.
Thus,   an $L^{2}$-orthonormal  basis for ${\mathcal H}^{+}_{g}$, followed by  
 an $L^{2}$-orthonormal  basis for 
${\mathcal H}^{-}_{g}$, results  in   a basis for $H^{2}({\mathbb R})$
in which the intersection form is represented by  the diagonal matrix  
\eqref{int-matrix}, with 
\begin{equation}
\label{overshoes}
b_\pm  (M) = \dim {\mathcal H}^\pm_{g}.
\end{equation}
This in particular shows that the dimensions of these two spaces of harmonic $2$-forms  
are independent of our choice of  Riemannian metric $g$  on $M$.

The assignment $g\mapsto {\mathcal H}_g^+\subset H^2_{dR} (M, \RR)$  provides us with   a natural  map
$$\{\mbox{Riemannian metrics on }M\}\longrightarrow Gr^+_{b_+}[H^2(M,\RR)]$$
from the infinite-dimensional space of all metrics to 
the finite-dimensional Grassmannian of $b_+(M)$-dimensional 
subspaces of $H^2(M,\RR)$ on which the intersection form $Q$  is 
positive-definite. This map is called the
{\em period map} of $M$. It is actually  invariant under the
action of the identity component $\Diff_0(M)$
of the diffeomorphism group on the space of metrics, and can also be shown to be  
 invariant under the action of the smooth functions 
$M\to \RR^+$ by 
conformal rescaling.  
It is a  more subtle   fact  that this 
 period map is smooth, and has no critical points \cite{don}. One useful corollary is  the following:
 \begin{prop} \label{snap}
 If $M$ is any smooth compact manifold with $b_+\geq 1$, 
 then 
 $\bigcup_g (\mathcal{H}^+_g -0)$ is nonempty and open in $H^2(M, \RR)$. \end{prop}


We  have already  observed  that, on a {\em simply connected} $4$-manifold,  knowing the  Euler characteristic $\chi$  and signature $\tau$
  completely determines the isomorphism type 
 of the intersection form \eqref{int-form} on cohomology with real coefficients. However, the 
 natural  integer-coefficient refinement 
\begin{eqnarray}
Q_\ZZ  :
H^{2}(M, {\mathbb Z})\times H^{2}(M, {\mathbb Z})	
 & \longrightarrow & ~~~~ {\mathbb Z}  \label{int-form-z} \\
	( ~ \mathsf{a} \quad , \quad  \mathsf{b} ~) \qquad & 
	\longmapsto  & \langle \mathsf{a}\, \cup \, \mathsf{b}, [M]\rangle  \nonumber
\end{eqnarray}
of this pairing  certainly contains more information about $M$ than the real-coefficient  version   \eqref{int-form}.
In particular, one says that $Q_\ZZ$ is {\em even} if $Q_\ZZ (\mathsf{a}, \mathsf{a})$ is an even integer for every $\mathsf{a}\in H^2(M ,\ZZ)$;
otherwise, one says that $Q_\ZZ$  is {\em odd}. If an oriented  $4$-manifold $M$ is {\sf spin}, in the sense that $w_2(TM)=0$, then its
intersection form $Q_\ZZ$ is necessarily  even. But when $M$ is {\em simply connected},  the converse is also true:   a compact simply-connected
$4$-manifold is spin {\sf if and only if}  its integer intersection form $Q_\ZZ$  is even. This observation plays an important   role in    enabling  the following remarkable result:

\begin{thm}[Freedman/Donaldson]\label{fdmn}
Two smooth  simply-connected compact oriented $4$-manifolds are orientedly {\sf homeomorphic} 
if and only if 
\begin{itemize}
\item 
they have the same Euler characteristic $\chi$;
\item 
they have the same signature $\tau$; and 
\item both are spin, 
or both are 
non-spin.
\end{itemize} 
\end{thm}

Notice  that Theorem \ref{fdmn} only  concerns   {\em smooth} simply-connected  four-manifolds,   even though it concludes by 
only providing information about 
 when two such manifolds are 
{\em  homeomorphic}, rather than  diffeomorphic. (In fact,   we will eventually   see that  most such manifolds   admit infinitely many inequivalent smooth structures,  and that  this has 
a profound impact on the theory of Einstein $4$-manifolds.)
 Michael Freedman, using wildly non-differentiable techniques, actually    showed  \cite{freedman} that 
 compact simply-connected {\em topological} 
$4$-manifolds are classified by the isomorphism types of their intersection forms $Q_\ZZ$, in conjunction  with a  $\ZZ_2$-valued  invariant (due to Kirby and Siebenmann) that however   vanishes
for manifolds that admit   smooth structures. 
On the other hand,  Milnor \cite{milmex}  had  emphasized,   in a trailblazing  early paper on the   homotopy types  of  simply-connected $4$-manifolds,   
  that, when $Q$ is indefinite,   the isomorphism type of the unimodular form $Q_\ZZ$ is completely determined by the 
invariants listed in Theorem \ref{fdmn}; see  \cite{hm} for further details.  Donaldson's thesis \cite{donaldson} then  completed  the story by using moduli spaces of self-dual Yang-Mills instantons to prove   that the only  definite forms  that arise as   intersection forms of  smooth simply-connected $4$-manifolds 
are  the ``diagonal'' ones   represented  by plus-or-minus  the identity matrix.

In the non-spin case, Theorem \ref{fdmn} can be restated more concretely in terms of  the notion  of {\em connected sums}. 
If  $M_{1}$ and $M_{2}$ are  smooth connected compact oriented $n$-manifolds
\begin{center}
\mbox{
\beginpicture
\setplotarea x from 0 to 290, y from 10 to 50
\ellipticalarc axes ratio 3:1  360 degrees from 140 40
center at 110 30
\ellipticalarc axes ratio 3:1  -360 degrees from 185 40
center at 215 30
\ellipticalarc axes ratio 4:1 -180 degrees from 125 33
center at 110 33
\ellipticalarc axes ratio 4:1 145 degrees from 120 30
center at 110 29
\ellipticalarc axes ratio 4:1 180 degrees from 200 33
center at 215 33
\ellipticalarc axes ratio 4:1 -145 degrees from 205 30
center at 215 29
\endpicture
}
\end{center}
then their connected sum $M_{1}\# M_{2}$ is by definition  the smooth connected oriented $n$-manifold obtained 
by deleting a standard small ball from each manifold 
\begin{center}
\mbox{
\beginpicture
\setplotarea x from 0 to 290, y from 10 to 50
\ellipticalarc axes ratio 3:1  270 degrees from 145 40
center at 115 30
\ellipticalarc axes ratio 3:1  -270 degrees from 180 40
center at 210 30
\ellipticalarc axes ratio 4:1 -180 degrees from 130 33
center at 115 33
\ellipticalarc axes ratio 4:1 145 degrees from 125 30
center at 115 29
\ellipticalarc axes ratio 4:1 180 degrees from 195 33
center at 210 33
\ellipticalarc axes ratio 4:1 -145 degrees from 200 30
center at 210 29
\ellipticalarc axes ratio 1:4 360 degrees from 157 36
center at 157 30
\ellipticalarc axes ratio 1:3 180 degrees from 168 36
center at 168 30
{\setlinear 
\plot 145 40        157 36   /
\plot 145 20        157 24  /
\plot 180 40        168 36   /
\plot 180 20        168 24   /
}
\endpicture
}
\end{center}
and then identifying the
resulting $S^{n-1}$ boundaries 
\begin{center}
\mbox{
\beginpicture
\setplotarea x from 0 to 290, y from 10 to 50
\ellipticalarc axes ratio 3:1  270 degrees from 150 40
center at 120 30
\ellipticalarc axes ratio 3:1  -270 degrees from 175 40
center at 205 30
\ellipticalarc axes ratio 4:1 -180 degrees from 135 33
center at 120 33
\ellipticalarc axes ratio 4:1 145 degrees from 130 30
center at 120 29
\ellipticalarc axes ratio 4:1 180 degrees from 190 33
center at 205 33
\ellipticalarc axes ratio 4:1 -145 degrees from 195 30
center at 205 29
{\setquadratic 
\plot 150 40    163 37    175 40   /
\plot 150 20    163 23    175 20  /
}
\endpicture
}
\end{center}
via a reflection. 
But now  notice that 
if $M_1$ and $M_2$ are  simply connected $4$-manifolds, then $M_1\# M_2$
is also simply connected, and that the Mayer-Vietoris sequence moreover allows us to deduce that   $$b_\pm (M_1\# M_2)= b_\pm (M_1) + b_\pm (M_2).$$ 
If we now use 
$\CP_{2}$ to denote the complex projective plane with its {\em standard} 
orientation, and $\overline{\CP}_{2}$ denote the same smooth $4$-manifold
with the {\em opposite} orientation, then the 
iterated connected sum
 $$k \CP_{2}\# \ell \overline{\CP}_{2}
 =
\underbrace{\CP_{2}\# \cdots \# \CP_{2}}_{k } 
\# \underbrace{\overline{\CP}_{2}\# \cdots \# \overline{\CP}_{2}}_{\ell }$$
is therefore a simply connected $4$-manifold with $b_+=k$ and $b_-=\ell$. 
Since these manifolds   all have $w_2 \neq 0$, 
it is therefore easy to see that Theorem \ref{fdmn} asserts that the connected sums $k \CP_{2}\# \ell \overline{\CP}_{2}$  provide a complete list, albeit  only up to oriented {\em  homeomorphism}, of the  smooth compact non-spin 
simply-connected  $4$-manifolds.

In the spin case, one can formulate an analogous conjectural   list, but the question of whether it is genuinely  complete  remains unsettled.  In Theorem \eqref{hitme}, recall that we previously encountered
the  smooth compact $4$-manifold called $K3$; it  arises as the  smooth  manifold underlying a non-singular  quartic complex surface $\CP_3$,  and
  is a smooth, simply-connected compact oriented spin $4$-manifold with $\chi = 24$ and $\tau = -16$.
Taking iterated connected sums of copies of this example with copies of  the simply-connected spin manifold  $S^2 \times S^2$ then gives us an extensive, but quite manageable,    catalogue 
$$\ell  (K3) \# m (S^2 \times S^2)  =
\underbrace{K3 \# \cdots \# K3}_\ell
\# \underbrace{(S^2 \times S^2 )\# \cdots \# (S^2 \times S^2)}_{m}$$
of  smooth simply-connected compact spin four-manifolds. Conjecturally, along with the $4$-sphere $S^4$, these    provide   a complete list,
up to {\em unoriented homeomorphism}, of the smooth compact simply-connected spin $4$-manifolds. 
In fact, because {\sf Rokhlin's Theorem} \cite{rokhlin} asserts that  $\tau \equiv 0 \bmod 16$ for any smooth compact spin $4$-manifold,
one can easily  check that Theorem \ref{fdmn} guarantees that the examples we've listed precisely exhaust all the smooth simply-connected spin homeotypes for which 
\begin{equation}
\label{eighths}
b_2 (M) \geq \frac{11}{8} |\tau (M)|.
\end{equation}
The so-called $11/8$ conjecture is the assertion  that \eqref{eighths} must in fact hold for all compact simply-connected spin $4$-manifolds. 
However, as of 2026, this conjecture is still  open, and the 
 strongest  theorem known in this direction remains the following result  of  Furuta \cite{fur108}:  any smooth compact
simply-connected spin $4$-manifold $M$ with $b_2\neq 0$ must satisfy the weaker inequality 
$$
b_2 (M) \geq \frac{5}{4} |\tau (M)|+2=\frac{10}{8} |\tau (M)|+2.
$$
However, we will see in Corollary \ref{biz}  below  that the $11/8$ conjecture  does indeed  hold for simply-connected spin $4$-manifolds
that {\em admit  Einstein metrics}. Thus,  any simply-connected Einstein spin $4$-manifold must be (unorientedly) homeomorphic to either $S^4$ or to a connected sum of copies of $K3$ and 
$S^2\times S^2$.

We will eventually discuss  many  results about   simply-connected Einstein four-manifolds that will make use of     
Theorem \ref{fdmn} in order to  clearly contrast the roles of homeotype and diffeotype 
in the theory. However, 
  the results outlined  in \S \ref{optic} certainly make it clear  that 
 Einstein $4$-manifolds with large  fundamental groups are also of fundamental interest. 
On the other hand, the classification problem for smooth  compact $4$-manifolds of  arbitrary  fundamental group is    in general   formally undecidable,  because
 any finitely-presented group occurs as the fundamental group of such a space, and it  has long been  known  \cite{undecidable} that
 it is impossible to write an 
  algorithm that could  always determine  when two such groups are isomorphic. Still, there are  geometrically interesting 
 classes of infinite groups, including the fundamental groups of locally symmetric spaces, for  which this undecidability issue can be avoided.
Whether  the fundamental groups of general Einstein $4$-manifolds are sufficiently special as to similarly  tip the scales in favor of decidability 
 is still apparently  an open problem --- and one that is  surely  worth settling!

\section{The Hitchin-Thorpe Inequality}
\label{hitchstar}

In  \S \ref{non-simple},  we previously met the oriented rank-$3$ real vector 
bundle $\Lambda^+\to M$ of self-dual  $2$-forms over an oriented Riemannian $4$-manifold $(M,g)$.
But what does this bundle look like, topologically? Assuming that $M$ is compact
and connected, $\mathbf{SO}(3)$-bundles  $E\to M$   are classified, up to bundle isomorphism,  
by  the Pontryagin number  $\mathbf{p_1}(E) = \langle p_1 (E) , [M]\rangle\in \ZZ$  
and  the Stieffel-Whitney class $w_2 (E) \in H^2(M, \ZZ )$. 
In the case of $\Lambda^+\to M$, these invariants are given by
\begin{eqnarray} \label{instant}
\mathbf{p_1}(\Lambda^+) = (2\chi+ 3\tau ) (M) \\
w_2 (\Lambda^+) = w_2 (TM)
\end{eqnarray}
as can for example be proved by pulling $\Lambda^+$ back to the twistor space $$S(\Lambda^+)\stackrel{\wp}{\to} M,$$
where $\wp^*\Lambda^+\equiv \RR \oplus K_Z^{1/2}$ with respect to the standard twistor almost-complex structure on $Z=S(\Lambda^+)$. 
For details, see \cite{hitka,lebsphinx}.

For us, equation \eqref{instant}  has particularly important consequences,  because if $\triangledown$ is any 
connection on an $\mathbf{SO}(3)$-bundle $E\to M$, with curvature $F_\triangledown = F^+_\triangledown + F^+_\triangledown$, then 
the  Pontryagin number of $E$ is given by 
\begin{equation}
\label{pont}
\mathbf{p_1}(E) = \frac{1}{8\pi^2} \int_M (|F^+_\triangledown|^2 - |F^-_\triangledown|^2 )\, d\mu_g .
\end{equation}
If $(E,\triangledown )$  satisfies $F^-_\triangledown\equiv 0$, thus 
making it   a  {\em self-dual Yang-Mills instanton},  then its ``instanton number'' $\mathbf{p_1}(E)$  must satisfy  $\mathbf{p_1}(E)\geq 0$, with 
equality iff $\triangledown$ is a {\em flat} connection on $E$.  
But for  $\Lambda^+$, equipped with the connection induced by the Levi-Civita connection, we saw in \S \ref{non-simple}  that $(\Lambda^+, \nabla )$   is a self-dual instanton on $(M,g)$ 
iff $g$ is Einstein. This immediately yields the following remarkable result of Hitchin \cite{hit}, which refined and clarified  some earlier independent observations of  Thorpe \cite{tho} and Gray \cite{gray-ht}: 

\begin{thm}[Hitchin-Thorpe Inequality]\label{ht}
If the smooth compact oriented 4-manifold
$M$ admits an Einstein metric $g$, then 
\begin{equation}
\label{htineq}
(2\chi + 3\tau )(M) \geq 0 .
\end{equation}
Moreover,  equality holds in \eqref{htineq} iff $(M,g)$  has a finite Riemannian cover that is  either  a flat $4$-torus $T^4$
or  a   Calabi-Yau  $K3$.  
\end{thm}

While  the first part  of   Theorem \ref{ht}  clearly follows from the above discussion, the second part requires further explanation.
We have already seen that in the equality case, $(M,g)$ would have to be Ricci-flat, and that  $\Lambda^+$ would have to be flat with respect to the Levi-Civita connection.
  However, 
 the Cheeger-Gromoll splitting theorem \cite{cg} implies that the universal cover $(\widetilde{X},\widetilde{g})$ 
 of any compact Ricci-flat $n$-manifold $(X,g)$  must be    the  Riemannian 
 product of a  simply-connected compact Ricci-flat manifold  $(Y^{n-k},h)$ with   a  Euclidean space $\RR^{k}$ of  complementary   dimension  $k\in \{ 0, \ldots , n\}$.
 But since we have already seen that any Ricci-flat manifold of dimension  $< 4$ is actually flat, this   means, in our case,  that  either   $(\widetilde{M}^4,\widetilde{g})$  is compact,
  or  
  $(M^4,g)$ is flat.  However, \label{ricflat} 
  if  $(M^4,g)$ is flat,   Bieberbach's Theorem
  \cite{bie}  asserts that some finite cover of $(M^4,g)$ is a flat torus, and  we are  done. Otherwise, the flatness of $\Lambda^+$ implies that 
   the compact universal cover $(\widetilde{M},\widetilde{g})$ 
  has holonomy  $\subset \mathbf{SU}(2)$, thus making   $\widetilde{g}$  a Calabi-Yau metric on a simply-connected compact complex surface $(\widetilde{M},J)$
  with $c_1=0$.
     The classification theory  of compact complex surfaces \cite{bpv,kodrr} therefore  tells us   that $(\widetilde{M},J)$ must   be a $K3$ surface,
     and so must be diffeomorphic  to a non-singular  quartic  hypersurface in $\CP_3$. 
     
    \bigskip 
          
  Notice that  Theorems \ref{toro} and  \ref{hitme} 
  have  now  become   immediate corollaries of Theorem \ref{ht}.
 Another immediate consequence is the following:
     
     \begin{cor} \label{biz}
      Let $M$ be a smooth simply-connected spin $4$-manifold that admits an Einstein metric $g$. Then $M$ is (unorientedly) homeomorphic 
     to either $S^4$ or  a connected sum 
     $$\ell ( K3 )\# m (S^2 \times S^2),$$
     where   $m \geq \ell$ if $\ell > 1$.
     \end{cor}
     \begin{proof}
     By reversing the orientation of $M$ if necessary, we may assume that $\tau (M) \leq 0$. By the Rokhlin's theorem \cite{rokhlin}, Donaldson's theorem \cite{donaldson}, and the classification \cite{hm} of
     unimodular forms over $\ZZ$, the intersection form of $M$ is then  isomorphic to $2\ell \mathsf{E}_8\oplus k \mathsf{H}$ for suitable natural numbers $k$ and $\ell$,
     where 
$$ \mathsf{E}_8:= \left[\begin{array}{rrrrrrrr}-2 & 1 &  &  &  &  &  &  \\1 & -2 & 1 &  &  &  &  &  \\ & 1 & -2 & 1 &  &  &  &  \\ &  & 1 & -2 & 1 &  &  &  \\ &  &  &1 & -2 & 1 &  & 1 \\ &  &  &  &1 & -2 & 1 &  \\ &  &  &  &  & 1 & -2 &  \\ &  &  &  & 1 &  &  & -2\end{array}\right]$$
is a celebrated  negative-definite  unimodular  form on $\ZZ^8$, 
and where 
$$
\mathsf{H} := \left[\begin{array}{cc}0 & 1 \\1 & 0\end{array}\right]
$$
is the intersection form of $S^2 \times S^2$. 
    But a  simply-connected $4$-manifold $M$ with  intersection form  $2\ell \mathsf{E}_8\oplus k \mathsf{H}$ must have  $\chi (M) = 2+ 16 \ell + 2k$ and $\tau (M) = -16 \ell$, and hence 
     $(2\chi + 3\tau )(M) = 4 +4k - 16 \ell$. Thus,   $(2\chi + 3\tau )(M) \geq 0$ iff $k \geq 4\ell  -1$, and $(2\chi + 3\tau )(M) > 0$  iff $k \geq 4\ell$. Since \eqref{htineq} can only
     be saturated by a simply-connected Einstein $4$-manifold if it is diffeomorphic $K3$, and since the intersection form of $K3$ is isomorphic to $2 \mathsf{E}_8\oplus 3 \mathsf{H}$, this shows that
      any correctly-oriented  simply-connected     Einstein, spin $4$-manifold  will  have the same intersection form as $\ell ( K3 )\# m (S^2 \times S^2)$,
     where $m:= k-3\ell$ moreover satisfies $m \geq \ell$ if $\ell > 1$. Theorem \ref{fdmn} therefore  guarantees that any such $M$ must  be homeomorphic to $\ell ( K3 )\# m (S^2 \times S^2)$, where $m\geq \ell$
     if $\ell >1$.
     \end{proof}

In the non-spin case,  we similarly obtain the following::

 \begin{cor} \label{wiz}
  Let $M$ be a smooth simply-connected non-spin $4$-manifold that admits an Einstein metric $g$. 
  Then $M$ is orientedly homeomorphic to $k \CP_2\# \ell \cpbar$ for some $k,\ell \in \NN$  
   satisfying  $5k+3 \geq  \ell \geq (k-3)/5$.
  \end{cor}
  \begin{proof} We  already observed that Theorem \ref{fdmn} implies that any simply connected non-spin $M$ must be homeomorphic 
  to $k \CP_2\# \ell \cpbar$  for some $(k,\ell)\neq (0,0)$. We then have $(2\chi + 3\tau )(M) = 4+ 5k - \ell$. Since Theorem \ref{ht} implies that
 equality in  \eqref{htineq} cannot hold for a simply-connected non-spin manifold,  it follows in our setting that a necessary condition for the existence of 
 an Einstein metric $g$ on $M$ is that $5k+3 \geq  \ell$. But since the Einstein condition is independent of orientation, and since change of orientation 
 interchanges $k$ and $\ell$, we must also have $5\ell +3 \geq  k$ if $M$ admits an Einstein metric. 
  \end{proof} 
  
 Indeed, systematically reversing the orientation of $M$  allows us to deduce the following useful  consequence of  Theorem \ref{ht}: 

\begin{cor}
\label{mira}
If the smooth compact oriented 4-manifold
$M$ admits an Einstein metric $g$, then we also have 
\begin{equation}
\label{revineq}
(2\chi - 3\tau )(M) \geq 0 .
\end{equation}
Moreover,  equality holds in \eqref{revineq} iff $(M,g)$  has a finite Riemannian cover that is  either  a flat $4$-torus $T^4$
or  a   (reverse-oriented) Calabi-Yau  $\overline{K3}$.  
\end{cor}

 This can in fact be proved  in a rather different way, by applying our previous  observation 
that $(M,g)$ is Einstein iff 
$(\Lambda^-,\nabla )$ is  an {\em anti-self-dual} Yang-Mills instanton. On a compact
oriented Einstein $4$-manifold $(M,g)$, this then tells us that 
the ``instanton number'' of $\Lambda^-$ must satisfy
$$\mathbf{p_1}(\Lambda^-) \leq 0,$$
with equality iff the Levi Civita connection $\nabla$ on $\Lambda^-$ is flat. However, since 
$$-\mathbf{p_1}(\Lambda^-) = \langle p_1(\Lambda^-), [\overline{M}]\rangle = ( 2\chi - 3\tau ) (M),$$
Corollary \ref{mira} then follows rather directly.  

But all of this discussion  can also  be given  more  immediate  geometric meaning by  writing out \eqref{pont} explicitly for $(\Lambda^+, \nabla )$ as
\begin{equation}
\label{gb+}
(2\chi + 3\tau ) (M) = \frac{1}{4\pi^2}\int_M \left(\frac{s^2}{24}+2|W_+|^2
  -\frac{|\stackrel{\circ}{r}|^2}{2}
\right)d\mu_g
\end{equation}
and for $(\Lambda^-, \nabla )$ as 
\begin{equation}
\label{gb-}
(2\chi - 3\tau ) (M) = \frac{1}{4\pi^2}\int_M \left(\frac{s^2}{24}+2|W_-|^2
  -\frac{|\stackrel{\circ}{r}|^2}{2}
\right)d\mu_g~.
\end{equation}
In particular, when $g$ is Einstein, we therefore have 
\begin{equation}
\label{cible}
(2\chi + 3\tau ) (M) = \frac{1}{4\pi^2}\int_M \left(\frac{s^2}{24}+2|W_+|^2
\right)d\mu_g \geq 0 
\end{equation}
and 
\begin{equation}
\label{but}
(2\chi - 3\tau ) (M) = \frac{1}{4\pi^2}\int_M \left(\frac{s^2}{24}+2|W_-|^2
\right)d\mu_g \geq 0 .
\end{equation}
As we will see 
in later sections of this article,  Theorem \ref{hitme} and  Corollary \ref{mira} can therefore  be interestingly  improved whenever we can obtain   non-trivial lower bounds
for the integrals appearing in \eqref{cible} and \eqref{but}. 
It is also worth pointing out  that the sum and difference of equations (\ref{gb+}--\ref{gb-}) now also  yield  the  $4$-dimensional {\sf Gauss-Bonnet} and {\sf signature} formulas
  \begin{eqnarray} \label{cgb} 
\chi (M)&=& \frac{1}{8\pi^2}\int_M \left(\frac{s^2}{24}+|W^+|^2
+ |W^-|^2  -\frac{|\stackrel{\circ}{r}|^2}{2}
\right)d\mu \\ \label{thr} 
\tau (M)&=& \frac{1}{12\pi^2}\int_M \Big(|W^+|^2
- |W^-|^2  \Big) d\mu 
\end{eqnarray}
that  are  the usual starting points for most proofs of \eqref{htineq} and \eqref{revineq}; cf. e.g. 
 \cite{bes,gray-ht,hit,tho} or  other standard references. 

%
%

\

%

\section{The Seiberg-Witten Equations}
\label{sweet}

In \S \ref{hitchstar}, we saw that Theorem \ref{ht} and Corollary \ref{mira} imply that  many smooth compact $4$-manifolds 
do not admit Einstein metrics. However, we also noticed that these results could be strengthened if we could somehow find positive  lower bounds
for the right-hand integrals in \eqref{gb+} or \eqref{gb-}. In this section, we will introduce  the 
 Seiberg-Witten equations, and begin to see how they  can in fact  give rise
to positive lower bounds for these expressions. Remarkably, however, the resulting strengthened   obstructions to the existence of Einstein metrics  will  
depend on  the {\em diffeomorphism} type, rather than the {\em homeomorphism} type,
of the $4$-manifold in question.  This ultimately reflects the fact that the $4$-manifold invariants at play here are  defined directly in terms of   solution spaces of  certain PDE,
which therefore  directly involves  a choice  of  smooth structure on the relevant manifold. These 
 PDE in fact crucially involve a Dirac operator
coupled to a connection on a line bundle. To provide a coherent context for this, we will therefore need to begin by  briefly explaining the notions  of {\em spin} and  {\em spin$^c$ structures}
on  $4$-manifolds; for  more thorough  treatments of these topics,   see   \cite{lawmic,lebsphinx,morgan}.

We previously noted  in \eqref{spin4} that  the universal cover $\mathbf{Spin}(4)$ of $\mathbf{SO}(4)$ is isomorphic to 
the product $\mathbf{Sp}(1)\times \mathbf{Sp}(1)$ of two copies of the multiplicative group  $\mathbf{Sp}(1)$  of unit-norm quaternions. Given an oriented  Riemannian 
$4$-manifold, this makes it natural to ask whether the bundle $\mathfrak{F} \to M$ of oriented orthonormal frames has a double cover
$\widetilde{\mathfrak{F}} \to\mathfrak{F}$ which is fiber-wise modeled on the universal cover $[\mathbf{Sp}(1)\times \mathbf{Sp}(1)]\to \mathbf{SO}(4)$;
when this is possible,  a choice of such an $\widetilde{\mathfrak{F}}$ is called a {\em spin structure} on $(M,g)$.  One can show that $M$ admits a spin structure if and only if 
 $w_2(TM)=0$; and  when this happens,  one then says that $M$ is a {\em spin manifold}. When $M$ is spin,  $H^1(M,\ZZ_2)$   then acts freely and transitively on the 
 set of spin structures on $M$. Because the Gram-Schmidt procedure defines a natural diffeomorphism between the oriented-orthonormal-frame bundles
 associated to any two  metrics $g$ and $g^\prime$ on $M$, all of these  notions  are actually metric-independent. 
%
%
 
 This  now allows us to define 
  the (chiral) spin bundles of a spin Riemannian    $4$-manifold $(M,g)$.  Since    a choice of 
 spin structure on $(M,g)$  is specified by a  principal $\mathbf{Spin}(4)$-bundle
 $$\begin{array}{lc}
  \mathbf{Sp}(1)\times \mathbf{Sp}(1)    \to   & \widetilde{\mathfrak{F}} \\
          &\downarrow \\
         &M
\end{array} 
$$ 
that double-covers the oriented-orthonormal-frame-bundle $\mathfrak{F}$, 
the  standard isomorphism between $\mathbf{Sp}(1)$ and $\mathbf{SU}(2)$ turns the two factor-projections $ [\mathbf{Sp}(1)\times \mathbf{Sp}(1)]\to \mathbf{Sp}(1)$
into a pair $(\varrho_+, \varrho_-)$ of  independent  representations of $\mathbf{Spin}(4)$ on $\CC^2$. The associated bundle construction 
 therefore gives rise to a pair of rank-$2$ Hermitian vector bundles 
 $$\mathbb{S}_\pm =  \widetilde{\mathfrak{F}} \times_{(\mathbf{Spin}(4), {\varrho_\pm})} \CC^2$$
 over $M$, both of which  come equipped with built-in trivializations 
  $$\wedge^2 \mathbb{S}_+=\wedge^2 \mathbb{S}_-= {\CC}.$$
 Since  the bundles $TM$, $\Lambda^+$ and $\Lambda^-$ also arise from the associated bundle construction for specific  representations
 of $\mathbf{SO}(4)$ (and hence of $\mathbf{Spin}(4)$),  careful inspection now reveals that we also have canonical isomorphisms
\begin{eqnarray}
\Lambda^1 \otimes \CC  &=& \Hom (\mathbb{S}_+, \mathbb{S}_-) \label{dog1}\\
\Lambda^+ \otimes \CC &=& \End_0( \mathbb{S}_+) \label{dog+}\\
\Lambda^- \otimes \CC &=& \End_0 (\mathbb{S}_-), \label{dog-}
\end{eqnarray}
where $\End_0$ denotes the trace-free endomorphisms of the relevant bundle. 
On the other hand, since   the Levi-Civita connection on $TM$ naturally induces  a principal $\mathbf{SO}(4)$-connection on the
oriented-orthonormal-frame bundle  $\mathfrak{F}$, and because this then pulls back to a principal $\mathbf{Spin}(4)$-connection
on $\widetilde{\mathfrak{F}}$, the bundles $\mathbb{S}_\pm$ carry induced Hermitian connections $\nabla$ that  are compatible with  their  $\mathbf{SU}(2)=\mathbf{Sp}(1)$
structures. By composing the vector-bundle morphism 
$\bullet : {\Lambda}^1\otimes \mathbb{S}_+
\to \mathbb{S}_-$ induced by \eqref{dog1}
with the covariant derivative operator $$\nabla : \Gamma (\mathbb{S}_+ )\to \Gamma (\Lambda^1 \otimes \mathbb{S}_+ ),$$ 
 we then obtain the (chiral) {\em Dirac operator}

\setlength{\unitlength}{1ex}
\begin{center}\begin{picture}(36,17)(0,3)
\put(10,17){\makebox(0,0){$ \Gamma ({ {\mathbb S}}_+)$}}
\put(18,19){\makebox(0,0){$\dir$}}
\put(18,5){\makebox(0,0){$\Gamma ({\Lambda}^1\otimes { {\mathbb S}}_+)$}}
\put(26,17){\makebox(0,0){$\Gamma ({ {\mathbb S}}_-)$}}
\put(15,12){\makebox(0,0){${\nabla}$}}
\put(21,12){\makebox(0,0){${\Large \bullet} $}}
\put(11,15.5){\vector(2,-3){6}}
\put(19,6.5){\vector(2,3){6}}
\put(14,17){\vector(1,0){8}}
\end{picture}\end{center}
which is an  first-order operator, with  index  given   \cite{hitharm,lawmic} by
\begin{equation}
\label{ind-dir}
\Ind (\dir ) =\dim \ker (\dir ) -\dim \ker (\dir^* ) = - \frac{\tau (M)}{8}.
\end{equation}
Moreover, the interchanging $\mathbb{S}_+$ and $\mathbb{S}_-$ in this construction simply yields
the adjoint $\dir^*$ of our Dirac operator. 
This has two striking  consequences. 

First, because $\mathbb{S}_\pm$ are quaternionic line bundles, 
each carries a parallel complex-anti-linear endomorphism  ${\zap j}: \mathbb{S}_\pm \to \mathbb{S}_\pm$ with ${\zap j}^2=-I$,
and the Dirac operator and its adjoint therefore  satisfy $\dir\circ {\zap j} = {\zap j}\, \circ\dir$ and $\dir^*\circ{\zap j} = {\zap j}\, \circ\dir^*$. This makes both 
 $\ker (\dir )$ and $\ker (\dir^* )$ into  quaternionic vector spaces, so   both  must consequently  have {\em even} complex dimension. It follows
 that $\Ind (\dir )$ must be even, and \eqref{ind-dir} therefore implies {\sf Rokhlin's Theorem}  \cite{hitharm,lawmic,rokhlin}: {\em the signature
 of any smooth spin compact $4$-manifold must be divisible by $16$.} Here we remind the reader  that this fact has  previously been mentioned, and  played an important  role    in our discussion 
 of the spin case of Theorem \ref{fdmn}.

Second, the so-called Lichnerowicz Weitzenb\"ock formula \cite{hitharm,lawmic,lic} 
\begin{equation}
\label{lichne}
\dir^*\!\! \dir = \nabla^*\nabla +\frac{s}{4}
\end{equation}
now gives rise to obstructions to the existence of positive-scalar-curvature metrics on certain $4$-manifolds. 
Indeed, if $M$ is a smooth compact spin $4$-manifold with signature $\tau (M) < 0$, then, for any Riemannian metric
$g$ on $M$, the index formula \eqref{ind-dir} predicts that there must be some  $\Phi \in \Gamma (\mathbb{S}_+ )$ 
that satisfies $\dir \Phi =0$ and $\Phi\not\equiv 0$. Applying \eqref{lichne} to $\Phi$, and then taking the $L^2$-inner-product with $\Phi$, 
we then have 
$$0=\int_M |\dir \Phi |^2 d\mu_g = \int_M  |\nabla \Phi |^2  d\mu_g + \int_M \frac{s}{4} |\Phi |^2  d\mu_g\geq \int_M \frac{s}{4} |\Phi |^2  d\mu_g ,$$ 
which would be  a contradiction if the scalar curvature $s$ of $g$ were everywhere positive.
 For example,  this argument shows that 
 $K3$, which has $\tau = -16$, cannot admit a metric with $s>0$; and indeed, a refinement of the same argument shows that 
 the only Riemannian metrics on $K3$ with  $s\geq 0$ are the Ricci-flat, hyper-K\"ahler ones.  By reversing the orientation of $M$ when necessary, 
 we  also now see \cite{hitharm,lawmic} that a smooth compact spin $4$-manifold $M$ 
 with  $\tau (M) \neq  0$ can  thus never admit metrics of positive scalar curvature.
 
 However,  we have also seen that non-spin $4$-manifolds exist in  profusion, and it is only natural to wonder whether any analogue of these 
 arguments might be possible in a non-spin setting. Of course, 
  the above construction of the
 Dirac operator  falls apart in the absence of a spin structure.  Nonetheless, one {\em can}  always construct  twisted analogues  of the Dirac operator
 on any smooth compact $4$-manifold, by instead introducing the more flexible notion of  a
{\em spin$^c$ structure}.  
  Here,  a spin$^c$ structure on a Riemannian $4$-manifold $(M, g)$ is specified  by a choice of  a principal $\mathbf{Spin}^c (4)$-bundle
$\widehat{\mathfrak{F}}\to M$, 
where 
$$\mathbf{Spin}^c (4):= [\mathbf{Spin} (4)\times \mathbf{U}(1)]/\ZZ_2 = [\mathbf{Sp} (1)\times \mathbf{Sp} (1)\times \mathbf{U}(1)]/\langle (-1,-1,-1)\rangle ,$$
together with an  isomorphism between the principal $\mathbf{SO}(4)$-bundle 
$\widehat{\mathfrak{F}}/\mathbf{U}(1)$ and the oriented-orthonormal-frame bundle $\mathfrak{F}\to M$.
Up to isomorphism,  such a structure $\mathfrak{c}$ is  determined by the Chern class $c_1\in H^2 (\mathfrak{F},\ZZ)$
of the circle bundle $\widehat{\mathfrak{F}}\to \mathfrak{F}$, where  this  $c_1$   can    be  taken to be any element of $H^2 (\mathfrak{F},\ZZ)$ 
whose restriction to a fiber yields  the non-trivial element of $$H^2 (\mathbf{SO}(4),\ZZ) \cong \ZZ_2.$$ Every smooth compact $4$-manifold $M$ admits \cite{hiho,lebsphinx} 
such structures, while  the cohomology group $H^2(M, \ZZ)$ acts freely and transitively on the spin$^c$ structures on $M$ via  pull-back to $\mathfrak{F}$, followed by addition 
in $H^2(\mathfrak{F}, \ZZ)$.

Now dropping  one $\mathbf{Sp} (1)$-factor  or the other   gives rise to two different surjective homomorphisms
  $$ [\mathbf{Sp} (1)\times \mathbf{Sp} (1)\times \mathbf{U}(1)]/\langle (-1,-1,-1)\rangle
   \longrightarrow [\mathbf{Sp} (1)\times \mathbf{U}(1)]/\langle (-1,-1)\rangle = \mathbf{U}(2)$$
   and hence to a pair  $(\varrho_+, \varrho_-)$ of   different representations of $\mathbf{Spin}^c (4)$ on $\CC^2$. 
   The associated Hermitian vector bundles 
   $$\mathbb{V}_\pm =  \widehat{\mathfrak{F}} \times_{(\mathbf{Spin}^c(4), {\varrho_\pm})} \CC^2$$
   then satify   generalized versions of some key properties 
   \begin{eqnarray}
\Lambda^1 \otimes \CC  &=& \Hom (\mathbb{V}_+, \mathbb{V}_-) \label{bone1}\\
\Lambda^+ \otimes \CC &=& \End_0( \mathbb{V}_+) \label{bone+}\\
\Lambda^- \otimes \CC &=& \End_0 (\mathbb{V}_-), \label{bone-}
\end{eqnarray}
that we previously noticed  in our discussion of  the spin bundles. One  key difference, however, is  that the line-bundle 
$$L=\wedge^2 \mathbb{V}_+=\wedge^2 \mathbb{V}_-$$
is typically  no longer trivial, but  instead  satisfies 
\begin{equation}
\label{lift}
c_1(L) \equiv w_2 (TM) \bmod 2  ,
\end{equation}
in the sense that $\boldsymbol{\varrho} (c_1(L))=w_2 (TM)$ in   the long exact sequence 
\begin{equation}
\label{promo}
\cdots \to H^2 (M, \ZZ ) \stackrel{2\cdot}{\to} H^2 (M, \ZZ ) \stackrel{\boldsymbol{\varrho}}{\to}  H^2 (M, \ZZ_2 )  \stackrel{\boldsymbol{\beta}}{\to}  H^3 (M, \ZZ )\to \cdots
\end{equation}
where  the fact  that  $W_3(M):=\boldsymbol{\beta} (w_2(TM)) =0$ for any $M^4$ is the key point required to   show \cite{hiho,lebsphinx} that every  $4$-manifold $M$ admits spin$^c$ structures. 

At least locally, the spin and spin$^c$ formalisms are related by 
\begin{equation}
\label{formality}
 \mathbb{V}_\pm =  \mathbb{S}_\pm  \otimes  L^{1/2},
\end{equation}
and this   formal statement   also serves a useful purpose globally by reminding us that, 
 while the transition functions of $\mathbb{S}_\pm$ and  $L^{1/2}$ are {\em a priori} sign-ambiguous, 
 these ambiguous  signs can actually be consistently chosen so as to yield the    transition functions of  the well-defined  vector-bundles
$\mathbb{V}_\pm$. At the same time, the free and transitive action of $H^2 (M, \ZZ)$ on spin$^c$ structures  is concretely realized by 
$$\mathbb{V}_\pm \rightsquigarrow   \mathbb{V}_\pm \otimes E$$
where an arbitrary element of $H^2 (M, \ZZ)$ is represented by the first Chern class $c_1(E) \in H^2 (M, \ZZ)$ of
a complex line-bundle $E\to M$. This action   has  the effect that 
$$L \rightsquigarrow   L \otimes E^{\otimes 2}$$
and hence 
$$c_1(L) \rightsquigarrow c_1(L)  + 2 c_1(E) . $$
In particular,  \eqref{promo} guarantees   that   \eqref{lift} is   the only  constraint that needs to  be satisfied  for  an element of $H^2 (M, \ZZ)$ to arise as $c_1(L)$ for a spin$^c$ structure. 
Moreover,  if the $2$-torsion in $H^2 (M, \ZZ)$ is trivial --- as happens, for example,   when $M$ is simply-connected ---
a  spin$^c$ structure on $M$ is then  completely characterized by 
the corresponding 
cohomology classes  $c_1(L)$ satisfying \eqref{lift}. In general, however,  knowing $c_1(L)$ 
will only determine the  spin$^c$ structure modulo the $2$-torsion subgroup of $H^2(M, \ZZ)$.

The formal expression \eqref{formality} also reveals an important   way  in which spin$^c$ structures 
are more loosely tethered to 
Riemannian geometry than are  spin structures,   because  the Levi-Civita connection of $(M,g)$ does not induce  
a preferred connection on the twisted spin-bundles $\mathbb{V}_\pm$. Instead, these bundles inherit a
{\em family} of different Hermitian connections $\nabla_\theta$, parameterized by the set of  principal $\mathbf{U}(1)$ connections $\theta$
on  $L$. But once having made such a choice, we may then follow the covariant-derivative operator 
$$\nabla_{\theta} : \Gamma ({\mathbb V}_{+})\to \Gamma (\Lambda^1\otimes {\mathbb V}_{+})$$
with  the ``Clifford multiplication''  map
$$\bullet: \Lambda^1\otimes {\mathbb V}_{+}\to {\mathbb V}_{-}$$
induced by \eqref{bone+}, and thereby obtain a  
spin$^c$ analogue  
$$\dir_\theta : \Gamma (\mathbb{V}_+) \to \Gamma (\mathbb{V}_-)$$
of the standard Dirac operator:
\setlength{\unitlength}{1ex}
\begin{center}\begin{picture}(36,17)(0,3)
\put(10,17){\makebox(0,0){$ \Gamma ( \mathbb{V}_+)$}}
\put(18,19){\makebox(0,0){$\dir_\theta$}}
\put(18,5){\makebox(0,0){$\Gamma ({\Lambda}^1\otimes { {\mathbb V}}_+)$}}
\put(26,17){\makebox(0,0){$\Gamma (\mathbb{V}_-)$}}
\put(15,12){\makebox(0,0){$\nabla_\theta$}}
\put(21,12){\makebox(0,0){${\Large \bullet} $}}
\put(11,15.5){\vector(2,-3){6}}
\put(19,6.5){\vector(2,3){6}}
\put(14,17){\vector(1,0){8}}
\end{picture}\end{center}
This is once again 
 an elliptic first-order  operator,  but its   index is instead  now given   \cite{hitharm,lawmic} by
\begin{equation}
\label{ind-dir-c}
\Ind (\dir_\theta ) =\dim \ker (\dir_\theta ) -\dim \ker (\dir^*_\theta ) =  \frac{c_1^2(L)-\tau (M)}{8} 
\end{equation}
where $c_1^2(L):= \langle c_1(L)\cup c_1(L), [M] \rangle = Q (c_1(L), c_1(L))$. The Lichnerowicz Weitzenb\"ock formula  \eqref{lichne} now generalizes \cite{hitharm,lawmic} 
as 
\begin{equation}
\label{lichne-c}
\dir_{\theta}^{\ast}\! \dir_{\theta} = \nabla_\theta^*\nabla_\theta + \frac{s}{4} - \frac{1}{2} F_\theta^+\bullet
\end{equation}
where  the self-dual part  $F_{\theta}^{+}\in i  \Lambda^+$ of the curvature of $L$ 
  acts on $\mathbb{V}_+$ via \eqref{bone+}. As a consequence, any $\Phi \in \Gamma (\mathbb{V}_+)$ satisfies 
  \begin{equation}
\label{wtw} 
\langle \Phi , \dir_{\theta}^{\ast}\! \dir_{\theta} \Phi  \rangle = \frac{1}{2}\Delta |\Phi |^2 + |\nabla_{\theta} \Phi |^2 + 
\frac{s}{4} |\Phi |^2 + 2 \langle -iF_{\theta}^{+} , \sigma (\Phi ) \rangle 
\end{equation}
for a  real-quadratic map 
$$\sigma : \mathbb{V}_+\to \Lambda^+$$
that
is  uniquely characterized   by 
  $$\sigma (\Phi)\bullet  \Psi =  i \left[ \langle \Psi , \Phi \rangle \Phi - \frac{1}{2} |\Phi|^2 \Psi \right]$$
  in terms of  the (Clifford) multiplication  defined by  \eqref{bone+}, so that 
\begin{equation}
\label{radial}
|\sigma (\Phi )| = \frac{|\Phi|^2}{2\sqrt{2}}
\end{equation} 
for any $\Phi\in \mathbb{V}_+$, and 
\begin{equation}
\label{clef} 
 \langle \omega , \sigma (\Phi ) \rangle = -\frac{i}{4} \langle  \omega  \bullet \Phi , \Phi \rangle 
\end{equation}
for any complex-valued self-dual $2$-form $\omega \in \CC\otimes \Lambda^+$.

On  a spin manifold, the 
standard Dirac operator is entirely determined by the relevant   Riemannian metric,  and this explains why the Dirac operator
can be used to prove, via  \eqref{lichne},  that 
positive-scalar-curvature metrics cannot exist on $4$-dimensional  spin manifolds of non-zero signature.   
However, this  does not easily  generalize to the spin$^c$ setting, because the curvature
of the connection $\theta$ on the  line bundle $L$ appears  in \eqref{wtw},
and is in principle entirely independent of the Riemannian geometry of $(M,g)$. 
To overcome this, one therefore needs to impose  conditions on $\theta$ in order to somehow  tie it 
more closely to the   Riemannian geometry of $(M,g)$.  A truly remarkable way of doing this,   with powerful geometric ramifications,  was   first  introduced  by Witten 
 \cite{witten}, and  involves  requiring  $\Phi$ and $\theta$ to jointly  solve  a non-linear coupled 
system of PDE. 

Given a spin$^c$ structure on a compact oriented Riemannian $4$-manifold $(M,g)$, the {\sf Seiberg-Witten equations} demand 
  that a  Hermitian connection $\theta$ on $L$ and a twisted spinor $\Phi \in \Gamma (\mathbb{V}_+)$ jointly  satisfy 
 the coupled equations 
 \begin{eqnarray} \dir_\theta\Phi &=&0\label{drc}\\
 F_\theta^+&=&i \sigma(\Phi) .\label{sd}\end{eqnarray}
 However,  \eqref{wtw} and \eqref{radial} then imply  that  $\Phi$  satisfies  
  \begin{equation}
 0=	2\Delta |\Phi|^2 + 4|\nabla_{\theta}\Phi|^2 +s|\Phi|^2 + |\Phi|^4 ,	
 	\label{wnbk}
 \end{equation}
 where $\Delta = d^*d$ is the positive-spectrum Laplacian on functions.
 In particular, 
one immediately sees that the Seiberg-Witten equations  (\ref{drc}--\ref{sd}) cannot admit an {\em irreducible}  solution 
relative to a metric $g$ with  $s >0$, where a solution $(\Phi, \theta )$ is  called {irreducible} if  $\Phi \not\equiv 0$. 
But  why would one even expect  solutions of this system  to exist, anyway? 
The Seiberg-Witten  system is non-linear, so one certainly cannot appeal to an index calculation
to predict the existence of  solutions. Instead,  Witten had the  remarkable insight that one can instead define a 
new invariant of a smooth compact oriented $4$-manifold with  fixed  spin$^c$  structure by  ``counting'' solutions
of these equations,   in a  manner that can then be shown to be metric-independent. 

But  whenever the solution space of the Seiberg-Witten equations  is non-empty, it is automatically infinite-dimensional, 
because 
 the {\em gauge group}  
$$\mathscr{G} = \{ \mbox{smooth maps } {\zap f} : M\to S^1\subset \CC \}$$
of circle-valued functions 
 acts on solutions of   (\ref{drc}--\ref{sd}) by 
 $$(\Phi , \theta ) \longmapsto ({\zap f}\Phi , \theta + 2d\log {\zap f}),$$
 thereby carrying solutions of (\ref{drc}--\ref{sd}) into new solutions that 
 essentially just differ by automorphisms of $L$, and so are  
  geometrically really
 the same. Given a spin$^c$ structure $\mathfrak{c}$ on $M$, we are  thus led to consider, for each  Riemannian metric $g$,
 the Seiberg-Witten {\em moduli space} 
 \begin{equation}
\label{moduli}
 \mathfrak{M}_\mathfrak{c} (g) = \{ \mbox{solutions of  (\ref{drc}--\ref{sd})}\}/ \mathscr{G}, 
\end{equation}
and, as we will see,  this moduli space can always  be shown to be compact. 
However, $ \mathfrak{M}_\mathfrak{c} (g)$ 
is not  necessarily a manifold, and overcoming this obstacle is the key  to defining Witten's invariant. 

To this end, we now 
 generalize the Seiberg-Witten equations by replacing  \eqref{sd} 
 with the ``perturbed'' equation 
\begin{equation}
\label{ptsd}
iF_\theta^++\sigma(\Phi) =\eta
\end{equation}
for some self-dual $2$-form $\eta\in \Gamma (\Lambda^+)$.
If the harmonic part $\eta_H$ of $\eta$ satisfies 
\begin{equation}
\label{no-harm}
\eta_H \neq 2\pi [c_1(L)]^+,
\end{equation}
where $[c_1(L)]^+$ is the image of $c_1(L)$ under the projection $H^2(M, \RR)\to {\mathcal H}^+_{g}$
defined by \eqref{harm}, then any solution of \eqref{drc} and \eqref{ptsd} is  irreducible, in the
sense that $\Phi\not\equiv 0$. The Smale-Sard theorem  \cite{smale-sard} then allows one to show that, 
 for a set of $\eta$ of the second Baire category,
the moduli space 
$$ \mathfrak{M}_\mathfrak{c} (g, \eta ) = \{ \mbox{solutions of  \eqref{drc} and \eqref{ptsd}}\}/ \mathscr{G}$$
is a compact manifold of 
dimension
\begin{equation}
\label{mdim}
\dim \mathfrak{M}_\mathfrak{c} (g, \eta ) = \frac{c_1^2(L) - (2\chi + 3\tau )(M)}{4},
\end{equation}
where, when  this ``expected dimension'' is negative, 
equation  \eqref{mdim}   is just understood to  mean    that $\mathfrak{M}_\mathfrak{c} (g, \eta ) = \varnothing$
for generic $\eta$.

To prove this claim, one first   imposes the 
harmless gauge-fixing condition 
 \begin{equation}
\label{gauge} 
d^*(\theta - \theta_0) =0 , 
\end{equation}
relative to some   arbitrarily-chosen reference connection $\theta_0$  on $L$. This then   cuts 
 down the action of the infinite-dimensional gauge group $\mathscr{G}$ to  that of the $1$-dimensional Lie group
$$\mathscr{G}_0 = \{ \mbox{harmonic  maps } {\zap f} : M\to S^1 \}= S^1 \rtimes H^1 (M, \ZZ ),$$
and  reduces the problem to understanding the fibers of  the ``monopole map''  
\begin{eqnarray}
L^2_{k} ({\mathbb V}_+)  \oplus  L^2_{k} (\Lambda^1)
&{\longrightarrow}& 
L^2_{k-1}({\mathbb V}_-) \oplus L^2_{k-1}(\Lambda^+)  \oplus L^2_{k-1}/\RR \label{raw}\\
(\Phi, ~\vartheta )\qquad  &\longmapsto& (D_{\theta_0 + i\vartheta} \Phi , ~iF^+_{\theta_0} - d^+\vartheta + \sigma (\Phi ) , ~d^*\vartheta )
\nonumber
\end{eqnarray}
for any sufficiently large $k$.  By ellipticity, this is a Fredholm map, and  the Index Theorem, applied to the linearization, 
tells us that the  Fredholm index of \eqref{raw} equals the right-hand side 
of  \eqref{mdim} plus one, where the $+1$ arises from our having modded out by the constant functions $\RR$
in the codomain. The  moduli space $\mathfrak{M}_\mathfrak{c} (g, \eta )$  then becomes the fiber  over 
 $(0,\eta , 0)$, modulo the action of  $\mathscr{G}_0$. 
Because the linearization of $\dir_\theta \oplus d^*$ at an irreducible solution  always maps surjectively  onto 
$L^2_{k-1}({\mathbb V}_-)   \oplus L^2_{k-1}/\RR$, 
 the Smale-Sard theorem \cite{smale-sard} then tells us that  $(0,\eta, 0)$ is a regular value 
for  generic $\eta$,
and  the corresponding fiber is therefore   a smooth manifold by the implicit function theorem. 
On the other hand, 
compactness follows from  the generalization 
\begin{equation}
\label{pwbk}
 0=	2\Delta |\Phi|^2 + 4|\nabla_{\theta}\Phi|^2 +s|\Phi|^2 + |\Phi|^4 - 8 \langle \eta , \sigma (\Phi )\rangle
\end{equation}
of \eqref{wnbk} arising from \eqref{drc} and \eqref{ptsd}, because this Weitzenb\"ock formula
implies that any irreducible solution satisfies the  point-wise  bound 
\begin{equation}
\label{squeeze}
 |\Phi |^2 \leq \max (2\sqrt{2}|\eta | - s)
\end{equation}
everywhere. Since the action of $\mathscr{G}_0$ also allows us to  assume that the harmonic part of $\vartheta$ is 
bounded,   boot-strapping then  shows that 
$(\Phi , \vartheta)$ belongs to a bounded subset of $L^2_{k+1} ({\mathbb V}_+)  \oplus  L^2_{k+1} (\Lambda^1)$
for any large $k$, and  the Rellich theorem then says that  its image in $L^2_{k} ({\mathbb V}_+)  \oplus  
L^2_{k} (\Lambda^1)$ is necessarily compact. Since \eqref{no-harm} again  implies that every solution is irreducible, 
 $\mathscr{G}_0$ acts
freely, and  $\mathfrak{M}_\mathfrak{c} (g, \eta )$
is therefore  a smooth compact manifold, whose  dimension is given by \eqref{mdim},  for ``most'' choices of $\eta$.  

To define the earliest version of the Seiberg-Witten invariant, we now restrict ourselves to the case where the ``expected dimension'' \eqref{mdim} 
of the moduli space is {\em zero}, and also assume that  $b_+(M) \geq 2$. Remarkably enough, the 
 expected dimension \eqref{mdim}  vanishes  if and only if 
 the spin$^c$ structure $\mathfrak{c}$ is the one  determined   by some  globally defined orientation-compatible 
almost-complex structure 
$J$ on  $M$. Because we have assumed that  $b_+(M) \geq 2$,  any two regular-value choices of $\eta$ satisfying  \eqref{no-harm}
can  be joined by a smooth path $\eta (t)$ satisfying \eqref{no-harm} for all $t$, and such a path 
 then has an arbitrarily  small deformation that is transverse to the monopole map \eqref{raw}. Taking the inverse image
and modding out by $\mathscr{G}_0$ then gives a $1$-dimensional cobordism between $0$-dimensional 
moduli spaces associated with our two points. We may therefore define a $\ZZ_2$-valued invariant  \cite{KM}  by simply
 setting $n_\mathfrak{c}(M)\in \ZZ_2$ equal to $\# \mathfrak{M}_\mathfrak{c} (g, \eta ) \bmod 2$ when $(0,\eta, 0)$
 a regular value of \eqref{raw}. This  invariant is then {\em metric-independent}, because one can more generally 
construct cobordisms of the moduli spaces by  considering paths
$(g(t), \eta (t))$ where  the metric also varies. When   $n_\mathfrak{c}(M)$ 
 is non-zero, it then  follows that the {\em unperturbed} Seiberg-Witten equations (\ref{drc}--\ref{sd}) must have a solution, relative to the
given spin$^c$ structure $\mathfrak{c}$,  for any metric $g$. Indeed, 
if there were no solution for a metric $g$ satisfying $[c_1(L)]^+\neq 0$,
then 
$\eta =0$ would satisfy \eqref{no-harm}, and the absence of solutions would then make $(0,0,0)$ 
  a  {\em regular value} of the monopole map \eqref{raw}; thus, the count of solutions would then 
say  that  $n_\mathfrak{c}(M)$ vanished, thereby  contradicting our hypothesis. 
On the other hand, when  $[c_1(L)]^+= 0$ with respect to a given  metric $g$, 
we can still  produce a {\em reducible} solution of the equations
by just setting $\Phi\equiv 0$ and then  choosing $\theta$ so as to make  the $2$-form $F_\theta$  harmonic; thus, the assertion 
remains  true even in this ``degenerate'' case, albeit for essentially   trivial reasons.

Of course, the above discussion would have been a total waste of effort  in the absence of   examples where the invariant is non-zero. 
Fortunately, however, such examples do actually exist in profusion:

\begin{thm}[Witten {\em et al.}] \label{proton}
 If $(M^4,J)$ is a compact  complex surface of {\em K\"ahler type}  with $b_+ > 1$,  then $n_\mathfrak{c} (M) \neq 0$
for the  spin$^c$ structure  $\mathfrak{c}$ determined by the integrable almost-complex structure $J$. 
 \end{thm}

 Here a compact complex surface $(M^4,J)$ is said to be of {\em K\"ahler type} if it admits K\"ahler metrics. Remarkably, 
 this is equivalent  \cite{nick,siu}  to the topological condition that $b_1(M)$ is {\em even}, and this in turn is  equivalent \cite{bpv}
 to the topological condition that $b_+(M)$ is {\em odd}. Here the spin$^c$ structure determined by the integrable complex structure 
 $J$ is the one for which 
 $$\mathbb{V}_+= \Lambda^{0,0} \oplus \Lambda^{0,2}, \quad \mathbb{V}_- =\Lambda^{0,1}$$ 
  in a canonical manner for any metric $g$ which is {\em Hermitian},  
 in the sense that $g(J\cdot , J\cdot) = g$.  In the special case when the $g$ is also K\"ahler,
 which happens iff the associated $2$-form $\omega = g(J\cdot , \cdot )$ of $g$ 
   is closed, 
 then there is moreover a canonical choice of connection $\theta_0$  (called the {\em Chern connection}) on the anti-canonical line bundle $L=K^{-1}=\wedge^2 \mathbb{V}_\pm$
 such that  the associated spin$^c$ Dirac operator $\dir_{\theta_0}$ is just given by 
 $$\sqrt{2} (\bar{\partial} + \bar{\partial}^*): \Gamma (\Lambda^{0,0} \oplus \Lambda^{0,2} ) \to \Gamma (\Lambda^{0,1})$$ 
 where $\bar{\partial}$ is the standard  Dolbeault operator of several complex variables, and where $\bar{\partial}^*$ is its formal adjoint.
Moreover, the bundle of real self-dual $2$-forms relative to  $g$ is just 
$$\Lambda^+ = \RR \, \omega \oplus \Re e (\Lambda^{0,2})$$
while   $\sigma : \mathbb{V}_+\to \Lambda^+$ is exactly given by 
$$\sigma (f, \phi) =  (|f|^2 - |\phi|^2 ) \frac{\omega}{4} + \Im m (\bar{f}\phi),$$
where  an arbitrary section $\Phi$ of $\Lambda^{0,0} \oplus \Lambda^{0,2}$ has been 
expressed as a pair $(f,\phi )$ of a complex-valued function $f$ and a $(0,2)$-form $\phi$. 
On the other hand, the curvature $F$ of the Chern connection $\theta_0$ is exactly $-i\rho$, where $\rho = r(J\cdot, \cdot )$ is the Ricci form of $g$, and its self-dual part  is therefore 
 $$
 F^+= -i \frac{s}{4}\omega , 
 $$
where $s$ and $\omega$ are once again the scalar curvature and K\"ahler form of our K\"ahler metric $g$.
Thus, if $g$ is any K\"ahler metric on $(M,J)$, and if we choose our perturbation self-dual $2$-form  to be
$$\eta = \frac{1+s}{4}\omega ,$$
 then 
$$\Phi = (1,0) \in \Gamma (\Lambda^{0,0} \oplus \Lambda^{0,2})$$
and $\theta_0$ together solve the perturbed Seiberg-Witten equations \eqref{drc} and \eqref{ptsd} for this choice of $\eta$.
Careful use of the Weitzenb\"ock formula \eqref{lichne-c} then allows one to show \cite[Theorem 2]{spccs} 
that, for this choice of $\eta$, the solution we have just  exhibited  is in fact the {\em unique} solution 
of (\ref{drc},\ref{ptsd}), modulo  the action of the gauge group $\mathscr{G}$. After checking that  this  explicit solution
 also   represents a regular point of the monopole map,  Theorem \ref{proton} therefore follows. 

Let us now   review  some of the rudiments of complex-surface theory, where one  of the key  operations is that of  {\em blowing up}  \cite{bpv,GH}.
This replaces some  point $p\in N$ of a complex surface  with a $\CP_1$ of normal bundle $\mathcal{O}(-1)$, thereby resulting in 
a new smooth complex surface $M$; moreover, this ``blown up'' complex surface $M$ is diffeomorphic to  $N \# \overline{\CP}_2$, where 
$\overline{\CP}_2$ is the oriented manifold obtained from $\CP_2$ by reversing 
its standard orientation. Here it is important to observe that the ``tautological'' $\mathcal{O}(-1)$  line-bundle over
$\CP_1$ is defined as a subbundle of the rank-$2$ trivial bundle, and so comes equipped with a
holomorphic map to $\CC^2$ that  collapses the zero section to the origin, but is a biholomorphism elsewhere. 
This allows us to  define
blowing up as the procedure which  replaces  a complex-coordinate ball  in the manifold with  the inverse image of 
a  ball 
about $0$ in $\CC^2$ in the $\mathcal{O}(-1)$  line bundle. 
\begin{center}
\includegraphics[scale=.4]{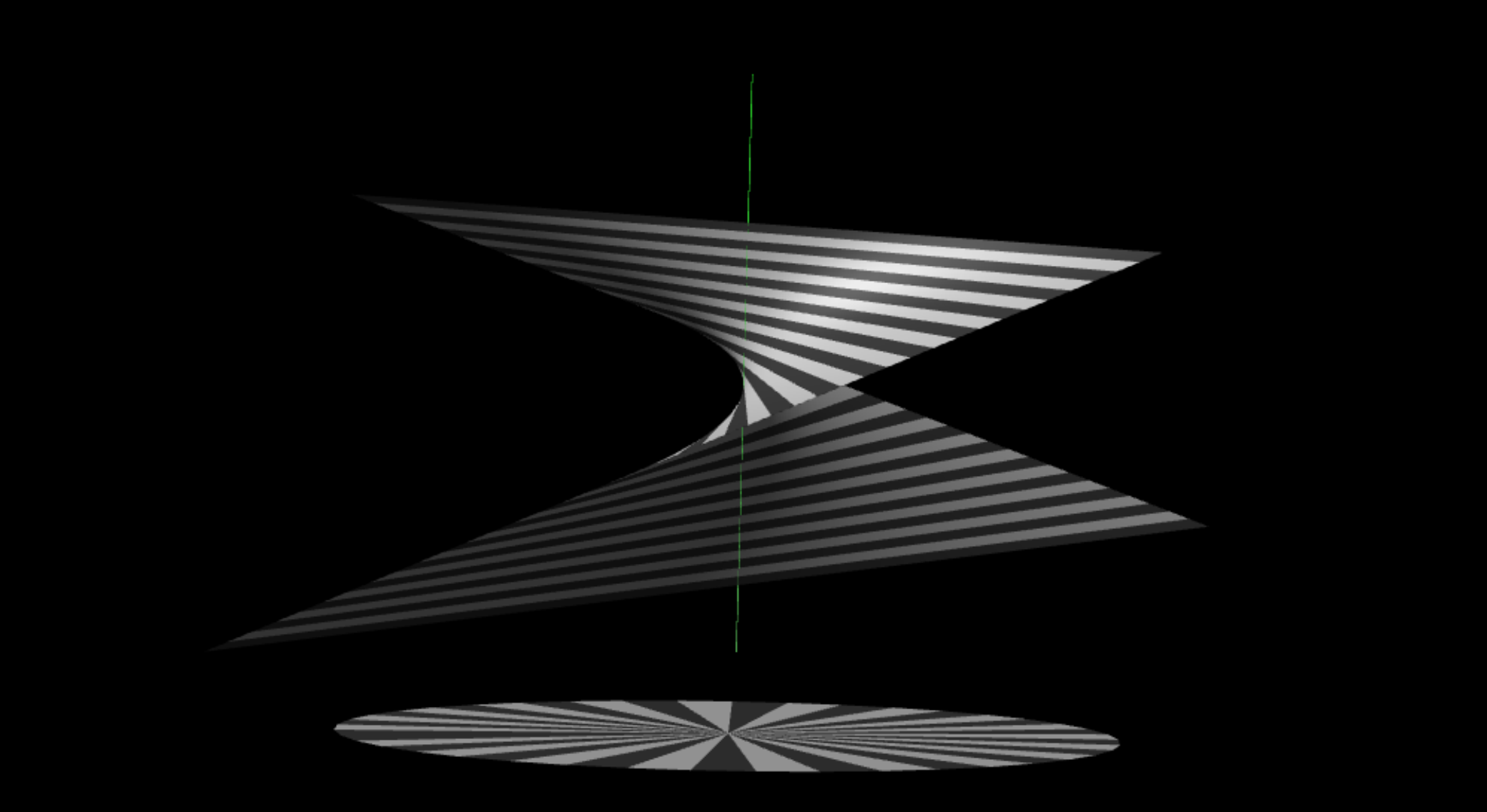}
\end{center} 


\noindent  Next    observe   that  that any choice of Hermitian inner product on the  line bundle $\mathcal{O}(-1) \to \CP_1$ 
defines an orientation-reversing diffeomorphism 
between the total space of  $\mathcal{O}(-1)$ and the total space of its dual line bundle $\mathcal{O}(1)$. Howeer,  the total space of  $\mathcal{O}(1)$
is biholomorphic to  $\CP_2-\{ pt\}$. These two observations therefore explain why, at  the level of 
smooth topology,  
 blowing up a point amounts to taking a connected sum with $\cpbar$.

 Any complex surface $M$ containing a $\CP_1$ of normal bundle $\mathcal{O}(-1)$
can conversely be ``blown down'' to produce a complex surface $N$, of which  $M$ then becomes  the one-point blow-up. 
This  operation of ``blowing down'' can in principle be iterated, but the process must terminate after finitely many steps, 
as each blow-down decreases $b_2$ by $1$. When  a complex surface
$X$ cannot be blown down any further, it is called {\em minimal}, and the upshot is that any
complex surface $M$ can be obtained from a {\em minimal} complex surface 
$X$ by blowing up finitely many times. In this situation,  we then say that $X$ is a {\em minimal model} of $M$. 

Since every complex surface has a minimal model, Kodaira's strategy for  classifying compact complex surfaces was therefore 
to proceed by  classifying  the minimal ones. In his classification scheme,  a single  complex-analytic 
 invariant,  nowadays called the {\em Kodaira dimension}  \cite{bpv,GH}, plays   the dominant  role. 
This invariant measures the growth-rate of the space of sections  of positive powers of the canonical line bundle $K:=\Lambda^{2,0}$, 
and   is defined to be  
$$\kod (M, J) = \limsup_{j \to +\infty}  \frac{\log \dim H^0 (M,\mathcal{O}(K^{\otimes j}))}{\log j}.$$
The only possible values of $\kod (M, J)$ are $-\infty$, $0$, $1$, and $2$, because 
the Kodaira dimension is actually  just the 
 largest complex dimension
of the image of $M\dasharrow \mathbb{P} [ H^0 (M,\mathcal{O}(K^{\otimes j}))]^*$ among all the various  ``pluricanonical''  maps 
associated with the  line bundles $K^{\otimes j}$, $j \in \ZZ^+$. (For these purposes,  we however impose the unusual  convention that  $\dim \varnothing :=-\infty$.)
Blowing up or down   leaves the Kodaira dimension unchanged, and the minimal model of a complex surface is moreover 
{\em unique} whenever  $\kod \neq -\infty$.
Finally, a compact complex surface  is  said to be 
of {\em general type} iff it has   $\kod = 2$.

 \begin{prop}
 Let $(M^4,J)$ be a compact complex surface that satisfies the strict Hitchin-Thorpe inequality 
 $$(2\chi + 3 \tau )(M) > 0.$$
Then $(M, J)$ is of K\"ahler type, and hence has $b_+ (M)$ odd. If $b_+ (M) > 1$, moreover, then the complex surface 
$(M^4,J)$ is  of general type. 
 \end{prop}

\label{hypersurface}

As an illustration of these ideas, consider the complex  surfaces
$$X_\ell := \{ [t:u:v:w]\in  \CP_3~|~ t^\ell + u^\ell + v^\ell + w^\ell =0\} . $$ 
Each is  the zero set of a holomorphic  section of a positive line bundle, namely  $\mathcal{O}(\ell)$,  that is transverse to the zero section.
The Lefschetz Hyperplane-Section Theorem \cite{GH} therefore implies that 
 every $X_\ell$ is simply connected. 
When $\ell\neq 3$, these complex surface $X_\ell$ are all {\em minimal};  however, by contrast, the cubic surface $X_3$  is {\em non-minimal},
and indeed  turns out \cite{GH} to be   biholomorphic to a blow-up of $\CP_2$ at six  distinct points. 
The  canonical line bundle $K:=\Lambda^{2,0}$
of $X_\ell$ is exactly  the restriction of the $\mathcal{O} (\ell - 4)$ line bundle from $\CP_3$, and, because  this is a positive line bundle
when $\ell \geq 5$, and trivial when $\ell =4$,  it is actually easy to  prove  that $X_\ell$ is 
minimal when $\ell \geq 4$. The fact that $K=\mathcal{O}(\ell - 4)$ for these examples also makes it easy  to show  that 
$$
\kod (X_\ell )  = \left\{
\begin{array}{ll} -\infty  &\mbox{if  }  \ell \leq 3\\
~ 0 &\mbox{if } \ell = 4  \\
 ~  2 &\mbox{if } \ell \geq 5.
\end{array}\right.
$$ 
In particular,  $X_\ell$ is of {\em general type} iff $\ell \geq 5$. On the other hand,  $M_4$ is just 
our familiar model for  $K3$.
One can also show that  that  $X_\ell$ is spin iff $\ell$ is even, and that 
$$b_+(X_\ell ) = \frac{(\ell -1) (\ell -2) (\ell -3) }{3} +  1 , \quad b_-(X_\ell) = b_+(X_\ell )+ \frac{(\ell+2) \ell (\ell-2)}{3},$$
and these facts   then allows one to exactly determine the homeomorphism type of $X_\ell$ by applying  Theorem \ref{fdmn}.

Next,  notice that $b_+(X_\ell ) > 1$ iff $\ell \geq 4$. In this range,   Theorem \ref{proton} then tells us that 
$n_{\mathfrak{c}}\neq 0$ for the spin$^c$ structure induced by the complex structure.
When $\ell \geq 4$, it therefore follows that  that the smooth $4$-manifold $X_\ell$ cannot admit a metric 
of positive scalar curvature, since the pointwise estimate \eqref{squeeze} would otherwise forbid the existence  of 
solutions of the perturbed Seiberg-Witten equations (\ref{drc},\ref{ptsd})  for any metric $g$ with $s>0$ and any sufficiently small perturbation $\eta$. 
 When $\ell \geq 4$ is {\em even}, this is not really new, because the spin manifold $X_\ell$ then has $\tau < 0$, and one can apply
the Lichnerowicz argument; but, more surprisingly, we also obtain,  for every {\em odd}  $\ell \geq 5$,  a simply-connected non-spin manifold that does not
admit metrics of positive scalar curvature. In particular, while these simply-connected non-spin manifolds are {\em homeomorphic} to connected sums of  copies of $\CP_2$ 
and $\cpbar$ by Theorem \ref{fdmn}, they are certainly not {\em diffeomorphic} to such connected sums, because a direct surgery argument \cite{gvln} shows that 
the connected sum of any finite collection of positive-scalar-curvature
$4$-manifolds also carries  positive-scalar-curvature metrics. On the other hand, our requirement  that  $\ell \geq 4$ in this discussion  is sharp, because $X_1$, $X_2$ and $X_3$ do
actually 
admit metrics with $s > 0$, because they are    respectively diffeomorphic to $\CP_2$, $S^2 \times S^2$, and 
$\CP_2\# 6 \cpbar$.

 All of these manifolds $X_\ell$ actually  admit Einstein metrics. Indeed, each of them admits a K\"ahler-Einstein metric\footnote{i.e. an Einstein metric that 
 happens to also be K\"ahler.} compatible with the given 
 complex structure. When $\ell \geq  5$,  this follows from the 
 Aubin-Yau theorem \cite{aubin,yau}, which produces K\"ahler-Einstein metrics with $\lambda < 0$ on these complex manifolds, because their canonical line bundles
 $K$ are all ample. Next,  the $K3$ surface $M_4$  is a compact complex manifold of K\"ahler type
 that has  trivial cononical line bundle, so Yau's theorem \cite{yauma} guarantees that it admits a $\lambda =0$  K\"ahler-Einstein metric
 in every K\"ahler class. Finally, when $1\leq \ell \leq 3$,  the $X_\ell$ admit $\lambda > 0$ K\"ahler-Einstein  metrics;  for  $\ell = 1, 2$,
 these are just the familiar Fubini-Study metric on $\CP_2$ and  product Einstein metric on $S^2 \times S^2$, while for  $\ell = 3$,
 the existence of a K\"ahler-Einstein metric on our so-called  Fermat cubic surface was first proved by Siu \cite{s}. (On an arbitrary   smooth
 cubic surface, an analogous K\"ahler-Einstein metric also exists, but the proof  of its existence  \cite{sunspot} came much  later, and is  far more delicate.) 
The fact that the Hitchin-Thorpe inequality
$$(2\chi+ 3\tau ) (X_\ell) = c_1^2 (X_\ell ) = \ell (\ell - 4)^2 \geq 0$$
holds for all of these examples is of course more elementary, and just   reflects  the fact that $c_1(X_\ell )$ ``has a sign''
for each $\ell$. As a consquence,   when $\ell$ is even,  Corollary \ref{biz} guarantees that the corresponding simply-connected spin  manifold
is always    homeomorphic to the connected sum of an appropriate number  of copies of $K3$ and $S^2 \times S^2$.

For $4$-manifolds with  $b_+(M)=1$, the above  examples   of $X_1$, $X_2$ and $X_3$ show  that Seiberg-Witten theory must somehow be modified in this setting. 
However, the required modification is surprisingly modest, even if  it does force  one 
to  pay
 careful attention to an  additional  subtlety. 
 The key observation is that  when $b_+(M) = 1$, the vector space 
 $\mathcal{H}^+_g$ is $1$-dimensional, and  so   removing a point from $\mathcal{H}^+_g$  disconnects 
it 
into  two open rays. 
Consequently, 
 for each spin$^c$ structure $\mathfrak{c}$, 
the set of pairs $(g, \eta)$, where $g$ is a Riemannian metric, and where $\eta$ is a self-dual $2$-form
satisfying \eqref{no-harm} with respect  to $g$,  consists of   exactly two connected 
 components, 
$\sphericalangle^+$ and $\sphericalangle^-$,  called {\sf chambers}. 
Applying the previous discussion to each chamber  then produces  two distinct invariants
$n_{\mathfrak{c}} (M, \sphericalangle^\pm)\in \ZZ_2$, which  are typically different; an expression for their  difference  is  called a {\sf wall-crossing formula}. 
 As long as $[c_1(L)]^+\neq 0$ for  a given  metric $g$ and spin$^c$ structure $\mathfrak{c}$, the same  arguments
used when $b_+(M)\geq 2$  will then guarantee the existence of an irreducible solution of the ``unperturbed'' Seiberg-Witten equations 
(\ref{drc}--\ref{sd}) whenever 
 $n_\mathfrak{c}\neq 0$   for the  chamber containing $\eta=0$.

Now  notice that if $b_+(M) =1$ and $b_-(M) \neq 0$, then $(H^2(M, \RR), Q )$ is
 a copy of $b_2(M)$-dimensional Minkowski space. The set of ``timelike'' cohomology classes $\alpha\in H^2(M, \RR)$ with
 $$\alpha^2:= Q( \alpha , \alpha ) > 0$$ is thus an open double cone consisting  of  two connected components, or {\em nappes}. 
 The choice of a ``time orientation'' for $(H^2(M, \RR), Q )$ then amounts to labeling one of these nappes, 
 henceforth denoted by  $\mathscr{C}^+$,  
 as the set of ``future-pointing'' time-like vectors, while  declaring that  the remaining  nappe  $\mathscr{C}^-$ 
consists of ``past-pointing'' time-like vectors.
The impact of the chamber-dependent  invariants on the unperturbed Seiberg-Witten equations (\ref{drc}--\ref{sd})
is  then entirely  governed by   the image
$[c_1(L)]^+\in \mathcal{H}_g$ of the first Chern class 
under Minkowski-space-orthogonal projection   to  $\mathcal{H}^+_g\subset H^2(M,\RR)$. 
Since $\mathcal{H}_g^+-\{ 0\}\subset \mathscr{C}^+ \cup \mathscr{C}^-$ for any Riemannian metric $g$, 
 the ``no reducible solutions'' condition $[c_1(L)]^+\neq 0$ implies 
  that either $[c_1(L)]^+\in \mathscr{C}^-$ or $[c_1(L)]^+\in \mathscr{C}^+$.  In the first case, the self-dual $2$-form
$\eta =0$ belongs to the chamber $\sphericalangle^+$ 
containing the open-ray component of $\mathcal{H}_g^+-\{ 0\}$ that terminates in $\mathscr{C}^+$,
while in the second case $\eta =0$ belongs to the chamber $\sphericalangle^-$
containing the open-ray component of $\mathcal{H}_g^+-\{ 0\}$ that terminates
 in $\mathscr{C}^-$. Thus, in order to use the invariant $n_\mathfrak{c}(M, \sphericalangle^\pm)$  to predict the the existence of solutions to
 (\ref{drc}--\ref{sd}), one must simply  keep track of whether $[c_1(L)]^+$ is future-pointing or past-pointing for a given metric $g$.
 
 But if we are just trying to determine whether  $M$ admits  Einstein metrics, we may also  assume, without loss of generality,  that  $M$ satisfies the Hitchin-Thorpe inequality \eqref{htineq}.
 However, the above discussion  also assumed that the {\em expected dimension} \eqref{mdim} of the Seiberg-Witten moduli space {\em is zero},  
 or equivalently that 
 $$c_1^2 (M) = (2\chi + 3\tau )(M).$$
  In more geometrical terms, this  exactly  means  that the given spin$^c$ structure $\mathfrak{c}$ arises from an almost-complex structure $J$ on $M$.
 The Hitchin-Thorpe inequality \eqref{htineq} thus   tells us that 
 $$c_1^2 (L)\geq 0,$$
so that $c_1(L)$ is then a time-like or null vector in $(H^2(M, \RR), Q)$. If $c_1(L)$ is not a torsion class,  
it therefore follows that $[c_1(L)]^+$ belongs to the same 
\begin{center}
\includegraphics[scale=.9]{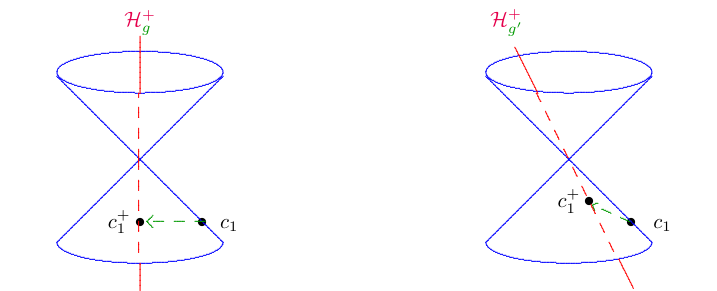}
\end{center} 

\noindent 
nappe,
 say  $\mathscr{C}^-$,  for all metrics $g$, so that if  $n_{\mathfrak{c}}(M, \sphericalangle^+)\neq 0$ for the 
future  chamber $\sphericalangle^+$, one can then conclude  that the (unperturbed)  Seiberg-Witten equations (\ref{drc}--\ref{sd}) must have a solution
for every Riemannian metric $g$ on $M$, and for the fixed spin$^c$ structure $\mathfrak{c}$. On the other hand, 
 there are many natural geometrical problems that venture  far enough outside  the  realm of  Einstein metrics that they require  one to 
consider Seiberg-Witten theory on $4$-manifold with $b^+=1$ but for spin$^c$ structures with   $c_1^2(L) < 0$. 
In these contexts, the answer to  whether $[c_1(L)]^+$ is past-pointing or future-pointing  
  depends on the metric.
  \begin{figure}[htb]
  \centerline{\includegraphics[scale=.9]{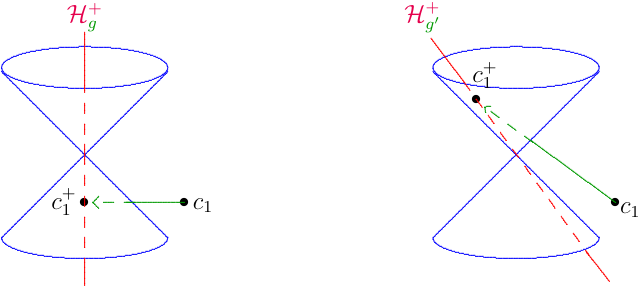} }
\end{figure}
 Fortunately, this technical difficulty  
  can  often  be overcome    \cite{FM,lno} by carefully playing  several spin$^c$ structures off against one another.
  
  \bigskip
  
  In order to  better understand  the chambered aspect of our story, let us now point out some of the most salient ramifications in the case of 
  compact complex surfaces: 

\begin{prop}\label{subtlety}
Let $(M,J)$ be a compact complex surface that  satisfies the strict Hitchin-Thorpe inequality
$(2\chi + 3\tau)(M) > 0$, but has  $b_+(M) =1$. Then either $\kod (M, J) = -\infty$ or $\kod (M, J) = 2$.
Moreover,   $(M,J)$ is of K\"ahler type, so we may therefore   define 
a time orientation on $(H^2_{dR}(M, \RR),Q)$ 
  by declaring any K\"ahler class $[\omega]\in H^2_{dR}(M, \RR)$ to be future-pointing.

  \bigskip
  
   \noindent 
 If $\kod (M, J) = 2$, so that $(M,J)$ is of general type, then $c_1^+$ is past-pointing for any metric $g$ on $M$ and the spin$^c$ structure determined 
 by $J$. This implies that 
 $\eta =0$ belongs to the chamber $\sphericalangle^+$ containing all future-pointing perturbations, and $n_{\mathfrak{c}}(M, \sphericalangle^+)\neq 0$
 for this chamber.  For any smooth  metric $g$ on $M$ and the spin$^c$ structure determined by $J$, there is consequently a solution of the unperturbed Seiberg-Witten equations.

  \bigskip
  
  \noindent 
  By contrast, if $\kod (M, J) = -\infty$, then $(M,J)$ is obtained from $\CP_2$ by a sequence of blow-ups and blow-downs, and $M$ therefore admits Riemannian metrics
  of positive scalar curvature. In this case,  $\eta =0$ always belongs to the chamber $\sphericalangle^-$ containing the  past-pointing perturbations, 
  and $n_{\mathfrak{c}}(M, \sphericalangle^-)= 0$ for this chamber. Nonetheless,   $n_{\mathfrak{c}}(M, \sphericalangle^+)\neq 0$ for the remaining chamber. 
  \end{prop}

 We now have   an effective criterion for predicting  the existence of solutions of the unperturbed Seiberg-Witten equations
  on many  complex surfaces. But, independent of {\em why} they might exist, the mere existence of  solutions of (\ref{drc},\ref{sd}) has 
   concrete,  remarkable geometric consequences \cite{lno,lric}.

\begin{prop}\label{best}
Let $(M^4,g)$ be a smooth compact oriented Riemannian manifold,
let $\mathfrak c$ be a spin$^c$ structure on $M$, and let 
$c_1^+= [c_1(L)]^+$ denote the self-dual part of the harmonic $2$-form 
representing the first Chern class $c_1(L)$ of $\mathfrak c$. 
If   there is a  solution of the unperturbed Seiberg-Witten equations (\ref{drc}--\ref{sd}) 
on $M$ for $g$ and  ${\mathfrak c}$, then the scalar curvature 
$s$ of $g$ satisfies 
$$
\int_{M}s^{2}d\mu_{g} \geq \int_{M}s_-^{2}d\mu_{g} \geq 32\pi^{2} [c_{1}^{+}]^{2} 
$$
where $s_-:=\min (s,0)$ at each point. Moreover, when $c_1^+\neq 0$,  equality can only  occur if  $g$
is a K\"ahler metric of constant {negative} scalar curvature whose 
 complex structure $J$ induces the spin$^c$ structure $\mathfrak{c}$ on $M$. 
\end{prop}
\begin{proof}
Integrating the Weitzenb\"ock formula (\ref{wnbk}), we have
$$0= \int [ 4|\nabla_\theta \Phi |^2 + s|\Phi|^2 + |\Phi|^4 ] d\mu , $$
and it follows that 
$$\int (-s_-) |\Phi|^2 d\mu \geq \int |\Phi|^4 d\mu .$$
Applying the Cauchy-Schwarz inequality to the
left-hand side therefore yields  
$$
\left(\int s_-^2 ~d\mu  \right)^{1/2}\left(\int |\Phi |^4 d\mu  \right)^{1/2} \geq \int |\Phi |^4 d\mu .
$$
We therefore have  
$$
\int s^2 d\mu \geq \int s_-^2 ~d\mu \geq  \int |\Phi |^4 d\mu  = 8 \int |F_{\theta}^{+}|^2 d\mu ,
$$ 
and   the inequality is strict unless  $\nabla_\theta \Phi \equiv 0$ and $s$ is a non-positive constant. 
However, $iF_{\theta}^{+}-2\pi c_1^+$  is an exact form plus a co-exact form, and 
so is $L^2$-orthogonal to the harmonic  forms. This gives us
the inequality 
$$\int |F_{\theta}^{+}|^2 d\mu \geq 4\pi^2 \int |c_1^+|^2 d\mu = 4\pi^2 \int c_1^+\wedge c_1^+,$$
and the last expression may be re-interpreted as the intersection pairing 
$[c_1^+]^2$ 
of 
the de Rham class  $c_1^+$ with itself. This gives us the desired
inequalities
$$\int s^2 d\mu \geq \int s_-^2 ~d\mu \geq  32\pi^2 [c_1^+]^2,$$
and, when the right-hand side is non-zero, equality can only happen if 
$\sigma (\Phi )$  is parallel
and $g$  has constant negative scalar curvature.  
\end{proof}

 With this in mind, let us take a fresh look at the borderline  $c_1^2 (M , J)=0$ case of the Hitchin-Thorpe inequality \eqref{htineq}.
 Here it will be useful to introduce to the so-called {\em Kummer model} of $K3$, which is obtained from the product
 $T^4=T^2 \times T^2$ of two {\em elliptic curves} (genus-one Riemann surfaces)  by first remembering that each factor can be obtained from 
 $\CP_1$ as a ramified   double cover with four branch-points. Thus, each $T^2$ factor  admits a $\ZZ_2$-action with 
 four fixed points, and applying this {\em Weierstrass involution} to both factors simultaneously then gives us a holomorphic 
 involution of the product $T^4$ with $16$ fixed points: 
 
\begin{center}
\mbox{
\beginpicture
\setplotarea x from 0 to 200, y from -5 to 160
\putrectangle corners at 80 160 and 160 80
\put {$_{\times}$} [B1] at 80 105
\put {$_{\times}$} [B1] at 105 105
\put {$_{\times}$} [B1] at 135 105
\put {$_{\times}$} [B1] at 160 105
\put {$_{\times}$} [B1] at 80 80
\put {$_{\times}$} [B1] at 105 80
\put {$_{\times}$} [B1] at 135 80
\put {$_{\times}$} [B1] at 160 80
\put {$_{\times}$} [B1] at 80 135
\put {$_{\times}$} [B1] at 105 135
\put {$_{\times}$} [B1] at 135 135
\put {$_{\times}$} [B1] at 160 135
\put {$_{\times}$} [B1] at 80 160
\put {$_{\times}$} [B1] at 105 160
\put {$_{\times}$} [B1] at 135 160
\put {$_{\times}$} [B1] at 160 160
\put {$T^2$} [B1] at 160 10 
\put {$T^2$} [B1] at  10 70
\put {$T^4$} [B1] at  175  120
\ellipticalarc axes ratio 5:2  360 degrees from 150 40
center at 120 30
\ellipticalarc axes ratio 4:1 -180 degrees from 135 33
center at 120 33
\ellipticalarc axes ratio 4:1 145 degrees from 130 30
center at 120 29
\ellipticalarc axes ratio 2:5  -360 degrees from 40 150 
center at 30 120 
\ellipticalarc axes ratio 1:4  180 degrees from  33 135
center at  33 120
\ellipticalarc axes ratio 1:4 -145 degrees from  30 130
center at  29 120
\ellipticalarc axes ratio 1:2 140 degrees from 70 35
center at 70 30
\ellipticalarc axes ratio 1:2 140 degrees from 70 25
center at 70 30
\ellipticalarc axes ratio 2:1 140 degrees from 35 68
center at 30 68
\ellipticalarc axes ratio 2:1 140 degrees from 25 68
center at 30 68
\arrow <3pt> [.1,.3] from 70 35 to 71 35
\arrow <3pt> [.1,.3] from 70 25 to 69 25
\arrow <3pt> [.1,.3] from 35 68 to 35 67
\arrow <3pt> [.1,.3] from 25 68 to 25 69
{\setlinear
\setdashes 
\plot   30 62  30 82  /
\plot  60 30 83 30   /
\plot  160 30 180 30   /
\plot   30 160  30  180  /
}
\endpicture
}
\end{center}
The quotient $T^4/\ZZ_2$ is an orbifold complex surface with $16$ singular points, with $\CC^2 /\{ \pm 1\}$ serving as a local model near each singularity.\footnote{In 1865, Kummer \cite{dolgachev,kummer} 
showed that if we  take our $T^4$ to be the Jacobi variety of any complex curve of genus two, the resulting complex orbifold is isomorphic to a quartic hypersurface in $\CP_3$ with $16$  nodal 
singularities. This fact  begins to explain why its desingularization is diffeomorphic to our non-singular quartic model $X_4$ of $K3$.}
However, we can avoid  introducing  these singularities by first blowing up the $16$ fixed points before modding out by $\ZZ_2$.
The exceptional divisors then map to $16$ copies of $\CP_1$, each with normal bundle $\mathcal{O}(-2)$. The result is then a compact complex surface with $c_1=0$,
and hence a $K3$ surface by 
Andr\'e  Weil's   definition. The map from this $K3=\widetilde{T^4/{\ZZ_2}}$ to $\CP_1=T^2/\ZZ_2$ induced by the first-factor projection also now gives us an ``elliptic fibration'' for which the regular
  fibers are  all isomorphic  to $T^2$,
but for which there are also
four singular fibers.

\begin{center}
\mbox{
\beginpicture
\setplotarea x from 0 to 200, y from -5 to 160
\putrectangle corners at 80 160 and 160 80
\put {$_{\times}$} [B1] at 80 105
\put {$_{\times}$} [B1] at 105 105
\put {$_{\times}$} [B1] at 135 105
\put {$_{\times}$} [B1] at 160 105
\put {$_{\times}$} [B1] at 80 80
\put {$_{\times}$} [B1] at 105 80
\put {$_{\times}$} [B1] at 135 80
\put {$_{\times}$} [B1] at 160 80
\put {$_{\times}$} [B1] at 80 135
\put {$_{\times}$} [B1] at 105 135
\put {$_{\times}$} [B1] at 135 135
\put {$_{\times}$} [B1] at 160 135
\put {$_{\times}$} [B1] at 80 160
\put {$_{\times}$} [B1] at 105 160
\put {$_{\times}$} [B1] at 135 160
\put {$_{\times}$} [B1] at 160 160
\put {$\circ$} [B1] at 120 33 
\put {$\bullet$} [B1] at 158 27
\put {$\bullet$} [B1] at 82 27
\put {$\bullet$} [B1] at 107 27
\put {$\bullet$} [B1] at 135 27
\put {${\CP_1}$} [B1] at 185 25
\put {$T^2$} [B1] at  10 70
\put {$\widetilde{T^4/{\ZZ_2}}$} [B1] at  185  120
\ellipticalarc axes ratio 5:2  360 degrees from 150 40
center at 120 30
\ellipticalarc axes ratio 2:5  -360 degrees from 40 150 
center at 30 120 
\ellipticalarc axes ratio 1:4  180 degrees from  33 135
center at  33 120
\ellipticalarc axes ratio 1:4 -145 degrees from  30 130
center at  29 120
\arrow <2pt> [1,3] from 120 75 to 120 50
{\setlinear 
\plot   120 80 120 160   /
}
\endpicture
}
\end{center}

 Now, we have already noted that  $K3$ admits Ricci-flat K\"ahler metrics, and that  $\kod (K3) =0$.  Moreover, $n_{\mathfrak{c}}(K3)\neq 0$
 for the standard spin$^c$ structure on $K3$ with $c_1(L)=0$. But we  also  have $n_{\mathfrak{c}^\prime}(K3)= 0$ for any other spin$^c$
 structure $\mathfrak{c}^\prime\neq \mathfrak{c}$ arising from an almost-complex structure $J^\prime$. Indeed, since $(2\chi +3\tau )(K3)=0$, and because $H^2(K3,\ZZ)$
 is torsion-free, the corresponding 
 Chern class 
 $c_1(L^\prime)$ of any such $\mathfrak{c}^\prime$  would be a non-zero null vector in $H^2(M,\RR)$,
 and would therefore satisfy $c_1^+ =[c_1(L^\prime)]^+\neq 0$ for any Riemannian metric $g$ on $K3$. If the Seiberg-Witten invariant
 were non-zero for any such spin$^c$ structure, Proposition \ref{best} would therefore imply that $K3$ could not admit a Riemannian 
 metric with $s\equiv 0$. But this  contradicts the existence of Ricci-flat metrics on $K3$, which of course {\em   do} 
 have vanishing scalar curvature! 
 
 \label{xoK3}
 However, Kodaira \cite{homK3} discovered a surprising sequence $\{ Y_q\}_{q=1}^\infty$ of other complex surfaces that are homotopy-equivalent to  $K3$,
 and thus, by Theorem \ref{fdmn}, are actually homeomorphic to it. To obtain these, we first choose some regular $T^2$-fiber $F$ of
the {elliptic fibration} $K3\to \CP_1$ constructed above, and then notice that our map  has the nice property that $F$ 
has a tubular neighborhood that is biholomorphic to $F \times D$, where $D \subset \CC$ is the unit $2$-disk.  On the other hand, 
$F$ is  itself biholomorphic to $F=\CC/\Lambda$ for a lattice generated by $1$ and some vector $\mathbf{v}$ in the upper half-plane. 
Finally,   after choosing an odd positive integer $2q+1$, we will replace  our chosen  fiber $F$ with the new complex curve $\mathfrak{f}= F/\ZZ_{2q+1}$.
We do this using a trick due to Kodaira called a {\em logarithmic transformation}. Namely, we consider the $\ZZ_{2q+1}$ action  on $F\times D =(\CC/\Lambda)\times D$ 
generated by 
 $$([z],\zeta ) \mapsto \left(\left[z+\frac{1}{2q+1}\right], e^{2\pi i/(2q+1)}\zeta \right)$$
 and then notice that $[F\times (D-\{0\})]/\ZZ_{2q+1}$ is biholomorphic to $F\times (D-\{ 0\})$ via  the map induced by 
 $$[( z, \zeta )]\longmapsto \left( \left[ z - \frac{1}{2\pi i}\log \zeta \right] ,\zeta^{2q+1}\right).$$
This allows us  delete our fiber $F$ from $K3$ and glue in a copy of $(F\times D)/\ZZ_{2q+1}$
along $F\times (D-\{0\})$, 
thereby replacing $F$ with $\mathfrak{f}= F/\ZZ_{2q+1}$.
The new complex surface $Y_q$ this constructs   is simply connected, and has 
$$c_1(Y_q) = -2q \mathfrak{f}.$$  
This implies  that $Y_q$ is spin, with $\chi =24$ and $2\chi +3\tau =0$, so that $Y_q$ is therefore  homeomorphic to 
$K3 = Y_0$ by Theorem \ref{fdmn}. However, for $q > 0$,  one has $n_{\mathfrak{c}^\prime}(Y_q) \neq 0$ for a spin$^c$ structure 
for which $c_1(L^\prime )=-2q \mathfrak{f}\neq 0$, but for which  $c_1(L^\prime)$ is also divisible by $q$ in $H^2 (Y_q, \ZZ)$. 
By varying  $q$, it follows that infinitely many different elements of $H^2(K3, \ZZ)$ that can be realized as the first Chern classes 
of ``exotic'' integrable complex structures, each of which is compatible with some ``exotic'' smooth structure on our fixed topological $4$-manifold $K3$. 
 But since
$[c_1^+]^2 = |[c_1^-]^2|$ for all these spin$^c$ structures, the bound given by 
Proposition \ref{best} implies that  only  finitely many such spin$^c$ structures could have $n_{\mathfrak{c}^\prime} \neq 0$ 
for any specific  smooth structure, 
or even  for any finite collection of smooth structures. This implies  that there must be an infinite number of
different smooth structures on $K3$, even among the sequence of examples   $\{ Y_q \}$ discovered by Kodaira. 
But also notice that Theorem \ref{ht} guarantees that no such ``exotic'' $K3$ can admit an Einstein metric.
It is also worth pointing out   that, while $\kod (K3)=0$,  we instead have $\kod (Y_q) =1$ for any $q\geq 1$. Indeed, the Kodaira classification of 
complex surfaces \cite{bpv,GH} can easily be used to show   that 
 $\kod =1$ for any complex surface that realizes  an   ``exotic'' differentiable structure on $K3$.  

\bigskip
This starts to give us  a reasonably good     understanding of Einstein metrics on complex surfaces of Kodaira dimension $0$ and $1$, especially 
in the simply-connected case. Let us therefore now turn the problem of understanding 
general  Einstein metrics on  complex surfaces of Kodaira dimension  $2$. 

\begin{cor} Let $M$ be the underlying  smooth compact $4$-manifold of a compact complex surface that admits a K\"ahler-Einstein metric $g$
with $\lambda < 0$. Then $g$ minimizes $\int s^2 d\mu$ among all Riemannian metrics on $M$. Moreover, any other minimizer is also 
a K\"ahler-Einstein metric with $\lambda < 0$. 
\end{cor}
\begin{proof} Our assumption implies that $c_1(M , J) < 0$, where $J$  is a   complex structure compatible with $g$. 
Since $(M,J)$ therefore has ample canonical line bundle $K$, it then  follows that  $(M,J)$ is of general type. Thus, even if $b_+(M)=1$, there is a solution, 
for the spin$^c$ structure determined by $J$, 
of the unperturbed Seiberg-Witten equations with respect to any other metric $g^{\prime}$ on $M$. But for any such metric, 
one has $$[c_1^+]^2 = c_1^2(M)+ |[c_1^-]^2|$$
where $c_1^-$ is the anti-self-dual part of the harmonic representative  of $c_1(M,J)$. 
Proposition \ref{best} therefore implies that any  metric $g^\prime$ on $M$ satisfies
$$\int s^2 d\mu_{g^\prime} \geq  32\pi^2 c_1^2 (M)$$
with equality iff $g^\prime$ is a K\"ahler metric of constant scalar curvature $s < 0$ for which the harmonic representative
of $c_1(M)$ is self-dual. However, for any  K\"ahler metric of constant scalar curvature,  the Ricci form $\rho$ 
is harmonic, and so is exactly the harmonic representative of 
 $2\pi c_1 (M)$. 
 On the other hand,  the self-dual piece of $\rho$  is  $s\omega /4$ for any K\"ahler metric, where $s$ is the scalar curvature, and where
$\omega$ is the K\"ahler form. Thus, another K\"ahler metric  $g^\prime$ will saturate the  bound 
\begin{equation}
\label{button}
\int s^2 d\mu_{g^\prime} \geq  32\pi^2 c_1^2 (M),
\end{equation}
if and only if $g^\prime$ is a $\lambda < 0$ K\"ahler-Einstein metric, albeit perhaps adapted to quite a different complex structure 
 on our smooth $4$-manifold $M$. 
\end{proof} 

Theorem \ref{cohyp} is  now   a straightforward  consequence \cite{lmo}  of the following:

\begin{cor}\label{alibaba}
Let $(M,g)$ be a smooth compact Einstein $4$-manifold that carries a spin$^c$ structure of almost-complex type for which $n_{\mathfrak{c}} \neq 0$,
and assume that  $g$ is not flat. 
Then 
$M$ satisfies the generalized Miyaoka-Yau inequality 
\begin{equation}
\label{elmo}
\chi (M) \geq 3 \tau (M), 
\end{equation}
with equality iff the universal over $(\widetilde{M},\widetilde{g})$ of $(M,g)$ is isometric to a rescaled version  of the complex hyperbolic plane $\CC\mathcal{H}_2 = SU(2,1)/U(2)$. 
Moreover, when equality holds in \eqref{elmo}, the moduli space $\mathscr{E}(M)$ of Einstein metrics on $M$ consists of a single point, represented by $g$. 
\end{cor}
\begin{proof}
Combining equation  \eqref{gb-} and  inequality \eqref{button} shows that the given Einstein metric $g$ must satisfy 
\begin{eqnarray*}
3(2\chi - 3\tau ) (M)  &=& \frac{3}{4\pi^2} \int_M \left( \frac{s^2}{24} + 2|W_-|^2 \right) d\mu_g\\
 &=& \frac{1}{32\pi^2} \int_M s^2 d\mu_g +  \frac{3}{2\pi^2} \int_M |W_-|^2 d\mu_g \\ &\geq &c_1^2(M) = (2\chi + 3\tau)(M),
\end{eqnarray*}
and since our assumption that $g$ is non-flat guarantees that both sides are non-zero,  
 equality can therefore only occur if  $g$ is a $\lambda < 0$ K\"ahler-Einstein metric with  $W_-\equiv 0$. 
Inequality \eqref{elmo}  thus   follows by a routine   algebraic manipulation,  and equality can only occur  when the
 surviving pieces of the  curvature tensor \eqref{deco-riem}  are parallel. When equality holds, 
  $(M,g)$ therefore satisfies $\nabla \mathcal{R}=0$,  and so 
   is  a locally symmetric space \cite{bes}. 
The form of the curvature tensor $\mathcal{R}\not\equiv 0$ at one point moreover  suffices to force the  universal cover $(\widetilde{M},\widetilde{g})$  
to  be a constant rescaling of the complex hyperbolic plane $\CC\mathcal{H}_2$, and    Mostow rigidity \cite{bcg,mostow2} then guarantees
that $(M,g)$ is determined,  up to isometry and rescaling,  by its  fundamental group $\pi_1(M)$.
\end{proof}

Notice that the inequality $\chi \geq 3\tau$ proved by  Corollary \ref{alibaba} is markedly   stronger than the inequality $\chi \geq \frac{3}{2} \tau$ of Corollary \ref{mira}. 
For discussion of the Miyaoka-Yau inequality \eqref{elmo} in its original context, and  its importance in the geography of  compact complex surfaces, see \cite{bpv,bes,yauma}.

So far, we have just seen that the Seiberg-Witten equations are sensitive to, and  provide  information about, the scalar curvature. 
However, it turns out that they are also sensitive to the self-dual Weyl curvature. The key    result   in this direction is the following \cite{lric,lebsurv2,lebeta}:

 \begin{prop}\label{upscale}
 Let $(M^4,g)$ be a smooth compact oriented Riemannian manifold,
 and suppose, for a fixed spin$^c$ structure $\mathfrak{c}$,  that there is a solution of the unperturbed Seiberg-Witten equations  (\ref{drc}--\ref{sd})  for every 
 metric $\widehat{g}$ in the conformal class of  $g$.  
Then the scalar curvature 
$s$ and self-dual Weyl curvature $W_+$ of $g$ must satisfy
\begin{equation}
\label{voila}
\int_M \left(s-\sqrt{6}|W_+|\right)^{2}d\mu_g  \geq  72\pi^2 [c_1^+]^2 .
\end{equation}
 Moreover, if $c_1^+:= c_1(L)^+$ is non-zero, equality can only occur  if $g$ is an almost-K\"ahler metric for which $s-\sqrt{6}|W_+|$ is a  negative constant,  for which the 
 symplectic form everywhere belongs to the bottom eigenspace of $W_+$,  and for which the two highest eigenvalues of $W_+$ are everywhere equal. 
   \end{prop}
 \begin{proof} The Seiberg-Witten equations are certainly not conformally invariant. Instead, 
given a   smooth  function $f> 0$ on $M$, the fact that there is solution of
(\ref{drc},\ref{sd}) for $\widehat{g}=f^{-2}g$ can be rewritten with respect to $g$ as  
saying that there is a solution $(\Phi , \theta )$ of the the ``rescaled'' Seiberg-Witten equations 
\begin{eqnarray} \dir_{\theta}\Phi &=&0\label{rsdrc}\\
 F_{\theta}^+&=&if\sigma (\Phi) \label{rssd}\end{eqnarray}
 and our hypothesis thus insists that there is a solution of (\ref{rsdrc}--\ref{rssd}) 
 for any  choice of smooth function $f > 0$.
But applying our Weitzenb\"ock formula \eqref{wtw} to a solution   $({\Phi}, \theta)$ of (\ref{rsdrc}--\ref{rssd}) now yields 
 $$
0 = 2\Delta |\Phi |^{2} + 4 |\nabla_{A} \Phi |^{2} + s|\Phi |^{2}
 + f|\Phi |^{4} .
$$
Multiplying by $|\Phi |^2$ and integrating, we thus obtain 
the  inequality 
\begin{equation}
\label{two}
0\geq \int_M\left[4 |\Phi |^{2}|\nabla_{A} \Phi |^{2} + s|\Phi |^{4}
 + f|\Phi |^{6}\right] d\mu 
\end{equation} 
 Now set   $\psi = 2\sqrt{2} \sigma (\Phi )$, and 
 observe that   the definition of $\sigma$ then implies that 
$$
|\Phi |^{4} =	|\psi |^{2}   ,~~~~~~~~
	   4 |\Phi |^{2}|\nabla_\theta \Phi |^{2}  \geq  |\nabla \psi |^{2}  . 
$$
Thus, the  inequality  (\ref{two})
 tells us that 
 $$0\geq \int_M\left[|\nabla\psi  |^{2} + s|\psi |^{2}
 + f|\psi |^{3}\right] d\mu . $$
 However, the Weitzenb\"ock formula 
 \begin{equation}
\label{outil}
(d+d^\ast)^2 \psi =     \nabla^{*}\nabla \psi - 2W_{+}(\psi , 
\cdot ) + \frac{s}{3} \psi  
\end{equation}
for the Hodge Laplacian on self-dual $2$-forms, and the  fact that $(d+d^\ast)^2 $ is a non-negative operator, together tell us that 
\begin{equation}
\label{wiseguy}
\int_{M} |\nabla \psi |^{2}d\mu \geq 
\int_M\left(-2\sqrt{\frac{2}{3}}|W_+|-\frac{s}{3}\right)|\psi |^{2} d\mu 
\end{equation}
for any self-dual $2$-form $\psi$.
Combining these facts  therefore yields 
$$
0\geq \int_M\left[  \left(\frac{2}{3} s-2\sqrt{\frac{2}{3}}|W_+|\right) |\psi |^{2} 
 + f|\psi |^{3}\right] d\mu .
$$
Set $\varphi = \frac{3}{2}\psi = 3\sqrt{2}\sigma (\Phi )$. We then have
$$
0\geq \int_M\left[\left( s-\sqrt{6}|W_+|\right) |\varphi |^{2} 
 +   f|\varphi |^{3}\right] d\mu .
$$
Rewriting this as 
$$
\int_M\left[ -\left( s-\sqrt{6}|W_+|\right)f^{-2/3}\right] \left( f^{2/3}|\varphi  |^{2}\right) d\mu \geq 
\int_M f|\varphi  |^{3}  d\mu 
$$
and applying the H\"older inequality to the left-hand side then yields 
    $$
  \left[\int_M \left|s-\sqrt{6}|W_+|\right|^3f^{-2}d\mu\right]^{1/3}
  \left[ \int_M f|\varphi  |^{3}d\mu \right]^{2/3}
   d\mu \geq \int_M f|\varphi  |^{3}d\mu ,
   $$
  which is to say that 
    $$
  \int_M \left|s-\sqrt{6}|W_+|\right|^3f^{-2}d\mu   \geq 
   \int_M f|\varphi  |^{3}d\mu .
   $$
    But the H\"older inequality also tells us that 
    $$
   \left(\int_M f^4d\mu\right)^{1/3}   \left( 
   \int_M   f|\varphi  |^{3}d\mu\right)^{2/3} \geq 
  \int_M f^{4/3} \left[  f^{2/3} |\varphi  |^2\right] d\mu ~,    $$
   where equality holds only if $|\varphi |$ is a constant multiple of $f$. 
 Hence 
  $$
   \left(\int_M f^4d\mu\right)^{1/3}   \left( \int_M \left|s-\sqrt{6}|W_+|\right|^3
   f^{-2}d\mu\right)^{2/3} \geq 
  \int_M f^2 |\varphi |^2 d\mu ~.    $$
   But since $f\varphi = 3\sqrt{2} f\sigma (\Phi ) = 3\sqrt{2}(-iF_\theta^+)$,  we also  have 
   $$ \int_M f^2 |\varphi|^2 d\mu = 18 \int_M |F_\theta^+|^2 d\mu \geq 
  18 (2\pi c_1^+)^2 
  =
   72\pi^2 [c_1^+]^2 $$
   because  $iF_\theta\in 2\pi c_1^\RR (L)$. 
   Thus 
   \begin{equation}\label{clem} 
\left(\int_M f^4d\mu_g \right)^{1/3}\left( \int_{M}\left|s-\sqrt{6}|W_+|\right|^{3}f^{-2}d\mu_g 
 \right)^{2/3}\geq  72\pi^2 [c_1^+]^2 
\end{equation}
 for any smooth positive function $f$ on $M$.
 
 Now choose a sequence of smooth positive functions $f_j$ on 
 $M$ with 
 $$ f_j \searrow \sqrt{\left|  s-\sqrt{6}|W_+|   \right|}$$
uniformly on $M$.  Since the inequality $f_j^2 \geq {\left|  s-\sqrt{6}|W_+|   \right|}$
 implies 
$$
\int_M f_j^4 d\mu \geq  \left(\int_M f_j^4d\mu_g \right)^{1/3}\left( \int_{M}\left|s-\sqrt{6}|W_+|\right|^{3}f_j^{-2}d\mu_g 
 \right)^{2/3},  $$
we  therefore have 
$$
\int_M f_j^4 d\mu \geq  72\pi^2 [c_1^+]^2 
$$
 as a consequence of   (\ref{clem}). But since 
 $$
 \int_M \left(s-\sqrt{6}|W_+|\right)^{2}d\mu = \lim_{j\to \infty} 
 \int_M f_j^4 d\mu , 
 $$
this shows that 
 $$\int_M (s-\sqrt{6}|W_+|)^2d\mu_g \geq 72\pi^2[c_1^+]^2,$$ 
as  claimed. 
 
Finally, in the equality case,  $s-\sqrt{6}|W_+|$ is necessarily  constant, 
because  $g$ minimizes the $L^2$-norm of $\mathfrak{S}= s-\sqrt{6}|W_+|$  in the conformal class $[g]$, and
 this quantity  transforms under conformal changes  by a rule 
 $$\hat{g} = u^2 g \quad \Longrightarrow \quad \hat{\mathfrak{S}}u^3 = (6\Delta + \mathfrak{S})u$$ 
  that is  
exactly parallel to the Yamabe equation  the governs  the behavior  of the usual scalar curvature $s$ under conformal changes. 
Moreover, this constant must be non-zero if $c_1^+\neq 0$. Thus, if equality holds, 
one can just set  $f=\sqrt{\left|s-\sqrt{6}|W_+| \right| }$, and then notice that every inequality in the proof must be an equality 
for the given metric $g$. From this, one  deduces that  $\psi$ is a non-trivial harmonic self-dual $2$-form of constant length with respect to  $g$, making $g$ an  {\em almost-K\"ahler
metric}, in the sense that it is related to a symplectic form $\omega$ by $g=\omega (\cdot , J\cdot )$ for some almost-complex structure $J$. 
Moreover, because this  symplectic form $\omega$ is just  a constant multiple of $\psi$, where the self-dual harmonic $2$-form  $\psi$ saturates the inequality \eqref{wiseguy}, 
$\omega$ must everywhere belong to the bottom eigenspace of $W_+$, and   the top two eigenvalues of $W_+$ must  be equal everywhere. For further details, see
\cite{lric,lebsurv2,lebeta}.
 \end{proof}
 
The estimate \eqref{voila} proved in  Theorem \ref{upscale} is sharp, because    it is actually  saturated by any K\"ahler metric of constant negative scalar curvature. 
Nonetheless, the peculiar way that the scalar curvature and self-dual Weyl curvature are combined in \eqref{voila} is not really perfectly suited to controlling 
the sorts of curvature expressions that arise in Gauss-Bonnet-type  formulas like \eqref{cible}. 
To make headway, we will therefore trade \eqref{voila}   in for the  following, nominally weaker   consequence:

 \begin{cor} \label{servus}
 Let $(M,g)$ be a smooth compact oriented Riemannian manifold,
 and suppose, for a fixed spin$^c$ structure $\mathfrak{c}$,  that there is a solution of the unperturbed Seiberg-Witten equations  (\ref{drc}--\ref{sd})  for every 
 metric $\widehat{g}$ in the conformal class $[g]$ of  $g$.  
Then the scalar curvature 
$s$ and self-dual Weyl curvature $W_+$ of $g$ must satisfy
\begin{equation}
\label{voici}
\frac{1}{4\pi^2} \int_M \left( \frac{s^2}{24} + 2 |W_+|^2 \right)d\mu_g  \geq  \frac{2}{3} [c_1^+]^2 .
\end{equation}
 Moreover, if $c_1^+:= c_1(L)^+$ is non-zero, and if $M$ satisfies both \eqref{htineq} and the strict form of \eqref{revineq}, 
  then   equality cannot hold in \eqref{voici}. 
 \end{cor} 
 \begin{proof} In light of Proposition \ref{best}, we may assume,  without loss of generality,  that  $s$ is negative somewhere. 
 The triangle inequality then tells us that 
  $$\|s\|_{L^2} + \sqrt{6}  \|W_+\|_{L^2} \geq \left[ \int_M (s-\sqrt{6}|W_+|)^2d\mu_g \right]^{1/2},$$
with equality only if  $|W_+|$ is a constant, non-positive   multiple of $s$. 
However,  the Cauchy-Schwarz inequality in $\RR^2$ also implies that 
$$(24+ 3)^{1/2}  \left[ \int_M \left( \frac{s^2}{24} + 2 |W_+|^2 \right)d\mu_g  \right]^{1/2} \geq (2\sqrt{6}, \sqrt{3}) \cdot (\frac{1}{2\sqrt{6}}\| s\|_{L^2} , \sqrt{2} \|W_+\|_{L^2}),$$
with equality iff $\|s\|_{L^2} = 8\sqrt{6} \|W_+\|_{L^2}$. 
 In conjunction,  these two inequalities therefore   imply  that 
 $$
 \int_M \left( \frac{s^2}{24} + 2 |W_+|^2 \right)d\mu_g \geq \frac{1}{27} \int_M (s-\sqrt{6}|W_+|)^2d\mu_g , 
 $$
 with equality iff $s \equiv - 8\sqrt{6} |W_+|$. When   $c_1^+\neq 0$, combining this with \eqref{voila} therefore yields 
 $$
 \int_M \left( \frac{s^2}{24} + 2 |W_+|^2 \right)d\mu_g \geq \frac{8\pi^2}{3} [c_1^+]^2,
 $$
 with equality only if $(M,g)$ is an almost-K\"ahler manifold on  which $s$ is a negative constant, and for which $W_+(\omega , \omega )\equiv s/12$,
 where  $\omega$ is the associated self-dual symplectic form with $|\omega|^2 \equiv2$.
 But since   in this equality case, $\omega$ is harmonic and of constant length,   the Weitzenb\"ock formula \eqref{outil} then tells us that 
 $$|\nabla \omega|^2 = 2W^+(\omega , \omega ) - \frac{2s}{3} =  -\frac{s}{2} ,$$
 where the right-hand side would also be  a positive constant, and so in particular non-zero everywhere. 
 However, Armstrong \cite{arm1}  elegantly  observed that,  on any 
  almost-K\"ahler $4$-manifold, 
  $\nabla \omega$ can be uniquely expressed as the real part of a 
  section of $\Lambda^{1,0}\otimes K$,  
and    therefore, in the compact case, can  be   non-zero  everywhere only  if $M$   
  satisfies the topological constraint that 
 $$
 (5\chi + 6\tau ) (M) =  \langle c_2 (\Lambda^{1,0}\otimes K),[M]\rangle = \mathsf{e} (\Lambda^{1,0}\otimes K)= 0.
 $$
 But since 
 $$
 (5\chi + 6\tau ) (M) = \frac{9}{4}(2\chi  + 3\tau )(M) + \frac{1}{4}(2\chi  - 3\tau )(M),
 $$
  our assumptions  $(2\chi  + 3\tau )(M) \geq  0$ and $(2\chi  - 3\tau )(M)>0$ together therefore  guarantee that $\nabla \omega$  have a zero somewhere, 
which  would therefore yield a contradiction if equality held in \eqref{voici}. 
We must therefore  have
 $$
\frac{1}{4\pi^2} \int_M \left( \frac{s^2}{24} + 2 |W_+|^2 \right)d\mu_g > \frac{2}{3} [c_1^+]^2
 $$
if  $c_1^+\neq 0$, and if $M$ satisfies both \eqref{htineq} and the strict form of \eqref{revineq}. 
  \end{proof}

 Now let $X$ be a minimal compact complex surface of general type, and then blow it up at $k$ distinct points $p_1 , \ldots , p_k$  in order to produce
 a non-minimal complex surface
$$
M\approx { X}\# { k} { \overline{\CP}_2}, 
$$
that, since $\kod (M)=\kod (X)=2$,  will  also be of  general type.  For the     complex structure $J$ on $M$ produced by the blow-up construction, one  has 
$${ c_1} (M , J) = { c_1}({ X}) - \sum_{j={ 1}}^{{ k}} { E}_j ,$$
where the $E_1 , \ldots , E_k$ are Poincar\'e dual to the homology classes of the $k$ embedded copies of $\CP_1$ with which we have replaced the $k$ chosen blown-up points. 
(These copies of $\CP_1$ are often  called $(-1)$-curves   or {\em exceptional divisors}.) 
However, there are various other complex structures $J^\prime$ on the smooth $4$-manifold $M$ that instead have 
 $${ c_1} (M , J^{\prime}) = { c_1}({ X}) + \sum_{j={ 1}}^{{ k}} (\pm) { E}_j$$
for  each and every one of the $2^k$ choices of the independent $\pm$ signs in this expression. 
In fact, one can actually produce complex structures for each such choice by simply pushing forward  our standard $J$ through a self-diffeomorphism
$\Psi : M\to M$ that is equal the identity outside a union of disjoint tubular neighborhoods of the  $E_1 , \ldots , E_k$, but acts on some or all of the $2$-spheres $E_j$ 
by a reflection. To see this, first notice that there is a 
 self-diffeomorphism of $\RR^4= {\CC^2}$ that  
agrees with complex conjugation ${\CC^2}\to {\CC^2}$ near the origin, \label{checkers} 
but which  is the identity outside a ball about origin. Indeed, for a suitable real-valued function $u (\mathsf{r})$, such a map can be constructed as multiplication by the  $\mathsf{r}$-dependent matrix,
$$\left[ \begin{array}{cccc}
1&0&0&0\\ 0&\cos (u (\mathsf{r} ) )& 0& -\sin (u (\mathsf{r}) )\\
0&0&1&0 \\ 0 & \sin (u (\mathsf{r}) )& 0& \cos (u (\mathsf{r}) )\end{array} \right] , $$
 where $\mathsf{r}$ denotes the  Euclidean radius  in $\RR^4$.
For each $\mathsf{r}$, this matrix belongs to $\mathbf{SO}(4)$, and so preserves (and acts diffeomorphically on)  the $3$-sphere of radius $\mathsf{r}$.
But if we now choose  $u(\mathsf{r})$ to be a smooth ``bump''  function which is $\equiv \pi$ for small $\mathsf{r}$, 
but   $\equiv 0$ for large  $\mathsf{r}$,   
\setlength{\unitlength}{1pt}
\begin{center}
\mbox{
\beginpicture
\setplotarea x from 0 to 300, y from 0 to 100
\arrow <2pt> [1,2] from 0 10  to 258 10
\arrow <2pt> [1,2] from 10 0 to 10 100
\put {{$\mathsf{r}$}} [B1] at 270 7
\put {{\large $0$}} [B1] at -5 5
\put {{\Large $\pi$}} [B1] at -5 55
\put {{$u(\mathsf{r})$}} [B1] at  115 70 
\put {{$u$}} [B1] at  10 105 
{\setlinear   
\plot 10 60 111 60 /
\plot 154 10 250 10  /
\plot 5 10  15 10 /
}
{\setquadratic  
\plot  127 35    138 15   154 10 /
\plot      103 60 117 54 127  35 /
}
\endpicture
}
\end{center}
we have then produced a self-diffeomorphism of $\RR^4= \CC^2$ that  looks like complex conjugation $(z_1, z_2 ) \mapsto (\bar{z}_1, \bar{z}_2)$ 
 on a small ball about the origin, but which agrees with the identity  map outside a specified  larger ball. 
But since blowing up $0\in \CC^2$ amounts to replacing the origin with  the $\CP_1$ of complex lines in $\CC^2$ through $0$,
and since complex conjugation $\CC^2 \to \CC^2$ sends complex lines to complex lines, complex conjugation in $\CC^2$ induces  a well-defined anti-holomorphic 
self-diffeomorphism of the $\mathcal{O}(-1)$ line-bundle that  sends  the zero-section to itself by the orientation-reversing  reflection  $$[z_1:z_2]\mapsto [\bar{z}_1: \bar{z}_2].$$
Applying this model near any chosen collection of the $\{ p_1 , \ldots p_k\}$, and then lifting to the blow-up,  then produces a self-diffeomorphism of $M$ that acts on 
$$H^2 (M , \ZZ) =H^2(X, \ZZ) \oplus \ZZ  E_1 \oplus \cdots \oplus  \ZZ  E_k$$ by the identity on $H^2(X, \ZZ)$,  and multiplies each $E_j$ by $\pm  1$, for any desired  choice  of signs. 
Thus, by making a favorable choice of $\pm$ signs, one sees that, for any given metric $g$ on $M$, there is some  spin$^c$ structure $\mathfrak{c}$ of almost-complex type on $M$ for which 
$$[c_1^+]^2 = [c_1 (X)^+]^2 + 2 \left\langle  c_1(X)^+,   \sum_{j=1}^k \pm E_j^+\right\rangle  + \sum_{j=1}^k [E_j^+]^2 \geq   [c_1 (X)^+]^2\geq c_1^2 (X)$$
and for which $n_{\mathfrak{c}} (M) \neq 0$ if $b_+(M) > 1$, or for which  $n_{\mathfrak{c}} (M,\sphericalangle^+) \neq 0$ if  $b_+(M) =1$.
As a consequence, we   immediately obtain the following:

 \begin{thm} \label{proteus} 
  Let $(M,J)$ be a compact complex surface of general  type, and let $(X, \check{J})$ be its minimal model. 
 Then any Riemannian metric $g$ on $M$ satisfies the following  curvature estimates: 
\begin{eqnarray}
\int_M s^2d\mu &\geq& 32\pi^2 c_1^2 (X) \label{minnie}\\
\int_M (s-\sqrt{6}|W_+|)^2d\mu &\geq& 72\pi^2  c_1^2 (X) \label{moocher}
\end{eqnarray}
where $s$ and $W_+$ respectively denote the scalar and Weyl curvatures of $g$. 
 \end{thm} 
 
 Indeed, with extra work \cite{lebeta},  one can show that both \eqref{minnie} and \eqref{moocher} are both strict unless $M=X$ and $g$ is a $\lambda < 0$ K\"ahler-Einstein metric.
 But what we've shown here is already quite  sufficient to imply  the following:
 
  \begin{thm} 
  \label{lever} 
  Let $(M,J)$ be a compact complex surface of general  type, and let $(X, \check{J})$ be its minimal model. 
 Then any Riemannian metric $g$ on $M$ satisfies the   curvature estimate 
 $$
\frac{1}{4\pi^2} \int_M \left( \frac{s^2}{24} + 2 |W_+|^2 \right)d\mu_g \geq  \frac{2}{3} c_1^2(X),
 $$
and  the inequality is strict if $M$ satisfies the Hitchin-Thorpe inequality \eqref{htineq}.  
 \end{thm} 

 \begin{proof}
 We have just seen that there must be a spin$^c$ structure $\mathfrak{c}$ of almost-complex type on $M$ with   $[c_1^+]^2 \geq c_1^2 (X)$ with respect  to $g$, and    for which either 
   $n_{\mathfrak{c}} (M)\neq 0$ if $b_+(M)>1$, or else for which $c_1^+$ is past-pointing and for which $n_{\mathfrak{c}} (M,\sphericalangle^+)\neq 0$ if $b_+(M)=1$. 
   Since  the  strict form of the reverse Hitchin-Thorpe inequality 
  \eqref{revineq} holds for any complex surface $M$ of general type,  e.g.\  as a consequence  of the Miyaoka-Yau inequality \cite[\S VII.4]{bpv}, the result therefore  follows from Corollary \ref{servus}. 
 \end{proof} 
%

This then yields the following\footnote{While weaker   results of this type were first proved in \cite{lno,lebweyl}, several years  elapsed before  Theorem \ref{nein} was    proved in \cite{lric}.
The proof  presented here  is    based on \cite{lebsurv2,lebeta}.} main result:

\begin{thm} \label{nein}
Suppose that  $X$ is a minimal complex surface of general type, and let $M\approx X\# k \cpbar$ be its blow-up 
at $k$ points. If $k\geq \frac{1}{3}c_1^2(X)$, then the underlying smooth $4$-manifold of  $M$ does not admit  Einstein metrics.
\end{thm} 

\begin{proof} Consider the contrapositive.  If  $M=X\# k \cpbar$  admits an Einstein metric $g$,  then Theorem \ref{ht}  guarantees that  $M$ 
satisfies the Hitchin-Thorpe inequality \eqref{htineq}.   Hence Theorem \ref{lever} tells us  that $(M,g)$  satisfies
$$ \frac{1}{4\pi^2} \int_M \left( \frac{s^2}{24} + 2 |W_+|^2 \right)d\mu_g >  \frac{2}{3} c_1^2(X).$$
However,  since $g$ is Einstein, equation \eqref{gb+} then tells us that this inequality can be rewritten as 
$$ (2\chi +3\tau )(M) > \frac{2}{3} c_1^2(X).$$
But since $c_1^2 (M) = (2\chi +3\tau )(M)$, and each blow-up reduces $c_1^2 =2\chi + 3\tau$ by 
precisely $1$,  this inequality can be rewriten as  
$$c_1^2 (X) - k > \frac{2}{3} c_1^2(X).$$
Thus, the existence of an Einstein metric on $M=X\# k \cpbar$  implies that $c_1^2(X)/3 > k$. 
The result therefore    follows by contraposition. 
\end{proof}

It should be  emphasized that a  complex surface of general type 
can only admit a {\em K\"ahler-Einstein} metric if it is minimal, which would correspond   to the  $k=0$ case in our discussion. 
By contrast, since  $(2\chi + 3\tau )(M) = c_1^2 (X) -k$,  the Hitchin-Thorpe inequality only rules out Einstein metrics 
on our blow-up $M=X\# k \cpbar$   for  much larger values of $k$, namely when $k\geq c_1^2 (X)$.  On the other hand, 
the following prototypical example will  make it clear that, in the intermediate range, 
the improved estimate offered by Theorem \ref{nein}  genuinely depends on diffeomorphism type, rather than just  on homeomorphism type.

\begin{xpl}
Let $X$ be the sextic hypersurface $X_6\subset \CP_3$ introduced on page \pageref{hypersurface}, and let $M\approx X\# 8 \cpbar$ be its blow-up at eight points. 
Because $X$ is a minimal surface of general type with $c_1^2 (X) = 24$, and because  we have constructed $M$ by blowing up $X$ at exactly  $c_1^2 (X) /3$ points,
Theorem \ref{nein} therefore guarantees  that the smooth manifold $M$  cannot admit Einstein metrics.

But now, by analogy to the Kummer construction  of $K3$, consider the complex surface $N$ obtained in the following manner: start with  the product of
a compact complex curve of genus $2$ and a hyper-elliptic complex curve of genus $5$, and then mod out by the $\ZZ_2$ action given by
simultaneously applying the hyper-elliptic involution to both factors: 

\begin{center}
\includegraphics[scale=.5]{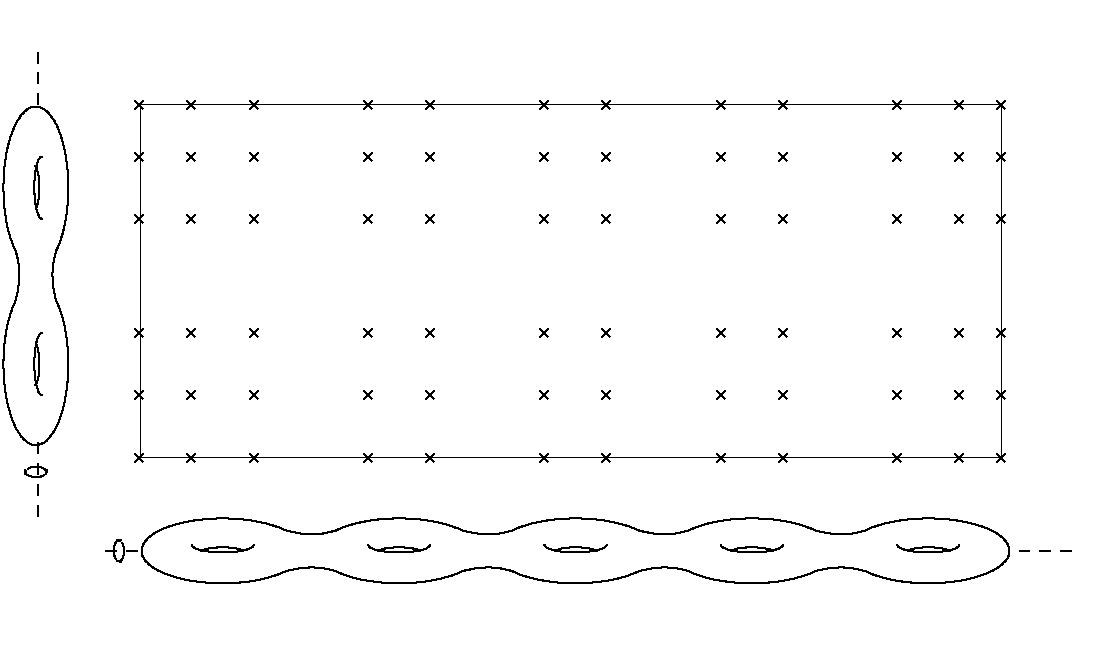}
\end{center} 

\noindent The quotient is then a complex orbifold with $72$ orbifold singularities modeled on $\CC^2/\{ \pm 1\}$.
We will now  replace a neighborhood of each orbifold point with a neighborhood of the zero section in 
$T^*S^2$, and then deform the complex structure in order to yield a  smooth complex surface $N$ with   $c_1 (N) < 0$. We can actually do this explicitly by
first noticing  that our orbifold model is a double branched cover of $\CP_1 \times \CP_1$, branched at $$(\{ 12~pts\} \times \CP_1)\cup (\CP_1 \times \{ 6~pts\}).$$
We can now smooth the singularities of this model  by defining $N$ to be the double branched cover of $\CP_1 \times \CP_1$ with branch locus 
a smooth curve of bidegree $(12, 6)$. The canonical line bundle $K$ of this branched cover $N$ is then just the pull-back of the positive line bundle $\mathcal{O} (4,1)$
from $\CP_1\times \CP_1$, and {\em Nakai's criterion} \cite{bpv}  therefore implies that $K$ is itself positive,  the key point being  that its pairing with any compact 
holomorphic curve in $N$ is positive.
 The Aubin-Yau theorem \cite{aubin,yau}  therefore guarantees
that $N$ admits a $\lambda < 0$  K\"ahler-Einstein metric. In particular, the underlying smooth $4$-manifold of $N$ therefore admits Einstein metrics. 

However, the simply-connected complex surfaces $M$ and $N$ both have $c_1^2 = 16$ and $h^{2,0} = 10$. 
This implies that  they both have Euler characteristic $\chi = 116$ and signature $\tau = -72$.
Since  $\tau$ is therefore not divisible by $16$ for either manifold,  
 Rokhlin's theorem therefore guarantees  that  $M$ and $N$ are both non-spin. Theorem \ref{fdmn}, therefore asserts that  $M$ is  
homeomorphic to $N$. However, these manifolds  are most certainly {\em not  diffeomorphic}, because $N$ admits an Einstein metric, while  $M$ does not!
\end{xpl}

It is now straightforward to  generalize this  example by  blowing up any  hypersurface  $X_\ell\subset \CP_3$ of  degree $\ell \geq 6$ at an appropriate number of points, and then producing a homeomorphic 
K\"ahler-Einstein partner  as  a  branched cover of a Hirzebruch surface. 
This shows, in particular,  that this phenomenon occurs for infinitely many homeotypes of  simply connected $4$-manifolds. For other  simple
  concrete examples of this kind, see  \cite{lric,lebsurv2}. However, we will prove much stronger results by more refined  methods in \S \ref{monopoly}  below. 

\bigskip

For concreteness, we have,  until now,  emphasized constructions rooted in complex geometry. However, 
because the K\"ahler form $\omega = g(J\cdot , \cdot)$ of any K\"ahler metric $g$ is a closed, non-degenerate $2$-form, 
it also defines definies a {\em symplectic structure} on the relevant manifold. Theorem \ref{proton} may therefore
 be viewed as being a special case of the following key result of Taubes \cite{taubes}:
 
 \begin{thm}[Taubes]\label{neutron}
 Let $(M^4,\omega )$ be a compact symplectic $4$-manifold with $b_+(M) > 1$,  and let $\mathfrak{c}$ be the 
 spin$^c$ structure on $M$ determined by any almost-complex structure $J$  compatible 
 with $\omega$. Then $n_\mathfrak{c}(M) \neq 0$. 
 \end{thm} 

Taubes' proof  is analogous  to the one we  sketched in the K\"ahler case, but is considerably  more delicate. 
Namely, he shows that if we choose a perturbation $\eta$ which is a large multiple  of the symplectic form $\omega$, 
plus a carefully chosen bounded term, there is exactly one solution of the Seiberg-Witten equations for the spin$^c$ structure
determined by any $\omega$-compatible almost-complex structure. The same proof also demonstrates the following:

 \begin{thm}[Taubes]
 Let $(M^4,\omega )$ be a compact symplectic $4$-manifold, and let $\mathfrak{c}$ be the 
 spin$^c$ structure on $M$ determined by an almost-complex structure $J$  compatible 
 with $\omega$. If   $b_+(M) = 1$, then $n_\mathfrak{c}(M, \sphericalangle^+) \neq 0$,
where $ \sphericalangle^+$ is the chamber containing large positive multiples of the symplectic form $\omega$. 
 \end{thm}

In a stunning pair  of sequels \cite{taubes2,taubes3}, Taubes then extended these ideas in ways 
that  completely revolutionized $4$-dimensional symplectic geometry, thereby  putting  the earlier   trail-blazing 
work  of Gromov  \cite{gromsym}
and McDuff  \cite{mcrules,dusa}   in a vivid new  context. While McDuff had previously shown that
blowing up and down had natural generalizations to symplectic $4$-manifolds, Taubes now showed that every symplectic
$4$-manifold $M$ with $b_+(M)> 1$ has a unique minimal model $X$ with $c_1^2(X)\geq 0$, and that 
symplectic minimality or non-minimality for such manifolds depends only on diffeomorphism type. 
Taubes' also showed that the canonical class $-c_1$ on any such symplectic manifold is always
represented by one of Gromov's  pseudo-holomorphic curves,  so that every symplectic $4$-manifold with 
$b_+> 1$ in particular  satisfies $Q(c_1 , [\omega ]) < 0$. When $b_+=1$,  Taubes' results still imply \cite{liliu2,mcd-sal-turk}
that every symplectic $4$-manifold has a minimal model, and that non-minimality is still  detected by the existence of embedded 
 pseudo-holomorphic $2$-spheres of self-intersection $-1$, although in this setting,
as in the complex-surface case, minimal models may  no longer be unique, and may not necessarily  have  $c_1^2\geq 0$. 
One beautiful application of these ideas is the following \cite{ohno}:

\begin{thm}[Ohta-Ono]  
\label{optimist} 
Let $(M,\omega)$ be a compact symplectic  $4$-manifold, and suppose that the smooth manifold $M$ also admits 
a Riemannian metric $g$ with scalar curvature $s>0$. Then $(M,\omega)$ is obtained by symplectic blow-up of a standard symplectic structure  on either 
 $\CP_2$ or else a $\CP_1$-bundle over a compact Riemann surface. In particular, $M$ is diffeomorphic to a rational or ruled  complex surface,  with $\kod = -\infty$ and $b_1$ even.
 \end{thm}
 
 The proof proceeds by using Taubes' results to produce a $J$-holomorphic $2$-sphere  $\Sigma \subset M$ with $[\Sigma ]^2 \geq 0$, and then invokes results of
 McDuff \cite{mcrules} to conclude  that $(M,\omega)$ must therefore be rational or ruled.
  
  Taubes' results  also allow us to  immediately  generalize Theorem \ref{nein} to the symplectic setting 
by    simply replacing complex surfaces of general type with the following 
  class of  symplectic $4$-manifolds  \cite{lno,likod}: 
 
 \begin{defn} 
 \label{sympathy} 
 A symplectic $4$-manifold $(M,\omega)$ is said to be of {\sf general type} 
 iff its minimal model $(X, \check{\omega})$ satisfies 
 $c_1^2(X) > 0$  and 
 $Q( c_1(X) , [\check{\omega}]) < 0$.
  \end{defn}

\section{Monopole Classes and Einstein Metrics}
\label{monopoly}

Even though Witten's invariant yields  many wonderful  results, the applications of Seiberg-Witten theory to 
Einstein metrics we have described   have 
generally  only depended on the existence of  a solution of the Seiberg-Witten equation for each 
metric on a given $4$-manifold. As it turns out, this can happen even in contexts where Witten's invariant vanishes
for every spin$^c$ structure. For example, Bauer and Furuta  discovered a generalization of
the $\ZZ_2$-valued Seiberg-Witten invariant $n_\mathfrak{c}$ that implies such an existence 
statement in contexts where the expected dimension \eqref{mdim} of the moduli space is positive \cite{baufu,bauer2}. 
In the simply-connected case, this invariant belongs to  an equivariant stable cohomotopy group
$\pi_{S^1}^{b_+} (\Ind \dir_\theta)$, and encodes homotopy-theoretic  information concerning 
 the monopole map \eqref{raw}  that arises, via a finite-dimensional-approximation scheme,  from a hierarchy of  $S^1$-equivariant   maps between 
one-point compactifications of finite-dimensional vector spaces. While all the main results illustrated  in this section 
will actually stem from  the Bauer-Furuta invariant, our aim here  will be to present a flexible and user-friendly  approach to the subject that 
is equally at home with solutions of the Seiberg-Witten equations that arise from  other known invariants, such as \cite{ozsz,rio},
or that even arise for completely unknown reasons. 
Because of the utility of this property,  the following codification, first proposed  by   Kronheimer \cite{K},
has proved to be extremely convenient:

\begin{defn}
Let $M$ be a smooth compact oriented $4$-manifold
with $b_{+}\geq 2$. An element $\mathbf{a}\in  H^{2}(M,\ZZ )/
\mbox{\rm torsion}$, $\mathbf{a}\neq 0$,  is  called a {\sf monopole
class} of $M$ iff there is some  spin$^{c}$ structure
$\mathfrak{c}$ 
on $M$ with first Chern class 
$$c_{1}(L)\equiv \mathbf{a} ~~~\bmod \mbox{\rm torsion}$$ for which   the  unperturbed  Seiberg-Witten 
equations (\ref{drc}--\ref{sd})
have a solution for every Riemannian  metric $g$ on $M$. 
\end{defn}

Notice that we have limited ourselves here,
for simplicity, 
 to the case of $b_+ > 1$;  
but  for an analogous theory when $b_+=1$,  see \cite{lebeta} for the notion of 
a {\sl retroactive class}.
 In any case, the notion of a monopole class, while something of  a black box, 
does  provide a user-friendly   repackaging  of many of the most important 
 Riemannian implications of Seiberg-Witten theory.

The following finiteness result was carefully proved in  \cite{lebsurv2}; cf.\   \cite{il1,il2,lebeta}. 

\begin{prop} \label{control}
If $M$ is  a smooth compact oriented manifold with $b_+ \geq 2$, consider the diffeomorphism-invariant subset of its cohomology defined  by 
$$\mathfrak{C} = \{ \mbox{monopole classes $\mathbf{a}$ on } M\} \subset H^2(M, \RR).$$
Then  $\mathfrak{C}$
is a {\sf finite} set, and is  invariant under reflection  through the origin
$$\mathbf{a} \longmapsto - \mathbf{a}.$$
\end{prop}
\begin{proof} Since any non-empty open set of a finite-dimensional real vector space spans the space, Proposition \ref{snap}
implies that there is a collection $\{ ([\phi_j], g_j)\}_{j=1}^{b_2(M)}$ of pairs, each consisting   of a deRham classes $[\phi_j]\in H^2(M, \RR)$
and a smooth metric $g_j$, such that  the  harmonic representative $[\phi_j]$ with respect to $g_j$ is self-dual, 
 and such that $\{ [\phi_j]\}_{j=1}^{b_2(M)}$ is a basis for $H^2(M, \RR)$.
We may then normalize our basis so that $Q([\phi_j], [\phi_j]) = 1$ for each $j$. Proposition \ref{best} then guarantees that the  linear functionals
\begin{eqnarray*}
L_j : H^2 (M, \RR) &\longrightarrow& \RR\\
\alpha &\longmapsto& Q ( [\phi_j], \alpha )
\end{eqnarray*}
must  satisfy 
$$
|L_j (\mathbf{a})| \leq \kappa_j
$$
for every monopole class $ \mathbf{a} \in \mathfrak{C}$,
where 
$$\kappa_j:= \frac{1}{4\pi \sqrt{2}}\|s_{g_j}\|_{L^2_{g_j}}.$$
This shows that $\mathfrak{C}$ is a subset of the   parallelepiped 
$$\{ \alpha \in H^2 (M,\RR)~|~ |L_j (\alpha )|\leq \kappa_j, j=1,\ldots , b_2\},$$
which is a compact subset of $H^2(M, \RR )$.
However,  $\mathfrak{C}$ is also, by definition, a subset of  the integer lattice $H^2(M,\ZZ )/\mbox{torsion}$, and 
so  is automatically  a  discrete subset of $H^2(M,\RR)$. It therefore follows that $\mathfrak{C}$ must be finite. 

To prove the invariance of $\mathfrak{C}$ under reflection through the origin, first notice that 
complex conjugation 
$\mathbb{V}_+\to \mathbb{V}_+^*$
acts on Seiberg-Witten solutions by 
$(\Phi, \theta ) \mapsto (\overline{\Phi}, \overline{\theta})$.
The effect of this on $L=\det ( \mathbb{V}_+)$ is then
$$c_1(L) \rightsquigarrow c_1(L^*) = -c_1(L).$$
Thus, if  $c_1(L) \in H^2_{dR}(M,\RR)$ is a monopole class, then 
$-c_1(L)\in H^2_{dR}(M,\RR)$ is a monopole class, too. 
\end{proof}

Notice  that the natural homomorphism $H^2(M,\ZZ )\to H^2(M, \RR)$ induces
an inclusion $H^2(M,\ZZ) /\mbox{torsion}\hookrightarrow H^2(M,\RR)$. In the above proof, and in many applications, 
we  may  thus 
  consider the set of monopole classes to be a subset of the real vector space 
$H^2(M,\RR)$.  From this perspective, we will now     define $\hull (\mathfrak{C})\subset H^2(M,\RR)$ to be the convex hull of the set of monopole classes.
Proposition \ref{control} then  allows us to 
define  a  non-negative numerical invariant of any smooth compact oriented $4$-manifold \cite{lebeta} that provides   a transparent 
distillation 
of many  key arguments in the subject:

\begin{defn}
Let $M$ be a smooth compact oriented $4$-manifold
with $b_{+}\geq 2$. If $\mathfrak{C}\neq  \varnothing$, we now define 
$$\beta^2 (M) = \max \{ ~Q(\alpha  , \alpha ) ~|~ \alpha \in \hull (\mathfrak{C}) \} .$$
Otherwise, if $\mathfrak{C}=  \varnothing$, we set $\beta^2 (M) =0$.
\end{defn}

Here the finiteness of $\mathfrak{C}$ guarantees that $\hull (\mathfrak{C})$  is compact; thus, 
when $\mathfrak{C}\neq \varnothing$,  the   above maximum
is actually achieved. On the other hand,  the refection-invariance  of $\mathfrak{C}$  guarantees  that $0\in \hull (\mathfrak{C})$ whenever  
 $\mathfrak{C}\neq \varnothing$, so we always have  $\beta^2(M) \geq 0$. 
Still,   it is important to remember  that, since   the intersection form $Q$ is typically  indefinite, the function 
$\alpha \mapsto \alpha^2:= Q(\alpha , \alpha)$ is usually  not   convex. Thus,   the maximum of 
$Q(\alpha , \alpha )$ on $\hull (\mathfrak{C})$  can (and, indeed,  often  does)    occur
 away from the vertices  of  the convex hull.

 Propositions \ref{best} and \ref{upscale} then allow one to  prove \cite{lebeta} the following generalization of Theorem \ref{proteus}:

\begin{thm} \label{summa} 
Let $M$ be a compact oriented $4$-manifold with $b_+\geq 2$. 
Then any metric $g$ on $M$ satisfies the curvature estimates 
\begin{eqnarray}
\int_M s^2d\mu &\geq& 32\pi^2 \beta^2(M)\label{mocha}\\
\int_M (s-\sqrt{6}|W_+|)^2d\mu &\geq& 72\pi^2 \beta^2(M)\label{java}
\end{eqnarray}
where $s$ and $W_+$ respectively denote the scalar and Weyl curvatures of $g$. 
Moreover, if $M$ carries a non-zero monopole class, 
equality  occurs in either (\ref{mocha}) or (\ref{java})   
iff $g$ is K\"ahler-Einstein, with negative Einstein constant. 
\end{thm}

\noindent Here, the most delicate part of the proof concerns  the equality case  of \eqref{java}, 
where one uses   detailed curvature properties of   the canonical line bundle on 
 an almost-K\"ahler $4$-manifold. For details, see  \cite[Theorem 4.10]{lebeta}.

Combining \eqref{java} with Corollary \ref{servus} therefore yields the following:

\begin{thm} \label{laude} 
Let $M$ be a compact oriented $4$-manifold with $b_+\geq 2$. 
Then any metric $g$ on $M$ satisfies the curvature estimates 
\begin{equation}
\label{ecco}
\frac{1}{4\pi^2} \int_M \left( \frac{s^2}{24} + 2 |W_+|^2 \right)d\mu_g  \geq  \frac{2}{3} \beta^2 (M) .
\end{equation}
Moreover, if $\beta^2 (M)>0$, 
equality never holds in    \eqref{ecco}. 
\end{thm}

 \noindent Here the last  observation follows from the fact that \eqref{ecco} is implied by  \eqref{java} via the triangle inequality and Cauchy-Schwarz 
 arguments used in the proof of Corollary \ref{servus},  while one nonetheless  has strict inequality in \eqref{ecco}  for the K\"ahler-Einstein metrics that saturate \eqref{java}. 
 
 \bigskip
 
 The finiteness of $\mathfrak{C}$ also makes the following definition \cite{il1} possible:

\begin{defn}\label{band}
	Let $M$ be a smooth compact oriented $4$-manifold 
	with $b_{+}(M)\geq 2$. Let ${\mathfrak C}\subset H^{2}(M, \ZZ )
	/\mbox{\rm torsion}$ 
	be the set of
monopole classes of $M$.  
If ${\mathfrak C}$ contains a non-zero
element (and hence at least two elements), we define
the {\sf bandwidth} of $M$ to be 
$$\BW (M )  ~=  
\max ~ \left\{ n\in \ZZ^{+}~|~ \exists ~\mathbf{a},\mathbf{b} \in {\mathfrak C}, ~\mathbf{a}\neq 
\mathbf{b}, 
~{\rm s.t.}~ 
2n|(\mathbf{a}-\mathbf{b}) ~ \right\} .$$
On the other hand, if ${\mathfrak C}\subset \{ 0\}$,
we define the bandwidth $\BW (M )$
 to be $0$. 
	\end{defn}
	
\noindent Here,  we say that  $2n|(\mathbf{a}-\mathbf{b})$ if  there is some $\mathbf{c}\in H^2(M,\ZZ)/\mbox{torsion}$ such that 
	 $\mathbf{a}-\mathbf{b}=2n \mathbf{c}$. Including a factor of $2$ in this definition is of course purely conventional, but is also  perfectly natural insofar as the difference 
	 between the Chern classes of any two spin$^c$ structures is automatically divisible by $2$. 

The Bauer-Furuta invariant  \cite{bauer2,baufu,il1,il2} provides many remarkable examples of monopole classes that are undetected by the Witten  invariant $n_\mathfrak{c}$. 
In particular, this machinery yields the following useful result:

\begin{prop}\label{vela} 
			Let $X$, $Y$, and $Z$ be compact oriented 		
	$4$-manifolds with  $b_1=0$ and $b_{+}\equiv 3\bmod 4$,
	and suppose that each of them is equipped with an almost-complex structure
	for which the corresponding   spin$^c$ structure 
	satisfies  $n_\mathfrak{c}\neq 0$. 
	Then, with respect to the canonical isomorphism 
			$$
			H^{2}(X\# Y\# Z, \ZZ)  =   H^{2}(X, \ZZ)\oplus H^{2}(Y, \ZZ) 
			\oplus H^{2}(Z, \ZZ) , 
             $$
	the cohomology classes 
	$$
	\pm c_{1}(X)   \pm c_{1}(Y)   \pm c_{1}(Z) 	  \in 
	H^{2}(X, \ZZ)\oplus H^{2}(Y, \ZZ) 
			\oplus H^{2}(Z, \ZZ) 
	$$
	are monopole classes  on the connected sum 
	on $X\# Y\# Z$.
	Here the $\pm$ signs are arbitrary, and  are completely independent
	of one another. 
		\end{prop}
		
		The proof of  involves the fact that Bauer-Furuta invariants of a
		connected sum are smash products of the corresponding invariants of 
		the summands. Similar results hold for   connected sums of up to four
		summands  \cite{bauer2,il2}. However, one does not obtain any such  result for  more  $b_+\equiv 3 \bmod 4$ summands, as the corresponding smash
		products then  turn out to vanish. 
		
		\medskip 
		
	Since the set of monopole metrics must be invariant under all orientation-preserving self-diffeomophisms of $M$, the 
	self-diffeomorphisms constructed on page \pageref{checkers},  transplanted to the symplectic setting, now yield the following:
		
		\begin{cor} Let $X$, $Y$, and $Z$ be compact  simply-connected minimal symplectic 		
	$4$-manifolds with  $b_{+}\equiv 3\bmod 4$. Then, for any $k\geq 0$, the cohomology classes 
	$$
	\pm c_{1}(X)   \pm c_{1}(Y)   \pm c_{1}(Z) 	 \pm E_1 \cdots  \pm E_k  \in 
	H^{2}(X \# Y\# Z\# k \overline{\CP}_2 , \ZZ )
	$$
are monopole classes on the connected sum $X \# Y\# Z\# k \overline{\CP}_2$, where $E_j$ is 
the generator of $H^2(\overline{\CP}_2, \ZZ )$ for the $j^{\rm th}$ copy of $\overline{\CP}_2$. 
\end{cor}

%

\noindent Because this implies that $c_1(X)+c_1(Y)+c_1(Z)\in \hull (\mathfrak{C})$, we therefore immediately obtain the following:

\begin{prop} \label{manamana}
Let $X$, $Y$, and $Z$ be minimal simply-connected  symplectic $4$-manifolds   with $$b_+\equiv 3\bmod 4.$$
Then, for any $k\geq 0$, 
\begin{equation}
\label{monomono}
\beta^2 (X\# Y \# Z \# k \overline{\CP}_2) \geq  c_1^2 (X) + c_1^2 (Y)+ c_1^2 (Z).
\end{equation}
\end{prop} 

\begin{rmk} In the  special case where $X$, $Y$, and $Z$ are actually complex surfaces,  one can moreover show that equality in fact holds  in \eqref{monomono}. This is proved
\cite{il2} by constructing a suitable sequence of Riemannian metrics $g_j$ on the connected sum for which $\int s^2d\mu$ tends to $32\pi^2[c_1^2 (X) + c_1^2 (Y)+ c_1^2 (Z)]$,  and then by appealing to \eqref{java}. However, in what  follows, the inequality \eqref{monomono}, in the form stated  above, will be sufficient  for our purposes. 
\end{rmk} 

\begin{thm}\label{oui}
	If  $X,$ $Y,$ and $Z$  are simply-connected 
	symplectic $4$-manifolds
	with  $b_{+}\equiv 3 \bmod 4$,
	then  the smooth manifold $X\# Y  \# Z \# k\overline{\CP}_{2}$  cannot
	admit an Einstein metric if 
    $$k+8 \geq \frac{c_{1}^{2}(X) + c_{1}^{2}(Y)+c_{1}^{2}(Z)}{3}.$$
	\end{thm}
\begin{proof}
We begin by noticing that 
\begin{equation}
\label{tiptop} 
(2\chi +3\tau) (X\# Y  \# Z \# k\overline{\CP}_{2}) = c_1^2 (X) + c_1^2 (Y)+  c_1^2 (Z) -k - 8,
\end{equation}
since $c_1^2 (X)= (2\chi +3\tau) (X) = 4 +5b_+(X) -b_-(X)$, and similarly  for  $Y$ and $Z$.
By \eqref{monomono} and the Hitchin-Thorpe inequality \eqref{htineq}, it then  follows that there cannot be an Einstein metric on 
 $M= X\# Y  \# Z \# k\overline{\CP}_{2}$ unless  $\beta^2(M) >0$. 
 If $M$  admits an Einstein metric $g$, inequality 
   \eqref{ecco} is therefore strict for $g$, and in conjunction with 
   equation  \eqref{gb+}  this  then    implies  that 
$$(2\chi + 3\tau ) (X\# Y  \# Z \# k\overline{\CP}_{2}) = \frac{1}{4\pi^2} \int_M \left( \frac{s^2}{24} + 2 |W_+|^2 \right)d\mu_g  >   \frac{2}{3} \beta^2 (M).$$
Combining this with \eqref{monomono}  and \eqref{tiptop} now  yields 
$$c_1^2 (X) + c_1^2 (Y)+  c_1^2 (Z) -k - 8 >   \frac{2}{3} [c_1^2 (X) + c_1^2 (Y)+ c_1^2 (Z)].$$
The existence of an Einstein metric on $X\# Y  \# Z \# k\overline{\CP}_{2}$ thus implies that 
$$ \frac{c_1^2 (X) + c_1^2 (Y)+ c_1^2 (Z)}{3} > k+8,$$
so that  our claim therefore now follows by contraposition. 
\end{proof}

%
%

One consequence of this is 
 an  improved   version  of    \cite[Theorem 12]{il1}:

\begin{cor}
	Let $(\ell, m)$ be a pair of natural numbers
	with 
	 $\ell \equiv 1 \bmod 4$  and $\ell \geq 9$.  If 
	$m \geq  \frac{7}{3}\ell +8$, there are infinitely many 
	distinct  smooth structures on $\ell \CP_{2}\# m \overline{\CP}_{2}$
	 for which no  compatible Einstein metric exists. Nonetheless,
there are also infinitely many choices of   	$(\ell , m)$  satisfying these conditions 
for which  $\ell \CP_{2}\# m \overline{\CP}_{2}$  also admits 
a different smooth structure for which Einstein metrics do actually  exist.
	\end{cor}

\begin{proof} For each positive integer $j$,  Akhmedov and Park \cite{akpa} have constructed  a 
compact simply-connected non-spin symplectic manifold
$X_j$ with $b_+(X_j)=4j-1$ and $b_-(X_j)=4j$. Among these, we will assign a rather privileged role to 
 $X_1$, which therefore  has $b_+(X_1)=3$ and $b_-(X_1) =4$. 
 Next, let $Y_q$ denote the sequence of ``exotic'' $K3$ surfaces
described on page \pageref{xoK3}. 
For any positive integers $j,k,q$, we  then  define
$$M_{j,k,q} =X_1\#  X_ j\#  Y_{q}\# k \overline{\CP}_{2}.$$ 
Since $c_1^2 (X_1) = 15$, 
  $c_1^2 (X_j)= 16j-1$,   and  $c_1^2 (Y_q)=0$, we therefore have 
$$c_1^2(X_1)+c_1^2(X_j)+c_1^2(Y_q) = 16j+14,$$
and since $X_1$, $X_j$ and $Y_q$ are simply-connected symplectic $4$-manifolds with $b_+\equiv 3 \bmod 4$, 
Theorem \ref{oui} therefore implies   that $M_{j,k,q}$ does not admit Einstein metrics if
\begin{equation}
\label{condiment}
k +8 \geq \frac{16j+14}{3}.
\end{equation}
However, if  we now set  
$$\ell =b_+(M_{j,k,q}) = 4j+5, \quad m =b_-(M_{j,k,q}) =4j+23+k,$$
then \eqref{condiment} becomes
$$m-\ell -10\geq \frac{4(\ell -5)+14}{3} = \frac{4}{3}\ell  -2$$
or in other words 
\begin{equation}
\label{whatwhat}
m\geq  \frac{7}{3} \ell  +8.
\end{equation}
Every positive-integer pair $(\ell, m)$ satisfying  $\ell \equiv 1 \bmod 4$, $\ell \geq 9$,  and \eqref{whatwhat} 
now arises from a positive-integer pair $(j,k)$ satisfying \eqref{condiment}, and so, 
 for every such choice of $(j,k)$  and any positive integer $q$, the  smooth  simply-connected non-spin $4$-manifold $M_{j,k,q}$ 
  does not 
admit compatible Einstein metrics, but 
is also  homeomorphic to $\ell \CP_2\# m \cpbar$  by Theorem \ref{fdmn}.
However, we also have 
$\BW (M_{j,k,q})\geq 2q$,  and hence 
$$\lim_{q\to \infty} \BW (M_{j,k,q}) =\infty$$
for each   fixed ($j,k)$. It therefore follows  that, as $q$ varies, 
the $M_{j,k,q}$ must realize infinitely many different smooth structures without Einstein metrics on the corresponding topological $4$-manifold 
$\ell \CP_2\# m \cpbar$.

However, for any  integer $p\geq 6$ with 
		$p\equiv 2 \bmod 4$,  we may in particular consider  	
	$\ell =p^{2}-3p+3$ and $m=3p^{2}-3p+1$. We then have
	$\ell \equiv 1\bmod 4$, $\ell  > 9$,  and $m >  \frac{7}{3}\ell +8$, 
 so the preceding  argument proves 
	  that  there are 
	infinitely many smooth structures on $\ell \CP_{2}\# m \overline{\CP}_{2}$
     for which no  Einstein metric exists. However, this
	 homeotype is also realized  by the ramified double 
	  cover of $\CP_{2}$,  branched at  a smooth  complex-algebraic curve $B$ of degree
	 $2p$.  \begin{center}
\mbox{
\beginpicture
\setplotarea x from 30 to 400, y from -30 to 120
{\setlinear 
\plot 65 100 135 75 /
\plot 60 75  130 100 /
\plot 60 0 97 15 /
\plot 97 15 135 0 /
\plot 130 100 127 79 /
\plot 65 100 63 77 /
}
{\setquadratic
\plot 133 75 130 37 133 0 /
\plot 56 75 52 37 56 0 /
\plot 94 89 90 49 93 15 /
}
\putrectangle corners at 220 110 and 380 -10
\circulararc 180 degrees from 332 68  center at 325 75
\circulararc -180 degrees from 332 32  center at 325 25
\circulararc -180 degrees from 268 68  center at 275 75
\circulararc 180 degrees from 268 32  center at 275 25
\put {${\mathbb C \mathbb P}_2$} [B1] at 365 -5
\put {$B$} [B1] at 258 10
\put {$B$} [B1] at 89 3
\arrow <6pt> [1,2] from 145 50  to 200 50
\put {$N$} [B1] at 39 45
{\setquadratic
\plot 282 82  300 72  318 82 /
\plot 282 18  300 28  318 18 /
\plot  332 32  322 50   332  68 /
\plot  268 32  278 50   268  68 /
}
\endpicture
}
\end{center}
For  example,  we
 could   take $B$ to be curve in $\CP_2=\{ [x:y:z]\}$ defined by 
	 $x^{2p}+y^{2p}+z^{2p}=0$, and then define  $N$ to be the set of elements $w$
	 of  the $\mathcal{O}(p)$ line-bundle over $\CP_2$  that satisfy 
	 $$w^2 =  x^{2p}+y^{2p}+z^{2p}.$$
	The canonical line bundle $K$ of $N$ is then isomorphic to the pull-back of
	$\mathcal{O}(p-3)$ from $\CP_2$. Nakai's criterion \cite{bpv} therefore  implies that 
	 $N$	 has ample canonical line bundle
	 \cite{bpv}; and  the Aubin/Yau theorem
	 \cite{aubin,yau} therefore  yields  a $\lambda < 0$ K\"ahler-Einstein metric
	  $N$, compatible with the given complex structure. For these choices of $(m,n)$, there are is thus also a  smooth structures on the 
	  topological  manifold $\ell \CP_{2}\# m \overline{\CP}_{2}$  for which Einstein metrics do indeed exist. 	  
	 \end{proof}

In the non-spin case, this shows that it is  in fact quite  common for a simply-connected topological $4$-manifold to admit one smooth structure for which there exists a smooth Einstein metric, while also 
 admitting infinitely many others for which no Einstein metric can exist. In the spin case, such examples are harder to produce. Still,   we saw  in \S \ref{sweet} that $K3$
does  fulfill these desiderata, and we will now show that  there is at least one other simply-connected  spin homeotype  with these properties:

\begin{cor} 
The topological spin manifold
\begin{equation}
\label{delish} 
3(K3) \# 4 (S^2 \times S^2)
\end{equation}
admits  one smooth structure for which there is a compatible  Einstein metric, but  also admits   infinitely-many other  smooth structures
for which no smooth   Einstein metric can exist. 
\end{cor} 
\begin{proof}
Park and Szabo \cite[Proposition 3.1]{parkszabo}, using the  knot-surgery technique of  Fintushel and Stern \cite{fistknot}, 
constructed a   symplectic $4$-manifold  $X$ with $c_1^2(X)=16$ that is  homeomorphic to $K3\# 4(S^2 \times S^2)$.
For each positive integer $q$, we now set 
$M_q=X\# Y_0 \#Y_q$, where $Y_0$ denotes  $K3$ with its standard differentiable structure, and  $\{ Y_q\}$ is the sequence of  ``exotic''
complex surfaces homeomorphic to $K3$   described 
on page  \pageref{xoK3}. Since $b_+(X)=7\equiv 3 \bmod 4$, while  $b_+(Y_0)=b_+(Y_q)=3$, 
and since Theorem  \ref{manamana} tells us that 
$$\beta^2(M_q) \geq c_1^2(X) + c_1^2 (Y_0) + c_1^2(Y_q) = 16 $$
for each $q$,
 the $k=0$ case of 
Theorem \ref{oui}
tells us that none of the smooth spin manifolds $M_q= X\# Y_0\# Y_q$ can admit an Einstein metric. 
On the other hand, $\BW (M_q) \geq 2q$, so 
$$\lim_{q\to \infty} \BW (M_q) =\infty.$$
Since Theorem \ref{control} implies that the supremum of the bandwidth over any finite collection of smooth structures is finite, 
it 
 therefore follows 
 that there must be infinitely many distinct diffeotypes among the $M_q$.
Because each $M_q$ is  homeomorphic  to \eqref{delish}, this construction thus shows that the topological manifold 
\eqref{delish} consequently admits infinitely many distinct smooth structures for which no Einstein metric exists.

But now let  $N$ be the smooth simply-connected complex surface obtained by taking a ramified double cover of $\CP_2$, branched at a smooth curve $B$ of degree $10$.
The canonical line bundle $K$ of $N$ is then isomorphic to  the pull-back to $N$ of the $\mathcal{O}(2)$ line-bundle on   $\CP_2$. 
This then implies, by Nakai's criterion \cite{bpv}, that $K$ is ample, so that 
 the Aubin-Yau theorem \cite{aubin,yau} theorem implies that $N$ carries a $\lambda < 0$  K\"ahler-Einstein metric that is compatible with its given complex structure. 
But the fact that $K$ is  the pull-back of  $\mathcal{O}(2)$  also implies that $N$ is spin, that $c_1^2 (N) =8$, and that $h^{2,0}(N) =6$. Thus, $N$ is a simply-connected 
spin manifold with 
$\chi (N)= 76$ and $\tau (N) = -48$, and Theorem \ref{fdmn} therefore guarantees that $N$ is 
 homeomorphic to \eqref{delish}. 
 \end{proof}

 Of course, these corollaries only provide a few illustrative consequences of  Theorem \ref{laude}. In particular, there are many examples of monopole
 classes that are not detected by the Bauer-Furuta invariant, but are still distinguished  by the better-known integer-valued Seiberg-Witten invariant 
 \cite{fistknot,morgan,taubes} that 
  refines the 
 $\ZZ_2$-valued invariant $n_\mathfrak{c}$ that we have chosen to  highlight in this article.
  For more  
 applications  of this circle of   ideas to the theory of Einstein $4$-manifolds, see for example \cite{bayish,brownguard,biss,rio,donsurv,ishno,ishpin,irs,ishisasa,outhouse,jp2,raresioana}.

\pagebreak 

\section{When the Einstein Constant is Positive} 

Until this point, our approach  has emphasized results concerning  $4$-manifolds that do not admit metrics of 
positive scalar curvature, and therefore, in particular, cannot admit $\lambda > 0$ Einstein metrics. However,
the results we have already explored in this article  also play a key role in proving certain  classification
results regarding the   $\lambda > 0$ case, such as the following \cite{chenlebweb}:

\begin{thm} \label{smack} 
Suppose that $M$ is  a smooth compact oriented $4$-manifold that  
carries an integrable, orientation-compatible  complex structure $J$.
Then $M$  also admits an  ({\em a priori} unrelated) Einstein metric $g$ with ${\lambda} > 0$ if and only if 
$$
M\stackrel{{\mbox{\tiny diff}}}{{\approx}} \begin{cases}\CP_2\# {k} \overline{\CP}_2, &0 \leq {k} \leq 8,\\ ~~~~\text{or}& \\
{S^2 \times S^2}.& 
\end{cases}
$$
\end{thm}

There are actually two  wildly different aspects to this 
 statement. On one hand, it excludes every diffeomorphism type other than the ten possibilities listed. 
 On the other hand, it claims that each of these ten listed $4$-manifolds admits a  $\lambda > 0$ Einstein metric, and also  admits
  an integrable  complex structure that is  {\em a priori} allowed to have nothing whatsoever to do with the metric. 
 These two different aspects  of the result are in fact proved by using two entirely unrelated sets of tools. 
 
   In one direction, to narrow down the possibilities to only  the ten listed  listed candidates, one first invokes
  Seiberg-Witten theory to show that $M$ must be a complex surface of Kodaira-dimension $-\infty$. One then invokes the 
  Hitchin-Thorpe inequality to show that one  must then  have $c_1^2(M) > 0$. The Kodaira-Enriques classification  \cite{bpv,GH}  then 
 implies that $M$ is diffeomorphic to one of the ten listed manifolds. 
   
  In the other direction, each of the ten allowed diffeotypes
  can be realized by a {\em del Pezzo surface} \cite{dolgachev}, meaning   a compact complex surface for which $c_1$ is a K\"ahler class. 
  On eight of the ten del Pezzo  diffeotypes,    Tian-Yau \cite{ty} and  Siu \cite{s}  showed  that there are $\lambda > 0$ K\"ahler-Einstein metrics that are compatible with 
  some del Pezzo complex structure. On the other hand, the two remaining del Pezzo surfaces, namely  $\CP_2\# \cpbar$ and  $\CP_2\# 2\cpbar$, cannot admit
  K\"ahler-Einstein metrics (by a theorem of Matsushima \cite{mats}, which asserts   that the Lie algebra of holomorphic vector fields on a compact 
   $\lambda >0$ K\"ahler-Einstein manifold must be reductive). Nonetheless, 
   Page \cite{page} had previously written down an explicit $\lambda > 0$ Einstein metric on what was later observed   to be $\CP_2\# \cpbar$, and 
  Derdzi\'nski \cite{derd}  then showed that Page's metric is actually a conformal rescaling of a K\"ahler metric. Indeed, Derdzi\'nski's argument actually implies that 
  the Page metric $g$ takes  the form $g=s^{-2}h$, where $h$ belongs to the family of extremal K\"ahler metrics  on $\CP_2\# \cpbar$ discovered by Calabi \cite{calabix},
  and   where the scalar curvature  $s>0$ of $h$  enjoys the special  property that its gradient $\nabla^h s$ is the real part of a holomorphic 
  vector field. This provided one of the  key clues along the road  that eventually led 
    to a  proof \cite{chenlebweb}  of the existence of a conformally-K\"ahler $\lambda > 0$ Einstein metric on $\CP_2\# 2\cpbar$. A second
    proof of the existence of this Einstein metric on  $\CP_2\# 2\cpbar$ was later found \cite{lebhem10} that, in conjunction with \cite{lebuniq},
    also shows that this  Einstein metric  is uniquely  characterized by the fact that it minimizes $\int |W_+|^2  d\mu$ among all conformally-K\"ahler metrics on 
    the fixed complex surface. Indeed, this variational characterization  also holds for all the other Einstein metrics used in proof of Theorem of \ref{smack}, including the 
    Page metric.

By combining Theorem \ref{optimist} with the above arguments, we immediately also   obtain  \cite{chenlebweb} 
 a  symplectic analog  of  Theorem \ref{smack}: 

\begin{thm} \label{wack}
Suppose that $M$ is  a smooth compact oriented $4$-manifold that
carries an orientation-compatible  symplectic form $\omega$.
Then $M$  also admits an ({\em a priori} unrelated) Einstein metric $g$ with ${\lambda} > 0$ if and only if 
$$
M\stackrel{{\mbox{\tiny diff}}}{{\approx}} \begin{cases}\CP_2\# {k} \overline{\CP}_2, &0 \leq {k} \leq 8,\\ ~~~~\text{or}& \\
{S^2 \times S^2}.& 
\end{cases}
$$
\end{thm}

Of course, while these results do not suppose that the Einstein metrics under discussion are in any way related to the complex structure $J$ or symplectic forms
$\omega$ that figure in the statements, 
 the metrics used in the existence part of the proof are in fact all conformally K\"ahler, and therefore  intimately related 
to both a particular  complex structure $J$ and a specific symplectic form $\omega$. We will say more about this curious situation in a moment, 
when we discuss the moduli spaces of Einstein metrics on these manifolds.

Of course, assuming that the Einstein constant is positive is essential for these results, as our discussions in \S\S \ref{sweet}--\ref{monopoly}  
make it clear  that the situation becomes  vastly more complicated if the 
Einstein constant is allowed to be negative. However, if we just weaken the $\lambda > 0$ hypothesis to read  $\lambda \geq 0$, Theorems \ref{smack} and \ref{wack}
still have nice generalizations \cite{lebcam}:

 \begin{thm} \label{mortar}
Suppose that $M$ is  a smooth compact oriented $4$-manifold that 
admits an integrable, orientation-compatible  {complex structure} $J$.
Then $M$ also admits an  Einstein metric  $g$ with $\lambda \geq 0$ if
and only if 
$$
M\stackrel{{\mbox{\tiny diff}}}{{\approx}} 
\begin{cases}\CP_2\# {k} \overline{\CP}_2,  \qquad 0 \leq {k} \leq 8, &\\ 
{S^2 \times S^2},& \\
{K3}, &\\ 
{K3}/\ZZ_2, &\\
{T^4},&\\
 {T^4}/\ZZ_2,   {T^4}/\ZZ_3,    {T^4}/ \ZZ_4 ,  {T^4}/ \ZZ_6 , &\\
 {T^4}/(\ZZ_2 \oplus \ZZ_2), 
 {T^4}/(\ZZ_3\oplus \ZZ_3), \mbox{or  }
 {T^4}/ (\ZZ_2 \oplus \ZZ_4).
 \end{cases}
 $$
\end{thm}

\begin{thm}\label{pestle} 
Suppose that $M$ is  a smooth compact oriented $4$-manifold that 
admits an orientation-compatible symplectic form  $\omega$.
Then $M$ also admits an  Einstein metric  $g$ with $\lambda \geq 0$ if
and only if 
$$
M\stackrel{{\mbox{\tiny diff}}}{{\approx}} 
\begin{cases}\CP_2\# {k} \overline{\CP}_2, \qquad 0 \leq {k} \leq 8, &\\ 
{S^2 \times S^2},&\\
{K3}, &\\ 
{K3}/\ZZ_2, &\\
{T^4},&\\
 {T^4}/\ZZ_2,   {T^4}/\ZZ_3,    {T^4}/ \ZZ_4 ,  {T^4}/ \ZZ_6 , &\\
 {T^4}/(\ZZ_2 \oplus \ZZ_2), 
 {T^4}/(\ZZ_3\oplus \ZZ_3), \mbox{or  }
 {T^4}/ (\ZZ_2 \oplus \ZZ_4).
 \end{cases}
 $$
\end{thm}

The ten manifolds added to the list in passing from Theorem \ref{smack} and \ref{wack} to Theorems \ref{mortar} and \ref{pestle} are
 exactly the complex surfaces that admit Ricci-flat K\"ahler metrics, and, on these manifolds,  Theorem \ref{ht} implies  that, conversely, every Einstein metric
 is Ricci-flat K\"ahler.  This allows one to show that  the  Einstein moduli spaces $\mathscr{E}(M)$ are connected for each of these ten added  manifolds.
By  contrast, however, we currently have only  a partial understanding of the moduli spaces $\mathscr{E}(M)$  for the ten  manifolds on the  shorter original list  of Theorems  \ref{smack} and \ref{wack}.
The known Einstein metrics on these ten del Pezzo manifolds are all conformally K\"ahler, and this condition can be shown \cite{lebcake,lebdet,pengwu} to be open and closed in this context. 
In conjunction with \cite{lebuniq,sunspot}, this allows one to show that the  known Einstein metrics sweep out exactly one connected component of the moduli space $\mathscr{E}(M)$.
Whether there are other connected components, however, currently remains a mystery.

%
%
%
  
\section{When the Fundamental Group is Large}

Most of the results discussed in this article have been based on the fact that 
the Seiberg-Witten equations  yield estimates for  the scalar and self-dual-Weyl curvatures in dimension four, in a manner that 
strongly depends on 
the smooth structure  the  $4$-manifold in question. While we have often illustrated 
these results  by applying them to simply-connected
manifolds, they are of course also  applicable to  manifolds with enormous 
fundamental groups, as we saw for example  in Corollary \ref{alibaba} and Theorem \ref{cohyp}. 
On the other hand, these results are intrinsically four-dimensional in nature, and 
tell us nothing at all about Riemannian geometry in 
 any   other dimension.

By contrast, the proof of Theorem \ref{rehyp} hinges on a remarkable 
set of results that are of great interest in any dimension $n$. 
 Let $(M,g)$ be a compact Riemannian manifold, and let 
$(\widetilde{M},\widetilde{g})$ be its universal cover. Let $x\in \widetilde{M}$,
and let $B_{\zap r}(x)\subset \widetilde{M}$ be the Riemannian distance ball of radius ${\zap r}$ centered at $x$,
which by definition consists of points that can be joined to $x$ by some smooth path of length 
$< {\zap r}$, and   let $\vol(B_{\zap r}(x))$
denote the volume of this distance-ball with respect to $\widetilde{g}$. Then 
the {\em volume entropy} of $(M,g)$ is defined to be 
$${\zap h} (M,g)=\lim_{{\zap r} \to \infty}
 \frac{\log \vol (B_{\zap r}(x))}{{\zap r}}.$$
This is independent of the base-point $x$, and  can be 
non-zero only if the fundamental group of $M$ is infinite. 
For example,  the volume entropy of
a flat $n$-torus is zero, while the volume entropy of   a compact 
  $n$-manifold of constant sectional curvature $K<0$
equals  $(n-1)|K|^{1/2}$. 

This last example makes it clear  that the volume entropy is certainly  not 
scale-invariant.
However, we can remedy this by instead considering the {\em normalized
volume entropy}
$$\widehat{\zap h} (M,g) = [\vol (M,g)]^{1/n}{\zap h} (M,g).$$
In  terms of this invariant, Besson, Courtois, and Gallot \cite{bcg} were able to prove the following truly remarkable result: 
 
\begin{thm}[Besson-Courtois-Gallot] \label{ntp}
Let $M$ be any compact quotient of a real, complex, quaternionic,
or octonionic hyperbolic space, and let $g_0$  be the tautological  metric
on $M$. Then any  metric $g$ on $M$ satisfies
$$\widehat{\zap h} (M,g) \geq \widehat{\zap h} (M,g_0),$$
with equality iff $g$ is locally symmetric.
\end{thm}

The proof  is essentially a calibrated geometry argument. One  embeds the universal cover  $\widetilde{M}$ in
$L^2$ of its sphere at infinity, and   compares the
entropy with the integral of a certain closed $n$-form 
over a fundamental domain; for details see \cite{bcg}.
One  dramatic consequence   is the following:

\begin{cor} \label{ctp}
Let $(M,g_0)$ be a compact quotient of hyperbolic
space ${\cal H}^n$,
and let $g$ be any metric on $M$ with Ricci curvature
$\geq -(n-1)$. Then $\vol (M,g) \geq \vol (M,g_0)$,
with equality iff $g$ has  sectional curvature $\equiv -1$. 
\end{cor}

Notice that, like Theorem \ref{ntp},  this result holds in any dimension $n$, but that
$(M,g_0)$ is now specifically  assumed to be a {\em real-hyperbolic} manifold. Thus, since we've 
assumed that the Ricci curvature of $(\widetilde{M},\widetilde{g})$ is greater than or equal to 
than the Ricci curvature  of real-hyperbolic space  ${\cal H}^n$,  the 
{\em Bishop-Gromov inequality} \cite{bishop} says that the volume of a Riemannian ball of radius 
${\zap r}$ in  $(\widetilde{M},\widetilde{g})$ must be less than or equal to  the volume of a corresponding 
 ball in ${\cal H}^n$. We therefore 
have 
$${\zap h}  (M,g) \leq (n-1) = {\zap h}  (M,g_0) .$$
But since we also have $[\vol(M,g) ]^{1/n}{\zap h}  (M,g)
\geq [\vol(M,g_0) ]^{1/n}{\zap h}  (M,g_0)$ by Theorem \ref{ntp}, 
it therefore  follows that $\vol(M,g) \geq \vol(M,g_0)$,
and that equality moreover only occurs if $g$ has constant curvature
$-1$. 

\bigskip

However, despite the fact that Theorem \ref{ntp} holds in all dimensions, it is only in dimension $n=4$ that it implies  a rigidity result,  like Theorem \ref{rehyp}, for Einstein metrics.
Here, the  crucial  extra ingredient is  the remarkably simple form taken by the $4$-dimensional Gauss-Bonnet formula \eqref{cgb}, which for a compact $4$-dimensional  Einstein manifold 
$(M,g)$ becomes
\begin{equation}
\label{goober}
\int_M \left( \frac{s^2}{24}+ |W_+|^2 +|W_-|^2 \right) d\mu_g = 8\pi^2 \chi (M).
\end{equation}
Now suppose that $(M,g_0)$ is  a compact real-hyperbolic $4$-manifold, with sectional curvature $-1$, and thus with Einstein constant $-3$ and scalar curvature $-12$. 
If $g$ is another Einstein metric on $M$, it  cannot have Einstein constant $\lambda \geq 0$, because $\pi_1(M)$ is an infinite group that (by Preissmann's theorem) does not contain $\ZZ\oplus \ZZ$;  thus, the  $\lambda >0$ case is  ruled out by Myers' Theorem \cite{myers}, while  the $\lambda =0$ case is ruled out (cf. page \pageref{ricflat}) by the Cheeger-Gromoll Splitting Theorem \cite{cg} and 
 Bieberbach's Theorem \cite{bie}.
Thus, our rival Einstein metric would have $\lambda < 0$,
and, by replacing  it with a suitable  constant multiple, we can  therefore  arrange for  $g$, like $g_0$, to have Ricci curvature $-3$ and scalar curvature $-12$. 
However, because  $g_0$ is  locally conformally flat, and so has vanishing Weyl curvature, equation  \eqref{goober} therefore tells us that 
$$ \vol (M,g) \leq \vol (M, g_0)= \frac{4\pi^2}{3} \chi (M).$$
But since $g$ has Ricci curvature $= -3$, Corollary \ref{ctp} also tells us that $\vol(M,g) \geq \vol(M,g_0)$. Hence $(M, g)$ and $(M,g_0)$ 
must actually have equal volume, 
and $g$ must therefore in fact have constant sectional curvature $-1$.  Mostow rigidity thus
guarantees that $g$ must actually just be $g_0$, moved by a self-diffeomorphism of $M$ that is homotopic to the identity, and Theorem \ref{rehyp}
therefore follows.

In fact,  the  calibrated-geometry argument used to prove Theorem \ref{ntp}  also yields a  new, self-contained  proof of Mostow rigidity, as a special case   of    the following stronger result  \cite{bcg}:

\begin{thm}[Besson-Courtois-Gallot] 
\label{dig}
Let $(M,g_0)$ be as in Theorem \ref{ntp}, and let $N$ be a compact oriented 
manifold of the same dimension. If there is a
smooth map $f: N\to M$ of degree $\ell >0$, then 
any metric $g$ on $N$ satisfies
$$\hat{\zap h} (N,g) \geq \ell \hat{\zap h} (M,g_0).$$
Moreover, if $\dim M > 2$,   equality can only be 
achieved if $f$ is homotopic to a covering map which 
is a local isometry. 
\end{thm}

A beautiful application of this last result to Einstein $4$-manifolds was then  discovered by Sambusetti \cite{samba}: 

\begin{thm}[Sambusetti]
Any integer pair $(\chi,\tau)$ with $\tau\equiv \chi \bmod 2$
can be realized as the Euler characteristic   and signature
of a compact  smooth  4-manifold  (with infinite fundamental 
group) 
that  is not homeomorphic to any Einstein manifold.
\end{thm}
 
Of course, this statement  is optimal, 
 since the Euler characteristic and signature  of a compact oriented $4$-manifold always  have the same parity. 
  Sambusetti's examples are of the form 
$$N= j M \# k (S^1\times S^3) \# \ell \CP_2 \# m \overline{\CP}_2,$$
where $M$ is some fixed compact oriented real-hyperbolic $4$-manifold. Because there is  an obvious   degree-$j$ map $f: N\to M$,
one can therefore use   Theorem \ref{dig} to show that  $N$ cannot admit Einstein metrics if 
$\chi (N) < j \chi (M),$
and this   happens iff 
$2 +\ell + m < 2j + 2k$.
But letting $j$, $k$, $\ell$, and 
$m$ vary, one can  check that these examples then realize 
 all  possible  $(\chi, \tau)$ with  $\chi \equiv \tau \bmod 2$.
 Moreover, since this  argument only depends on the 
existence of a suitable homotopy class of maps 
$f: N\to M$, changing the differentiable structure
of these examples would not change anything;  thus,  these 
examples 
$N$ are not even {\em homeomorphic} to Einstein manifolds. 
For further details, see \cite{samba}. 
 
 \medskip
 
In this overview, we have frequently  compared and contrasted the \linebreak $4$-manifolds where Einstein metrics are  known to exist
with those where we can prove that their existence  is  obstructed. However, the existence results we have appealed to 
in this regard have generally come from the  theory of K\"ahler-Einstein metrics,
 eventually   supplemented by the modest generalization  offered by conformally-K\"ahler, Einstein metrics.  However, this has simply  been a matter of
 practicality,  because
 K\"ahler geometry has  actually  provided  all the currently-available  examples of  simply-connected Einstein \linebreak $4$-manifolds, aside from  the standard round $4$-sphere. 
 Since  this situation will probably just turn out to  be 
a mere   historical accident,    the  ambitious reader 
 should by all means feel encouraged to spend a small amount of time trying   to dream up new ways  of constructing    examples!
 In any case, it now seems appropriate to  improve the balance of this article  by concluding  with a brief survey  of 
 some of the remarkable recent  constructions that have been found for 
  non-locally-symmetric Einstein metrics of generic holonomy on compact $4$-manifolds with 
infinite  fundamental group. 

The first major conceptual breakthrough in this direction  was made by  Anderson \cite{andcusp}, who proposed a
generalization of Thurston's {\em Dehn filling} construction to higher dimensions. Thurston's original construction began with a finite-volume complete hyperbolic $3$-manifold,
with only  $T^2\times \RR^+$ cusps at infinity, and then showed that any such space can be approximated by a sequence of compact hyperbolic $3$-manifolds
obtained by  {\em Dehn filling}, where each cusp  is first truncated, and then replaced with   a hyperbolic   $S^1 \times D^2$. Anderson's idea  was to 
start with a finite-volume hyperbolic $n$-manifold with only $T^{n-1} \times \RR^+$ cusps at infinity, then  truncate each cusp,   and finally  glue in 
 a  Riemannian-signature ``toroidal-horizon'' AdS Schwarzschild \cite{brillo}  black-hole metric  on $T^{n-2} \times D^2$ along each resulting $T^{n-1} $ boundary component.
  By truncating each cusp further and further way from some base-point, and considering different possible identifications  of the  $T^{n-1}$ boundaries 
and choices of cut-offs, one obtains sequences of Dehn-filled  metrics  which fail to solve the Einstein equation \eqref{eineq} by an ever-smaller  error. 
Anderson's program  
was  then to perturb these approximate solutions into genuine Einstein metrics via an inverse-function-theorem argument. 
But although  Anderson's paper outlines  most of the key geometric and topological  ideas needed to make this program work, 
many of the   analytical issues involved  did not seem to be   convincingly addressed by his paper, especially   in the  case where there are multiple  cusps. 
However,  Bamler \cite{bamcusp} was later able to convincingly complete Anderson's program  by supplying a more precise   approach to the analytic
details,
thus  clearly establishing  the existence of these Dehn-filled Einstein manifolds.  Bamler's approach  encodes each  topological choice for  a Dehn filling by  specifying  the homotopy classes of the meridian 
circles in the finite-volume hyperbolic manifolds with cusps,  and each such  homotopy class can  then be uniquely represented by a closed hyperbolic geodesic.
Bamler's main result  is then  that there is a 
universal constant, depending only on the dimension $n$ and volume of the  given non-compact hyperbolic manifold, such that if the lengths of these closed geodesics 
 are all larger than 
this universal constant, then there is an Einstein metric on the Dehn-filled manifold. The proof hinges on a 
subtle,  non-standard inverse-function theorem in appropriately-weighted H\"older spaces.
This provides a huge menagerie of new examples, not only in dimension $4$, but also in higher dimensions. 

For $n\geq 4$, these Dehn-filled  Einstein $n$-manifolds do not  have constant sectional curvature; indeed, Preissmann's theorem 
\cite{preissmann} guarantees that these spaces 
 cannot even admit metrics of strictly negative sectional curvature, because their fundamental groups contain $\ZZ\oplus \ZZ$  subgroups  arising from the tori glued in  at each cusp. 
Nonetheless, each of these solutions does contain  large regions where the sectional curvature is nearly equal to $-1$. 
And, like hyperbolic manifolds, each of these Einstein manifolds is actually  aspherical, in the sense that each one is a $K(\pi ,1)$.

But these Anderson/Bamler examples also  immediately   raise another question of fundamental interest: 
are there any compact Einstein $4$-manifolds with  strictly negative sectional curvature, other than the standard 
  locally symmetric
examples? An affirmative answer to 
this question was     provided  by Fine and Premoselli  \cite{fineprem}, who  constructed   Einstein metrics of negative sectional curvature on certain  branched covers of 
  hyperbolic $4$-manifolds. Specifically,  one starts with a compact hyperbolic $4$-manifold $X$ which contains a a compact totally-geodesic $2$-surface $\Sigma\subset M$
whose homology class  $[\Sigma ]\in H_2(X,\ZZ)$ is zero; for any chosen positive integer $\ell$ one can then take an $\ell$-fold branched cover $M\to X$ with branch-locus $\Sigma$. 
Here a  key piece of motivation was provided by a result   of 
 Gromov and Thurston \cite{groth}, who had previously constructed metrics of pinched negative curvature on such branched covers. Because many  of these Gromov-Thurston metrics have 
 nearly-constant negative sectional curvature, and therefore nearly constant Ricci curvature, Fine and Premoselli  may  have initially hoped to perturb these Gromov-Thurston metrics
  into genuine Einstein metrics by means 
 an inverse-function theorem.  However, they  eventually realized that the problem becomes much more tractable if one instead  first replaces the metric near  the branch-locus $\Sigma$ 
 with an exact solution of the Einstein equation \eqref{eineq}, namely  the   Riemannian  metric  on $\Sigma \times \RR^2$ that is the ``Wick-rotated'' version 
 of a ``higher-genus horizon'' AdS Schwarzschild metric  \cite{brillo}.
Because the sectional curvature  of this exact solution  rapidly tends to  $-1$ at infinity, one can construct metrics on $M$ that only fail to be Einstein on a $\Sigma \times S^1 \times (a,b)$ transition region 
where   the failure  to solve \eqref{eineq}  can  be arranged to be small  in appropriate weighted  H\"older norms if 
the injectivity radius of $X$ and normal injectivity radius of $\Sigma\subset X$ are both large enough, and the volume of $X$ is huge compared to the area of $\Sigma$.
To make the whole scheme work,  Fine and Premoselli  consider a sequence of hyperbolic manifolds $X_j$, each equipped  with a     totally-geodesic surface $\Sigma_j \subset X_j$, such that all three 
 of these geometric quantities tend to infinity at specified rates. They then prove that, for $j\gg 0$,
 the degree-$\ell$ branched cover $M_j$ of   $X_j\supset \Sigma$ carries
an Einstein metric of  negative sectional curvature that is very close to  their ``approximately Einstein'' metric. In particular, because the AdS Schwarzschild metric
that's been  glued in near the branch locus does not have constant curvature, the constructed Einstein metrics on $M_j$ cannot have constant curvature either. In particular, 
Theorem \ref{rehyp} therefore guarantees that $M_j$ is not even homotopy equivalent to a compact hyperbolic $4$-manifold.

While Fine and Premoselli \cite{fineprem} were able to construct many non-locally-symmetric compact Einstein $4$-manifolds of negative sectional curvature
in this framework, they were not able to construct such Einstein manifolds  in higher dimensions. However,  Hamenst\"adt and J\"ackel \cite{hamjack} were subsequently 
able to extend the result to all dimensions $\geq 4$ by introducing two types of technical improvements. Most importantly, they were able, by passing to covers,   to construct sequences of 
compact hyperbolic  $n$-manifolds with codimension-$2$ totally-geodesic submanifolds such that the normal  injectivity radius divided by the diameter of the submanifold
goes to infinity. This allows one to glue in the higher-dimensional AdS Schwarzschild metric with negligible damage  in the transition annulus. In addition, they introduced 
several new analytic ideas, such as the systematic use of a spectral gap argument,  that  considerably simplify the inverse-function-theorem proof. Not only does this extend
the Fine-Premoselli  result to higher dimensions, but it even leads to a considerably larger class of examples in dimension four.

Finally,  these developments  also  stimulated a remarkable breakthrough   \cite{gueham} in K\"ahler geometry. While the Aubin-Yau theorem \cite{aubin,yau}    might seem to be the end
of the story when it comes to classifying $\lambda < 0$  K\"ahler-Einstein manifolds, the fact that the K\"ahler-Einstein metrics produced by the theorem are by no means  
explicit  makes it extremely difficult to predict  much at all about the sectional curvature or other geometric properties of these K\"ahler-Einstein metrics.
For instance, in complex dimension $2$, the compact complex surfaces of general type with $\chi=3\tau$ (which    sit on  the Miyaoka-Yau line, and so  saturate the  Miyaoka-Yau inequality $\chi \geq 3\tau$) 
are exactly those of  the form $\CC\mathcal{H}_2/\Gamma$ for some discrete group $\Gamma$  of isometries of the complex-hyperbolic plane. For these special complex surfaces,
 the  K\"ahler-Einstein metrics promised by the Aubin-Yau theorem  thus have negative sectional curvature --- although these examples are also  of course  locally symmetric!  Are 
there other  complex surfaces with  K\"ahler-Einstein metrics of negative sectional curvature? The answer  is {\em yes}.  Indeed, a recent construction  of Guenancia and Hamenst\"adt \cite{gueham}  yields 
a vast hierarchy of  negatively-curved K\"ahler-Einstein metrics on cyclic branched covers of complex-hyperbolic surfaces. 
On the other hand,  does a complex surface with $c_1<0$ have a negative-sectional-curvature K\"ahler-Einstein metric iff
it is close to the Miyaoka-Yau line? Here 
the answer is {\em no}. In fact, many of the Guenancia-Hamenst\"adt examples are  not especially  close to the Miyaoka-Yau line.
And, more shockingly, Roulleau and Urz\'ua \cite{ruler}  have  proved that there are {\em simply-connected} minimal complex surfaces that are arbitrarily close to the 
Miyaoka-Yau line --- and these  of course  cannot carry negative-sectional-curvature metrics!  Thus, while our current  knowledge  sometimes  now allows us to  see many Einstein manifolds
through a glass, darkly, we may still be in for a shock when we  finally   meet some of them  face to face!

\vfill

\noindent {\sf Acknowledgments.} The author would like to thank  the Simons Foundation 
and the Fondation Sciences Math\'ematiques de Paris for their financial support, in the form of a Simons Fellowship
and a Chaire d'Excellence, during his  sabbatical stay in Paris, where a substantial part   of this article was written. 
Part of this work was also supported  by NSF grant DMS-2203572.
\pagebreak 
%

\end{document}